\documentclass{amsart}
\usepackage{amsmath,amssymb}
\usepackage{mathptmx}      % use Times fonts if available on your TeX system
\usepackage{mathrsfs}
\usepackage{latexsym}
\usepackage{amssymb}
\usepackage{bm}
\usepackage[usenames,dvipsnames]{xcolor}
\usepackage{graphicx,pifont}
\usepackage{multicol,multirow}
\usepackage{ulem,cancel}
\usepackage{tikz}
\usetikzlibrary[patterns]
\usepackage{comment}
\usepackage{hyperref}
\usepackage{array}
\usepackage{caption}
\usepackage{subcaption}
\usepackage{txfonts}
\usepackage{mathtools}
\usepackage{enumerate}

\newtheorem{theorem}{Theorem}[section]

\theoremstyle{definition}
\newtheorem{definition}[theorem]{Definition}

\theoremstyle{remark}

\numberwithin{equation}{section}
\numberwithin{figure}{section}

\newcommand{\abs}[1]{\left|#1\right|}

\newcommand{\eqdef}{\stackrel{\mathrm{def}}{=\joinrel=}}

\newcommand{\po}[1]{\frac{\partial}{\partial{#1}}}

\newcommand{\pp}[2]{\frac{\partial {#1}}{\partial {#2}}}

\newcommand{\cfl}{{\textrm{cfl}}}
\newcommand{\bc}{{\textrm{bc}}}
\newcommand{\upw}{{\textrm{upw}}}
\newcommand{\muscl}{{\textrm{muscl}}}

\newcommand{\minmod}{{\textrm{minmod}}}
\newcommand{\term}{{\textrm{term}}}

\begin{document}

\title[FVM of Infiltration in Tumor Growth]{On Finite Volume Discretization of Infiltration Dynamics in Tumor Growth Models}

\author[X.~Zeng]{Xianyi Zeng}
\address{Department of Mathematical Sciences,\\
         Computational Science Program, University of Texas at El Paso, El Paso, TX 79902, United States.\\
         Tel.: +1-915-747-6759}
\email[Corresponding author, X.~Zeng]{xzeng@utep.edu}

\author[M.~Saleh]{Mashriq Ahmed Saleh}
\address{Computational Science Program, University of Texas at El Paso, El Paso, TX 79902, United States.}
\email[M.~Saleh]{msaleh@miners.utep.edu}

%\author[B.~Niu]{Ben Niu}
%\address{Department of Mathematical Science, New Mexico State University, Las Cruces, NM 88003, United States.\\
%         Department of Mathematics, Harbin Institute of Technology at Weihai, Weihai, Shandong 264209, P. R. China.}
%\email[B.~Niu]{niu@nmsu.edu}

\author[J.~Tian]{Jianjun Paul Tian}
\address{Department of Mathematical Science, New Mexico State University, Las Cruces, NM 88003, United States.}
\email[J.~Tian]{jtian@nmsu.edu}

\date{\today}



\begin{abstract}
%My abstract: 150 -- 250 words.
We address numerical challenges in solving hyperbolic free boundary problems described by spherically symmetric conservation laws that arise in the modeling of tumor growth due to immune cell infiltrations.
In this work, we normalize the radial coordinate to transform the free boundary problem to a fixed boundary one, and utilize finite volume methods to discretize the resulting equations.
We show that the conventional finite volume methods fail to preserve constant solutions and the incompressibility condition, and they typically lead to inaccurate, if not wrong, solutions even for very simple tests.
These issues are addressed in a new finite volume framework with segregated flux computations that satisfy sufficient conditions for ensuring the so-called totality conservation law and the geometric conservation law.
Classical first-order and second-order finite volume methods are enhanced in this framework.
Their performance is assessed by various benchmark tests to show that the enhanced methods are able to preserve the incompressibility constraint and produce much more accurate results than the conventional ones.
\end{abstract}

\maketitle

\section{Introduction}
\label{sec:intro}
Modeling the tumor growth due to immune cell infiltration using partial differential equations (PDE) has been an active research area in recent years.
One of the earliest papers addressing this phenomenon from a mathematical point of view is by Evelyn F. Keller and Lee A. Segel~\cite{EKeller:1971a}, who model the cell movements by Brownian motion and conclude that they generally move towards a region with high chemoattractant concentration.
The Patlak-Keller-Segel (PKS) chemotaxis system, which describes the interaction between the cell and the chemoattractant, is then studied both theoretically and numerically by various authors~\cite{MBurger:2006a, VCalvez:2006a, FFilbet:2006a, BPerthame:2009a, IKim:2012a, EEspejo:2013a, AChertock:2018a}.
Existing literature focuses on solving the PKS system on a fixed domain; hence they are suitable for describing the cell movements inside the tumor but not for modeling how the tumor grows.
Recently, B. Niu and the authors of the current paper propose a free boundary model that extends the PKS system to describe the growth of tumor due to immune cell infiltration~\cite{BNiu:2018a}.
In this model, the immune cells are attracted by the chemoattractant that usually has higher concentration inside the tumor and enter the tumor boundary; the mean cell movement velocity is derived by assuming the cells are incompressible, i.e., the total cell number per unit volume is assumed to be constant.
The incompressibility is a crucial assumption -- because the cells have fixed volume, when immune cells enter through the tumor boundary they need to compete with native ones for space and eventually promote tumor growth.

It should be noted that treating biological systems as free boundary problems is by no means new.
In the literature, there are numerous successful studies addressing the existence and uniqueness of solutions to such PDE systems~\cite{XChen:2003a, SCui:2003a} as well as conducting well-behaved numerical simulations~\cite{WHao:2016a}.
We would like to emphasize, however, that these studies rely on the fact that the same velocity field is used for the advection of all cell species; hence a characteristic method (in the analytical approach) or a Lagrangian strategy (in the computational approach) can be applied.
This is not the case with infiltration dynamics, as by nature the invading species and the native ones are carried by different velocities.
It is worth mentioning that in a recent work by A. Friedman et al.~\cite{AFriedman:2010a}, the authors prove the global existence and uniqueness of solutions to a free boundary problem that contains infiltrating species; however, the governing equations therein are of parabolic type, which is very different from what we're considering here -- because there is no diffusion term for cell species, the model considered in this paper does not contain regularization and shocks do occur in the solution process.

Studies on free boundary problems of hyperbolic type that involves infiltration dynamics, to our best knowledge, remain scarce in the literature. 
In this work, we attempt to close this gap by proposing a new finite volume framework for the discretization of a general class of equations.
Particularly, we consider the migration of two categories of cell species -- the cell species belonging to the first category move inside the tumor and will never cross the boundary, whereas the second category involve all the infiltrating species; the motion of both types are governed by hyperbolic equations.
The methodology is described in a very general setting, in the sense that it is not restricted to any particular cell proliferation, apoptosis, and interaction models.
To this end, the method we propose is suitable for the investigation of any similar systems, such as the plaque development and the wound healing processes~\cite{AFriedman:2010a, WHao:2014a, AFriedman:2015a}.
However, for the ease of statement we set our context in tumor growth modeling and use the term ``cell'' to refer to any entities that play a part in the incompressibility constraint, see Section~\ref{sec:rev}.

\smallskip

In previous work~\cite{BNiu:2018a}, the spherically symmetric free-boundary problem is considered and solved numerically by first mapping the physical coordinates onto a fixed logical one and then discretization using the conventional finite volume methods, see also Section~\ref{sec:rev} for a brief review of this model.
Although the shocks are captured nicely, clear violation of the incompressibility assumption is observed, especially near the tumor center.
A major cause is that incompressibility is not enforced directly by the model; instead, it is assumed in the derivation of the velocity equation.
In addition, geometrical source terms appear when we change from the physical coordinate to the logical one; and existing finite volume methods cannot balance them well, even when the solutions are constants.

To resolve these issues, we investigate a simplified model that easily extends to the full tumor growth model of~\cite{BNiu:2018a}.
The totality conservation law (TCL) and the geometric conservation law (GCL) are defined and justified as the criterion for any numerical method to maintain constant solutions and satisfy the incompressibility condition.
The new finite volume methods are developed in three stages.
First, we design a general finite volume framework for solving the model system, and extend the TCL and GCL to the discrete level, called DTCL and DGCL, respectively, where the first letter ``D'' stands for ``discrete''.
Next, we propose several consistency properties, so that for any numerical flux that satisfies these properties, the resulting method will satisfy both DTCL and DGCL.
Finally, the classical first-order upwind method and the second-order MUSCL flux~\cite{BvanLeer:1979a} are enhanced according to these conditions. 

The remainder of the paper is organized as follows.
In Section~\ref{sec:rev} we briefly review the original tumor growth model as well as the incorporation of the incompressibility assumption.
Then, a simplified model that captures the most important features is described in Section~\ref{sec:model}.
The main results and the proposed methods are derived in Section~\ref{sec:fvm}, where we propose the DTCL and DGCL conditions and prove sufficient conditions for the numerical method to satisfy these conditions.
Extensive numerical tests are provided in Section~\ref{sec:num} to assess the performance of the enhanced methods, which is compared to the existing finite volume methods.
Finally, Section~\ref{sec:concl} concludes this paper.

\section{A Review of the Tumor Growth Model and Its Finite Volume Discretization}
\label{sec:rev}
In the tumor growth model proposed earlier~\cite{BNiu:2018a}, we consider the movement of glioma (or cancer) cells, necrotic cells, and immune cells, whose number densities are denoted by $G(r,t)$, $N(r,t)$, and $M(r,t)$, respectively.
Here $r$ is the distance from a point inside the (spherically symmetric) tumor to the center and $t$ is the time ordinate.
The cells are supposed to be incompressible, in the sense that one expects:
\begin{equation}\label{eq:rev_inc}
  G(r,t) + N(r,t) + M(r,t) = \theta\;,
\end{equation}
for some constant $\theta$ that designates the total number of cells per unit volume.

The velocities of the cell movements are determined by two aspects.
First, because of the incompressibility assumption each cell takes a fixed volume; hence when the cells are squeezed they tend to move to the nearby region and eventually cause the tumor to grow or shrink.
The velocity due to the cell-volume-preserving mechanism is denoted by $V(r,t)$, and it is the solely velocity that is responsible for the movement of glioma cells and necrotic cells.
Second, in addition to $V$, the immune cells are also guided by the chemoattractant concentration, as discussed by Evelyn F. Keller and Lee A. Segal~\cite{EKeller:1971a} in the early 1970s.
The corresponding velocity is denoted by $U(r,t)$, and it is positive related to the gradient $\partial A(r,t)/\partial r$, where $A(r,t)$ is the chemoattractant concentration.

In the spherical coordinates, the equations that govern the cell movements are thusly given by:
\begin{subequations}\label{eq:rev_eqn}
\begin{align}
\label{eq:rev_eqn_g}
&\pp{G}{t}+\frac{1}{r^2}\po{r}\left[r^2GV\right] = f(r,t,G,N,M)\;, \\
\label{eq:rev_eqn_n}
&\pp{N}{t}+\frac{1}{r^2}\po{r}\left[r^2NV\right] = g(r,t,G,N,M)\;, \\
\label{eq:rev_eqn_m}
&\pp{M}{t}+\frac{1}{r^2}\po{r}\left[r^2M(V+U)\right] = h(r,t,G,N,M)\;,\\
\label{eq:rev_eqn_u}
&U = \alpha\pp{A}{r}\;.
\end{align}
Here $\alpha>0$ is a positive parameter that is supposed to be constant; and $f,g,h$ are source terms that describe the production and diminishing of the cells.
In this paper, we follow the convention that a single upper case letter denotes a dependent variable to be solved, and a single lower case letter designates an independent variable or a prescribed function.
Our numerical method will not depend on the particular forms of the source functions; from a modeling point of view, however, examples of these functions are given below.
Let $\lambda$ and $\mu$ be the self-production and transformation rates of the cancer cells, we have:
\begin{equation}\label{eq:rev_eqn_sg}
f(G,N,M) = \lambda G - \mu G\;;
\end{equation}
here $\mu$ is the rate at which the cancer cells convert to necrotic cells, which are removed from the tumor by the rate $\delta$, hence one can model:
\begin{equation}\label{eq:rev_eqn_sn}
g(G,N,M) = \mu G - \delta N\;;
\end{equation}
and finally if the only way that the immune cells are gone is through their own death, which happens at the rate $\rho$, then the source term $h$ can be modeled as:
\begin{equation}\label{eq:rev_eqn_sm}
h(G,N,M) = -\rho M\;.
\end{equation}
For more details about the rationale behind these source functions, the readers are referred to~\cite{BNiu:2018a} and the references therein.

The equations~(\ref{eq:rev_eqn_g})--(\ref{eq:rev_eqn_u}) are valid for all $(r,t)\;:\ 0\le r\le R(t)$, where $R(t)>0$ is the radius of the tumor at time t, whose growth is governed by:
\begin{equation}\label{eq:rev_eqn_r}
R'(t) = V(R(t),t)\;.
\end{equation}

The equation for the velocity field $V(r,t)$ is derived by summing up (\ref{eq:rev_eqn_g})--(\ref{eq:rev_eqn_m}) and invoking the incompressibility assumption~(\ref{eq:rev_inc}):
\begin{equation}\label{eq:rev_eqn_v}
\frac{1}{r^2}\po{r}\left[r^2\theta V+r^2UM\right] = f+g+h\;.
\end{equation}

To complete the system, the chemoattractant $A$ is generally secreted by the glioma cells and subject to the diffusion rate $\nu$ and diminishing rate $\gamma$:
\begin{equation}\label{eq:rev_eqn_a}
\pp{A}{t}=\nu\frac{1}{r^2}\po{r}\left[r^2\pp{A}{r}\right]+\frac{\chi mG}{\beta+G}-\gamma A\;,\quad
0\le r< +\infty\;.
\end{equation}
Note that this equation is valid on the entire domain since the chemoattractant exists in the entire body, which is supposed to be much larger than the tumor.
The indicator function $\chi$ in the second term of the right hand side equals $1$ when $0\le r\le R(t)$ and equals $0$ otherwisely.
\end{subequations}

Finally, the governing equation~(\ref{eq:rev_eqn}) is complemented by appropriate initial conditions for $G$, $N$, $M$, and $A$ such that (\ref{eq:rev_inc}) is satisfied, and the following boundary conditions:
\begin{subequations}\label{eq:rev_bc}
\begin{align}
\label{eq:rev_bc_cell}
&\pp{G(0,t)}{r} = \pp{N(0,t)}{r} = 0\;, \\
\label{eq:rev_bc_imm}
&\pp{M(0,t)}{r} = 0\;,\quad
M(R(t),t) = M_{\bc}(t)\ \textrm{ if }\ U(R(t),t)<0\;, \\
\label{eq:rev_bc_chem}
&\pp{A(0,t)}{r} = 0\;,\quad\lim_{r\to+\infty}A(r,t) = 0\;, \\
\label{eq:rev_bc_vel}
&V(0,t) = 0\;.
\end{align}
\end{subequations}
Here the second part of (\ref{eq:rev_bc_imm}) is known as the incoming boundary condition and $M_{\bc}$ is the prescribed embient number density of immune cells.

\subsection{Conservation form in normalized coordinates}
\label{sec:rev_nrm}
To avoid the difficulty of dealing with a time-varying domain, we cast the equations to the normalized coordinates $(\eta,\tau)=(R(t)/t,t)$ and rescale the equations to obtain a conservation system:
\begin{subequations}\label{eq:rev_nrm_eqn}
\begin{align}
\label{eq:rev_nrm_eqn_g}
&\pp{(\eta^2R^2G)}{\tau} + \po{\eta}\left[\left(\frac{V}{R}-\frac{\eta R'}{R}\right)\eta^2R^2G\right] = \eta^2R^2f(G,N,M)-\eta^2R'RG\;,\\
\label{eq:rev_nrm_eqn_n}
&\pp{(\eta^2R^2N)}{\tau} + \po{\eta}\left[\left(\frac{V}{R}-\frac{\eta R'}{R}\right)\eta^2R^2N\right] = \eta^2R^2g(G,N,M)-\eta^2R'RN\;,\\
\label{eq:rev_nrm_eqn_m}
&\pp{(\eta^2R^2M)}{\tau} + \po{\eta}\left[\left(\frac{V}{R}-\frac{\eta R'}{R}+\frac{U}{R}\right)\eta^2R^2N\right] = \eta^2R^2h(G,N,M)-\eta^2R'RM\;,\\
\label{eq:rev_nrm_eqn_u}
&U = \frac{\alpha}{R}\pp{A}{\eta}\;,\\
\label{eq:rev_nrm_eqn_v}
&\frac{1}{\eta^2R}\po{\eta}\left[\eta^2(\theta V+UM)\right] = f+g+h\;,
\end{align}
for all $0\le \eta\le 1$ and $\tau\ge0$; and 
\begin{equation}
\label{eq:rev_nrm_eqn_a}
\pp{(\eta^2R^2A)}{\tau}+\po{\eta}\left[\left(-\frac{\eta R'}{R}\right)\eta^2R^2A\right]=\nu\po{\eta}\left(\eta^2\pp{A}{\eta}\right)+\frac{\chi m\eta^2R^2G}{\beta+G}-\gamma\eta^2R^2A-\eta^2R'RA\;,
\end{equation}
on the domain $0\le\eta<+\infty,\;\tau\ge0$, and $\chi$ is the indicator function that equals $1$ when $0\le\eta\le1$ and $0$ otherwise.
Finally, the radius is evolved as:
\begin{equation}
\label{eq:rev_nrm_eqn_r}
R'(\tau) = V(1,\tau)\;.
\end{equation}
\end{subequations}
Note that the $-\eta^2R'R$ terms are new, and they appear because of our change of coordinates.
The boundary conditions are given by:
\begin{subequations}\label{eq:rev_nrm_bc}
\begin{align}
\label{eq:rev_nrm_bc_cell}
&\pp{G(0,\tau)}{\eta} = \pp{N(0,\tau)}{\eta} = 0\;, \\
\label{eq:rev_nrm_bc_imm}
&\pp{M(0,\tau)}{\eta} = 0\;,\quad
M(1,\tau) = M_{\bc}(\tau)\ \textrm{ if }\ U(1,\tau)<0\;, \\
\label{eq:rev_nrm_bc_chem}
&\pp{A(0,\tau)}{\eta} = 0\;,\quad\lim_{\eta\to+\infty}A(\eta,\tau) = 0\;, \\
\label{eq:rev_nrm_bc_vel}
&V(0,\tau) = 0\;.
\end{align}
\end{subequations}

\subsection{Finite volume discretization}
\label{sec:rev_fvm}
In previous work, the conservative equations (\ref{eq:rev_nrm_eqn_g})--(\ref{eq:rev_nrm_eqn_m}) and (\ref{eq:rev_nrm_eqn_a}) are discretized by the standard finite volume methods, see for example~\cite{BvanLeer:1979a, RLeVeque:2002a}. 
We briefly review the first-order upwind method here as well as introduce some notations that will be used throughout the paper.

The logical domain $\eta\in[0,\;1]$ is divided into $N_\eta$ uniform intervals\footnote{To avoid confusion, we reserve the word ``cell'' exclusively for denoting the cell species; whereas the commonly used ``cell'' in finite volume discretization is referred to as ``interval'' throughout the paper.}, each of which has length $\Delta\eta=1/N_{\eta}$; and we denote the interval faces by $\eta_j=j\Delta\eta$ and interval centers by $\eta_{j-1/2}=(j-1/2)\Delta\eta$.
For easy reading, we use the integer subscripts to denote nodal variables, whereas the half-integer subscripts to denote the variables that are associated with intervals, such as the interval-averages.

In particular, because the cell numbers are conserved quantities, in the general finite volume discretization these variables are defined for each interval, and they're denoted by $G_{j-1/2}$, $N_{j-1/2}$, and $M_{j-1/2}$, where $1\le j\le N_{\eta}$.
Considering in addition the forward-Euler time integrator and designating the discrete solutions at time step $\tau^n$ by the superscript $n$, the general finite volume discretization reads:
\begin{subequations}\label{eq:rev_fvm_fe}
\begin{align}
\label{eq;rev_fvm_fe_g}
&\frac{(R^{n+1})^2G^{n+1}_{j-1/2}-(R^n)^2G^n_{j-1/2}}{\Delta\tau} + \frac{F^{G,n}_j-F^{G,n}_{j-1}}{\eta_{j-1/2}^2\Delta\eta} = (R^n)^2f_{j-1/2}^n-R'^nR^nG_{j-1/2}^n\;,\\
\label{eq;rev_fvm_fe_n}
&\frac{(R^{n+1})^2N^{n+1}_{j-1/2}-(R^n)^2N^n_{j-1/2}}{\Delta\tau} + \frac{F^{N,n}_j-F^{N,n}_{j-1}}{\eta_{j-1/2}^2\Delta\eta} = (R^n)^2g_{j-1/2}^n-R'^nR^nN_{j-1/2}^n\;,\\
\label{eq;rev_fvm_fe_m}
&\frac{(R^{n+1})^2M^{n+1}_{j-1/2}-(R^n)^2M^n_{j-1/2}}{\Delta\tau} + \frac{F^{M,n}_j-F^{M,n}_{j-1}}{\eta_{j-1/2}^2\Delta\eta} = (R^n)^2h_{j-1/2}^n-R'^nR^nM_{j-1/2}^n\;,
\end{align}
where $f_{j-1/2}^n=f(G_{j-1/2}^n,N_{j-1/2}^n,M_{j-1/2}^n)$ and $h_{j-1/2}^n$ and $g_{j-1/2}^n$ are similarly computed; the radius related quantities are:
\begin{equation}\label{eq:rev_fvm_fe_r}
R'^n = V_{N_\eta}^n\;,\quad
R^{n+1} = R^n + \Delta\tau V_{N_\eta}^n\;.
\end{equation}
We define the velocity at the nodes, and $V_{N_\eta}^n$ is the numerical approaximation to $V(1,\tau^n)$, see also the discussion below (\ref{eq:rev_fvm_upw}).
\end{subequations}

The numerical flux $F_j^{X,n}$, where $X$ stands for $G$, $N$, or $M$, is an approximation to the corresponding flux for $X$ at $\eta_j$.
If we apply the existing finite volume methods to compute these numerical fluxes, for example, by using the first-order upwind flux, we have:
\begin{subequations}\label{eq:rev_fvm_flux}
\begin{align}
\label{eq:rev_fvm_flux_g}
F_j^{G,n} &= \mathcal{F}^{\upw}\left(\frac{V_j^n}{R^n}-\frac{\eta_jR'^n}{R^n};\;\eta_{j-1/2}^2(R^n)^2G_{j-1/2}^n,\;\eta_{j+1/2}^2(R^n)^2G_{j+1/2}^n\right)\;, \\
\label{eq:rev_fvm_flux_n}
F_j^{N,n} &= \mathcal{F}^{\upw}\left(\frac{V_j^n}{R^n}-\frac{\eta_jR'^n}{R^n};\;\eta_{j-1/2}^2(R^n)^2N_{j-1/2}^n,\;\eta_{j+1/2}^2(R^n)^2N_{j+1/2}^n\right)\;, \\
\label{eq:rev_fvm_flux_m}
F_j^{M,n} &= \mathcal{F}^{\upw}\left(\frac{V_j^n}{R^n}-\frac{\eta_jR'^n}{R^n}+\frac{U_j^n}{R^n};\;\eta_{j-1/2}^2(R^n)^2M_{j-1/2}^n,\;\eta_{j+1/2}^2(R^n)^2M_{j+1/2}^n\right)\;.
\end{align}
\end{subequations}
Here the upwind flux is defined as:
\begin{equation}\label{eq:rev_fvm_upw}
\mathcal{F}^{\upw}(W;\;X_l,\;X_r) = \left\{\begin{array}{lcl}
WX_l\;, & & \textrm{ if } W\ge0 \\ \vspace*{-.12in} \\
WX_r\;, & & \textrm{ if } W<0
\end{array}\right.\;,
\end{equation}
where the subscripts $l$ and $r$ mean ``left'' and ``right'', respectively, and $W$ is the local advection velocity at the interval face between the two interval values $X_l$ and $X_r$.

In (\ref{eq:rev_fvm_flux}), the velocity variables $V_j^n$ and $U_j^n$, where $0\le j\le N_\eta$, are collocated at the interval face $\eta_j$.
Here the velocity $V$ is computed using the integral form of (\ref{eq:rev_nrm_eqn_v}):
\begin{subequations}\label{eq:rev_fvm_v}
\begin{align}
\label{eq:rev_fvm_v_0}
&V_0^n = 0\;,\\
\label{eq:rev_fvm_v_p}
&V_j^n = \frac{1}{\theta\eta_j^2}\sum_{k=1}^j\eta_{k-1/2}^2R^n(f_{k-1/2}^n+g_{k-1/2}^n+h_{k-1/2}^n) - \frac{1}{\theta}U_j^nM_j^n\;,\quad 1\le j\le N_{\eta}\;,
\end{align}
where $M_j$ is computed as the mean of surrounding interval-averaged values: $M_j=(M_{j-1/2}^n+M_{j+1/2}^n)/2$, except for the last node, in which case $M_{N_{\eta}}=M_{N_{\eta}-1/2}$.
\end{subequations}
As for the velocity $U$, we have $U_0^n=0$ and $U_j^n=\alpha(A_{j+1/2}^n-A_{j-1/2}^n)/(\Delta\eta R^n)$, where $A_{j-1/2}^n$ is the averaged chemoattractant concentration on $[\eta_{j-1},\;\eta_j]$.
The chemoattractant concentration is computed by approximating the convective term (\ref{eq:rev_nrm_eqn_a}) by straightforward finite volume discretization and the diffusion term by central difference approximation.
Because the only role of $A$ is to compute the nodal velocities $U_j^n$, the method we will propose later is independent of how $A$ is computed, as long as the nodal $U_j^n$ is computable; more details are provided in the next section.

% One can check (though tediously) that this model does not allow a space homogeneous solution if gamma is not zero
\subsection{A simple case study}
\label{sec:rev_case}
Whether (\ref{eq:rev_inc}) can be maintained by the solutions to (\ref{eq:rev_eqn}) remains an open problem, since analytical approach to solve these equations remain difficult. 
Nevertheless, one may justify that (\ref{eq:rev_inc}) should be respected by adding (\ref{eq:rev_eqn_g}) to (\ref{eq:rev_eqn_m}) to obtain:
\begin{displaymath}
  \pp{\Theta}{t} + \frac{1}{r^2}\po{r}\left[r^2\Theta V + r^2MU\right] = f+g+h\;,
\end{displaymath}
where $\Theta = G+N+M$; and then incorporating (\ref{eq:rev_eqn_v}):
\begin{equation}\label{eq:rev_case_inc}
  \pp{\Theta}{t} + \frac{1}{r^2}\po{r}\left[r^2(\Theta-\theta)V\right] = 0\;.
\end{equation}
Clearly, the equation~(\ref{eq:rev_inc}), or equivalently $\Theta\equiv\theta$ is a solution to the latest equation. 

In this section, we consider a simple case whose parameters and initial/boundary conditions are given as follows:
\begin{itemize}
  \item Most diminishing rates are set to zero except for $\lambda$, which models the self-production of the glioma cells:
  \begin{displaymath}
    \lambda = 1.0\;,\quad
    \mu = \delta = \rho = 0.0\;.
  \end{displaymath}
  \item We normalize the cell number by setting $\theta=1.0$, and in the chemoattractant equation:
  \begin{displaymath}
    m = 30.0\;,\quad
    \beta = 1.0\;,\quad
    \gamma = 0.0\;,\quad
    \nu = 1.0\;,\quad
    \alpha = 1.0\;.
  \end{displaymath}
\item The initial radius is $R(0)=1$, and the initial cell numbers are:
  \begin{displaymath}
    G(r,0) = 0.5\;,\quad
    N(r,0) = 0.0\;,\quad
    M(r,0) = 0.5\;,
  \end{displaymath}
  for all $0\le r\le 1$ and the initial chemoattractant concentration is:
  \begin{displaymath}
    A(r,0) = \left\{\begin{array}{lcl}
        \frac{5}{3}-\frac{1}{6}r^2 & & 0\le r\le 1 \\ \vspace*{-.12in} \\
        \frac{3}{2}e^{-\frac{2}{9}(r-1)} & & r\ge 1
    \end{array}\right. 
  \end{displaymath}
  \item The boundary condition for $M$ is $M_{\bc}=0.5$.
\end{itemize}

The numerical method of Section~\ref{sec:rev_fvm} is used to solve this problem until $T=1.0$ with $N_\eta=50$ uniform intervals and fixed time step size $\Delta t=0.005$, which satisfies the Courant stability condition for all steps. 
The radius growth history and the final cell numbers are plotted in the left panel and the right panel of Figure~\ref{fg:rev_case_sol}, respectively.
\begin{figure}\centering
  \begin{subfigure}[b]{.48\textwidth}\centering
    \includegraphics[width=.8\textwidth]{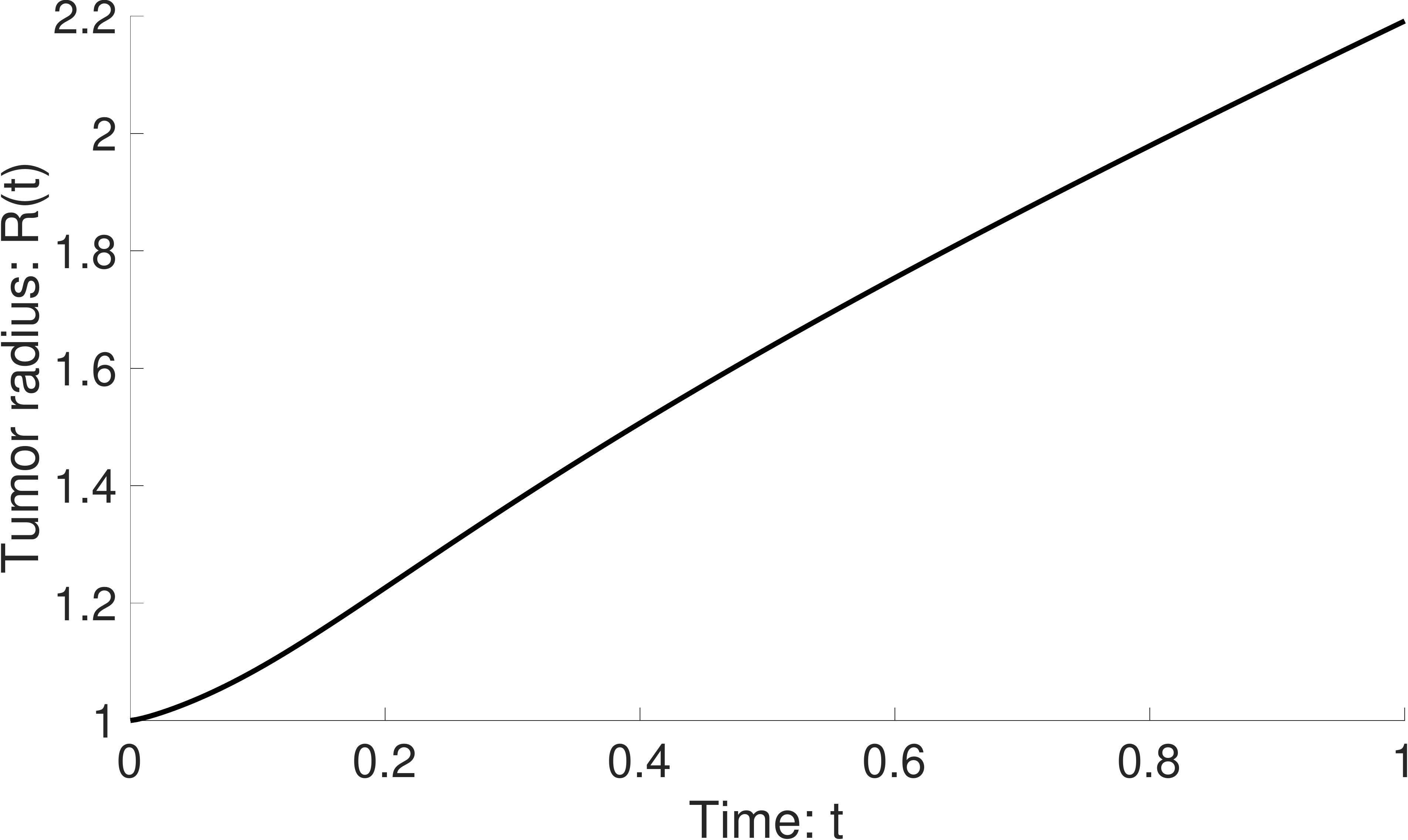}
    \caption{Radius growth history.}
    \label{fg:rev_case_sol_rad}
  \end{subfigure}
  \begin{subfigure}[b]{.48\textwidth}\centering
    \includegraphics[width=.8\textwidth]{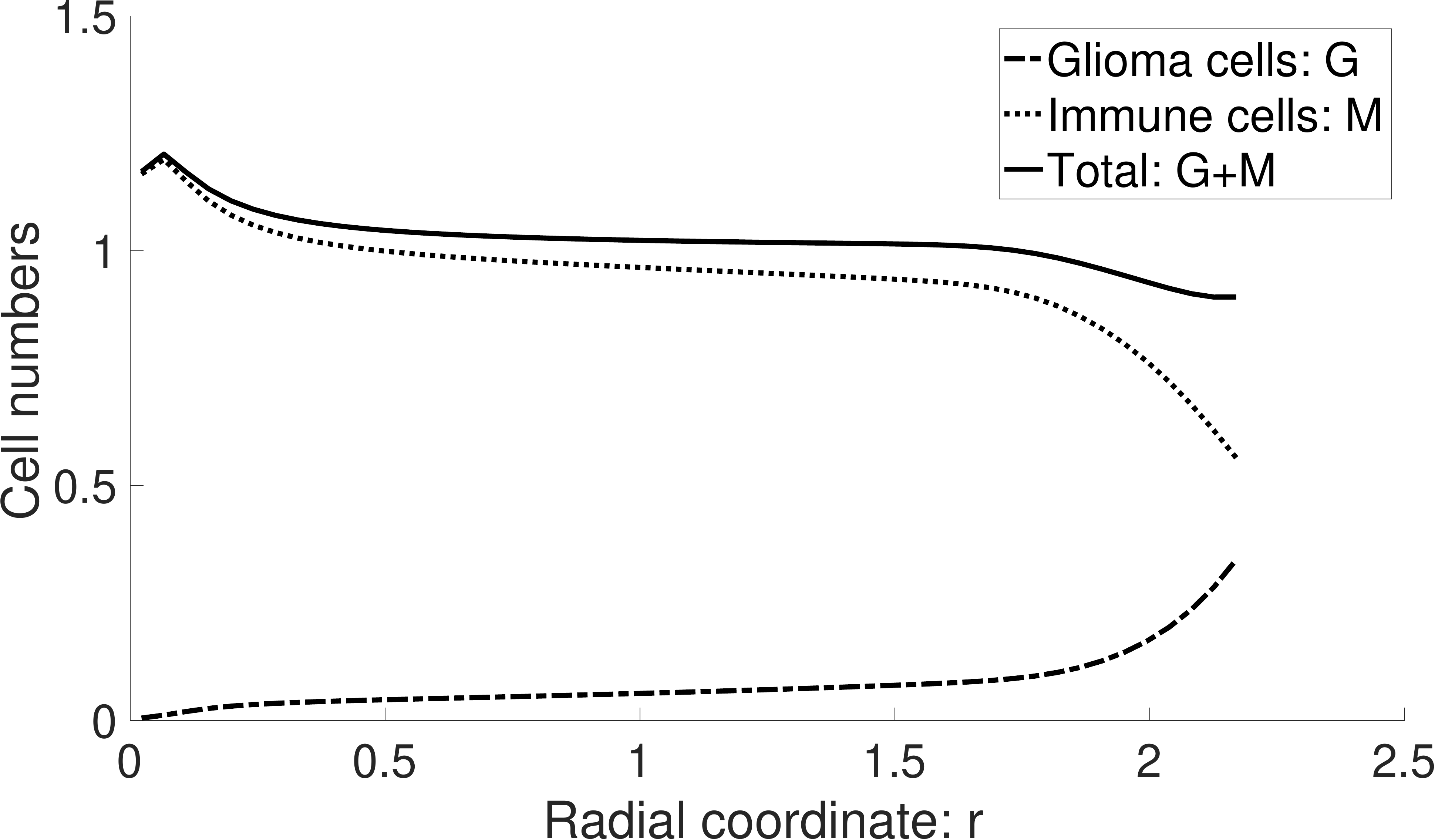}
    \caption{Cell numbers at $T=1.0$.}
    \label{fg:rev_case_sol_cell}
  \end{subfigure}
  \caption{Solutions to the tumor growth problem.}
  \label{fg:rev_case_sol}
\end{figure}

Note that for this problem, $N$ is always zero (and so is our numerical solutions), hence we clearly observe the violation of incompressibility in the numerical solutions at $T=1.0$, especially near the tumor center.
To make this point clearer, the $L_1\textrm{\sc-}$norm of $G+M-1$ is defined as:
\begin{equation}\label{eq:rev_case_l1}
  d_\theta(t^n) \eqdef \int_0^{R(t^n)}\abs{G(r,t^n)+M(r,t^n)-1}dr \approx \frac{R^n}{N_\eta}\sum_{j=1}^{N_\eta}\abs{G_{j-1/2}^n+M_{j-1/2}^n-1}\;,
\end{equation}
where $R^n\approx R(t^n)$ is the numerical solution of the radius at $t^n$.
The history of $d_\theta$ is provided in Figure~\ref{fg:rev_case_l1}, where we observe violation of the incompressibility constraint in increasing magnitude as $t$ grows.
\begin{figure}\centering
  \includegraphics[width=.384\textwidth]{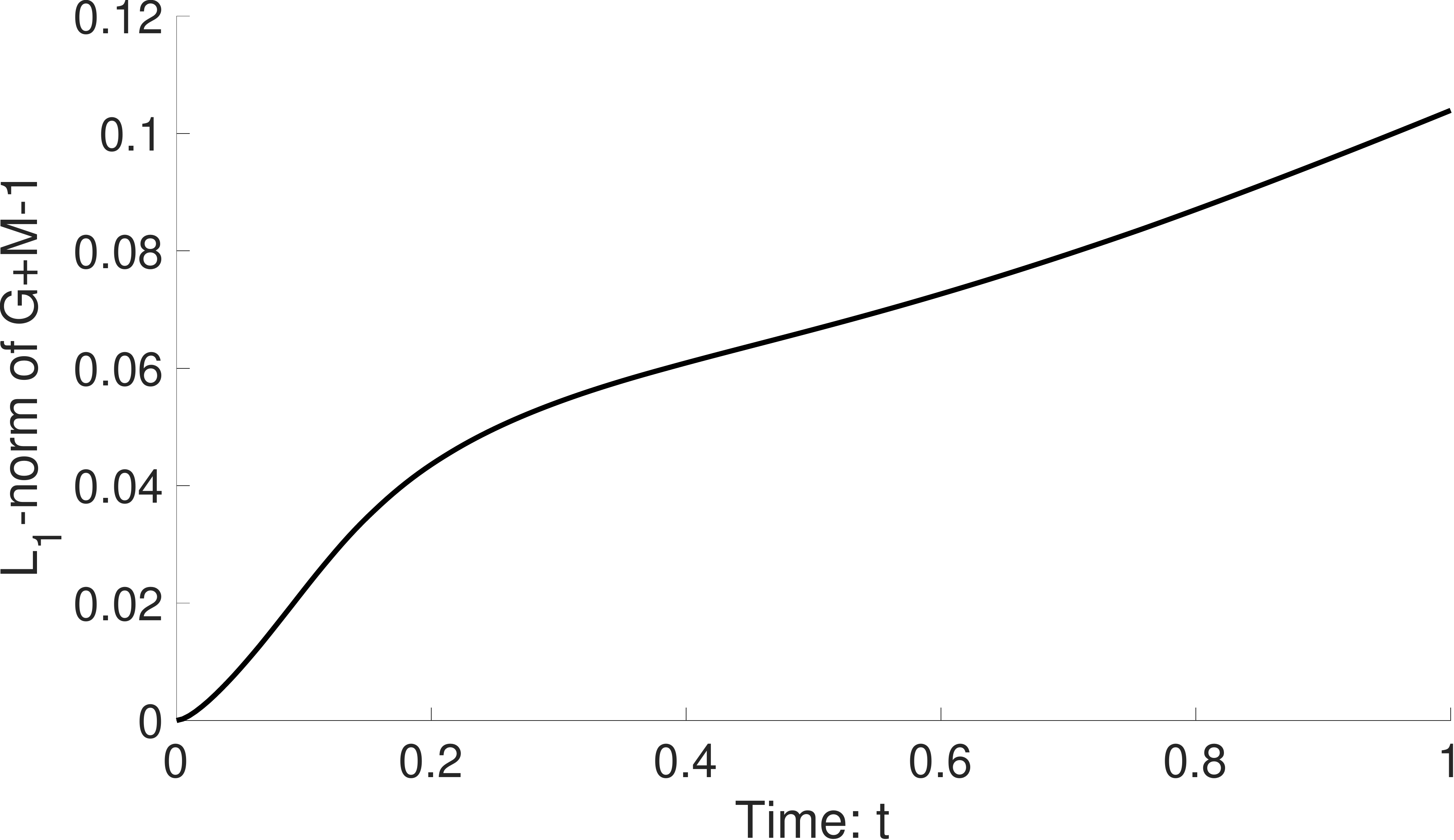}
  \caption{History of the incompressibility constraint violation index $d_\theta$.}
  \label{fg:rev_case_l1}
\end{figure}
In the rest of the paper, we try to address this issue and investigate enhancement of existing finite volume methods to improve the numerical results.

\section{A Model Problem and The Totality Conservation Law}
\label{sec:model}
To make the idea clear, we consider a simplified model instead of the original one.
First of all, only two cell species are considered, namely the glioma cells $G$ and the immune cells $M$.
Second, noticing that the chemoattractant $A$ is only used to compute the velocity field $U$, in this simplified model we treat $U$ as a given velocity field and denote it by $u$ since it is prescribed; $A$ is thusly ignored altogether.
To this end, the governing equations in spherical coordinate are given by:
\begin{subequations}\label{eq:model_pde}
\begin{align}
\label{eq:model_pde_g}
&\pp{G}{t} + \frac{1}{r^2}\po{r}\left[r^2GV\right] = f\;,\\
\label{eq:model_pde_m}
&\pp{M}{t} + \frac{1}{r^2}\po{r}\left[r^2M(V+u)\right] = h\;,\\
\label{eq:model_pde_v}
&\frac{1}{r^2}\po{r}\left[r^2(V+uM)\right] = f + h\;,\quad V(0,t) = 0\;;
\end{align}
where $0\le r\le R(t)$ and $0\le t\le T$; the radius function $R: [0,\;T]\mapsto\mathbb{R}^+$ satisfies:
\begin{equation}\label{eq:model_pde_r}
R'(t) = V(R(t),t)\;.
\end{equation}
As before, the lower case letters in (\ref{eq:model_pde}) represent prescribed functions:
\begin{equation}\label{eq:model_pde_presc}
  u = u(r,t,G,M)\;,\quad
  f = f(r,t,G,M)\;,\quad
  g = g(r,t,G,M)\;.
\end{equation}
We further require that $u=0$ at $r=0$ for all $t\in[0,\; T]$.
\end{subequations}

Note that in (\ref{eq:model_pde_v}) there is no $\theta$, comparing to the previous model; indeed, we have supposed that $\theta=1$ and require the initial condition to satisfy:
\begin{equation}\label{eq:model_ic}
G(r,0) + N(r,0) = 1\;,\quad
\forall\; 0\le r \le R(0)\;.
\end{equation}

Similar as before, we can define the total number $\Theta(r,t)=G(r,t)+H(r,t)$; then the incompressibility assumption requires $\Theta\equiv1$.
If this holds, we actually have a very convenient way to estimate the growth of the tumor.
In particular, let $C(t)$ denote the total number of cells in the tumor; then on the one hand the assumption $\Theta\equiv1$ indicates:
\begin{equation}\label{eq:model_tot}
  C(t) = \int_0^{R(t)}4\pi r^2\Theta(r,t)dr = \int_0^{R(t)}4\pi r^2dr = \frac{4}{3}\pi R(t)^3\;,
\end{equation}
hence the rate of change in $C(t)$ is:
\begin{equation}\label{eq:model_growth}
  C'(t) = 4\pi R'(t)R(t)^2\;.
\end{equation}
On the other hand, the only mechanism such that the new cells can enter the tumor is through the boundary condition for $M$ at $r=R(t)$:
\begin{equation}\label{eq:model_growth_alt}
  C'(t) = -4\pi R(t)^2u(R(t),t)\check{M}(t)\;,
\end{equation}
where $u$ is the prescribed infiltration velocity and $\check{M}(t)$ is the flow out of/into the tumor:
\begin{equation}\label{eq:model_growth_flow}
\check{M}(t) = \left\{\begin{array}{lcl}
M(R(t),t)\;, & & u(R(t),t)\ge0\;; \\
M_{\bc}(t)\;, & & u(R(t),t) < 0\;,
\end{array}\right.
\end{equation}
where as before $M_{\bc}(t)$ is the prescribed ambient number of immune cells.
Equating (\ref{eq:model_growth}) and (\ref{eq:model_growth_alt}), we obtain an ODE for $R(t)$:
\begin{equation}\label{eq:model_ode}
  R'(t) = -u(R(t),t)\check{M}(t)\;,
\end{equation}
which will help us design numerical tests for which the exact tumor growth curve can be calculated.

\subsection{The totality conservation law}
\label{sec:model_tclr}
Adding (\ref{eq:model_pde_g}) and (\ref{eq:model_pde_m}) then equating the right hand side with that of (\ref{eq:model_pde_v}), we obtain an analogy of (\ref{eq:rev_case_inc}):
\begin{displaymath}
  \pp{(r^2\Theta)}{t} + \po{r}\left[r^2\Theta V + r^2Mu\right] = r^2(f + g) = \po{r}\left[r^2(V+uM)\right]\;,
\end{displaymath}
or equivalently:
\begin{equation}\label{eq:model_tclr}
  \pp{(r^2\Theta)}{t} + \po{r}\left[r^2(\Theta-1)V\right] = 0\;,
\end{equation}
which admits the solution $\Theta(r,t)\equiv 1$ regardless of the other variables if the initial condition~(\ref{eq:model_ic}) holds.

If we replace one of $G$ and $M$ by their sum $\Theta$, an equivalent PDE system is obtained by replacing either (\ref{eq:model_pde_g}) or (\ref{eq:model_pde_m}) by (\ref{eq:model_tclr}) without changing the solutions.
Hence we expect the incompressibility constraint $G(r,t)+M(r,t) = \Theta(r,t) \equiv 1$ in the solutions of the original system of equations.

Because (\ref{eq:model_tclr}) describes the conservation of the sum of the two species, we call it the {\it totality conservation law} or {\it TCL} in the context of current work and expect the numerical method satisfies a discrete version to be specified later.

\subsection{The model, TCL, and GCL in normalized coordinate system}
\label{sec:model_norm}
Similar as before, after the coordinate transformation $(r,t)\mapsto(\eta,\tau) = (r/R(t),t)$, we obtain the model in the normalized coordinates:
\begin{subequations}\label{eq:model_norm_pde}
\begin{align}
\label{eq:model_norm_pde_g}
&\pp{(\eta^2R^2G)}{\tau} + \po{\eta}\left[\left(\frac{V}{R}-\frac{\eta R'}{R}\right)\eta^2R^2G\right] = \eta^2R^2f-\eta^2R'RG\;,\\
\label{eq:model_norm_pde_m}
&\pp{(\eta^2R^2M)}{\tau} + \po{\eta}\left[\left(\frac{V}{R}-\frac{\eta R'}{R}+\frac{u}{R}\right)\eta^2R^2M\right] = \eta^2R^2g-\eta^2R'RM\;,\\
\label{eq:model_norm_pde_v}
&\frac{1}{\eta^2}\po{\eta}\left[\eta^2\left(\frac{V}{R}+\frac{u}{R}M\right)\right] = f + g\;,\quad V(0,t) = 0\;;
\end{align}
the computational domain is $(\eta,\tau)\in[0,\; 1]\times[0,\; T]$ for some positive $T>0$; and $R: [0,\; T]\mapsto \mathbb{R}^+$ denotes the radious of the spherical domain, which satisfies:
\begin{equation}\label{eq:model_norm_pde_r}
R'(\tau) = V(1,\tau)\;.
\end{equation}
The lower case letters in (\ref{eq:model_norm_pde}) represent prescribed source terms.
\end{subequations}

Correspondingly, the equation~(\ref{eq:model_tclr}) is converted to:
\begin{displaymath}
\pp{(\eta^2R^2\Theta)}{\tau} - \frac{\eta R'}{R}\po{\eta}\left[\eta^2R^2\Theta\right] + \frac{1}{R}\po{\eta}\left[\eta^2R^2(\Theta-1)V\right] = 0\;,
\end{displaymath}
or equivalently:
\begin{equation}\label{eq:model_gcl_tcl}
  \pp{(\eta^2R^2\Theta)}{\tau} + \po{\eta}\left[\eta^2R(\Theta-1)V\right] - \po{\eta}\left[\eta^3R'R\Theta\right] = - \eta^2R'R\Theta\;,
\end{equation}
which is the TCL in the normalized coordinates.
In (\ref{eq:model_tclr}), both terms vanishes if we set $\Theta=1$; whereas in (\ref{eq:model_gcl_tcl}) setting $\Theta=1$ yields the identity:
\begin{equation}\label{eq:model_gcl_eqn}
  \pp{(\eta^2R^2)}{\tau} + \po{\eta}\left[-\eta^3R'R\right] = -\eta^2R'R\;.
\end{equation}
This equation only involve geometric quantities and it is rooted in using a mesh coordinate (normalized coordinate system) that is different from the physical one (the radial coordinates).
Similar identities are studied in other contexts, especially the arbitrary Lagrangian-Eulerian (ALE) methods, see for example~\cite{CFarhat:2001a, ALOrtega:2011a}, where it is called the {\it geometric conservation law} or {\it GCL}.
In this work, we follow this convention and call (\ref{eq:model_gcl_eqn}) the GCL for the free-boundary problem in radial coordinates.
At the continuous level, (\ref{eq:model_gcl_eqn}) holds naturally; but we will see in a moment that it may not hold at the discrete level.
Existing literature has demonstrated that violating GCL at the discrete level lead to unstable solutions; in this work, we thusly require the proposed method to satisfy a discrete version of GCL, called the discrete geometric conservation law (DGCL), which will also be specified in the next section.

\section{An Enhanced Finite Volume Method}
\label{sec:fvm}
Neither TCL nor GCL is automatically satisfied by classical finite volume discretizations.
For example to see why GCL could be violated, let us consider a numerical discretization of (\ref{eq:model_gcl_eqn}), so that GCL is satisfied discretely, and look at what this discretization may look like.
Using a mesh with $N_{\eta}$ uniform intervals and nodal velocities $V_j^n$ where $0\le j\le N_{\eta}$, if the straightforward forward Euler time-integrator is used (see for example, Section~\ref{sec:rev_fvm}), we have the following formula to update the solutions from $t^n$ to $t^{n+1}$:
\begin{subequations}\label{eq:fvm_case}
\begin{align}
\label{eq:fvm_case_gcl}
&\frac{1}{\Delta\tau^n}\left[\eta_{j-1/2}^2(R^{n+1})^2-\eta_{j-1/2}^2(R^n)^2\right] + R'^nR^n\mathcal{D}_{j-1/2}\left[-\eta^3\right] = -\eta_{j-1/2}^2R'^nR^n\;,\\
\label{eq:fvm_case_rp}
&R'^n=V_{N_{\eta}}^n\;,\\
\label{eq:fvm_case_r}
  &R^{n+1} = R^n+\Delta\tau^nV_{N_{\eta}}^n\;,
\end{align}
\end{subequations}
where (\ref{eq:fvm_case_gcl}) collocates at the interval center $\eta_{j-1/2}$ and $\mathcal{D}_{j-1/2}$ is the spatial discretization for $\partial_{\eta}$ at $\eta_{j-1/2}$ as a result of the finite volume discretizations of (\ref{eq:model_pde}).
Rearranging (\ref{eq:fvm_case_gcl}) there is:
\begin{displaymath}
  \mathcal{D}_{j-1/2}\left[-\eta^3\right] = -\eta_{j-1/2}^2 - \frac{\eta_{j-1/2}^2}{\Delta\tau^n}\frac{(R^n+\Delta\tau^nV_{N_{\eta}}^n)^2-(R^n)^2}{V_{N_{\eta}}^nR^n}
  = -\eta_{j-1/2}^2 - \eta_{j-1/2}^2\left[2+\frac{\Delta\tau^nV_{N_{\eta}}^n}{R^n}\right]\;.
\end{displaymath}
This is a highly undesirable property, because it means that when we apply the chosen numerical discretization $\mathcal{D}_{j-1/2}$ to a purely geometric quantity $-\eta^3$, the result needs to depend on the solutions of both $V$ and $R$.

An easy way to fix the issue is to make sure that the radius update satisfies:
\begin{equation}\label{eq:fvm_t_gcl}
  R'^n = \frac{(R^{n+1})^2-(R^n)^2}{2\Delta\tau^nR^n}\;,
\end{equation}
then (\ref{eq:fvm_case_gcl}) reduces to:
\begin{equation}\label{eq:fvm_s_gcl}
  \mathcal{D}_{j-1/2}(-\eta^3) = -3\eta_{j-1/2}^2\;,
\end{equation}
which is independent of $V$ and $R$ as desired.

For example, if one wish to update the radius as $R^{n+1} = R^n + \Delta\tau^n R'^n$, c.f., (\ref{eq:fvm_case_r}), then it requires $R'^n$ to be computed as $R'^n = V_{N_{\eta}}^n(1+(1/2)\Delta\tau^n V_{N_{\eta}}^n/R^n)$ rather than (\ref{eq:fvm_case_rp}).
In this paper, however, we propose to compute $R'^n$ as:
\begin{equation}\label{eq:fvm_t_gcl_rp}
  R'^n = \left(1-\frac{1}{4}\Delta\eta^2\right)^{-1}V_{N_{\eta}}^n\;,
\end{equation}
and then compute $R^{n+1}$ according to (\ref{eq:fvm_t_gcl}).
The motivation is to make sure that our numerical method is compatible with the no-flux biological condition at the moving boundary, see the discussion after the proof of (\ref{thm:fvm_gtcl}).

\medskip

The preceding case study indicates that we must design the time-integrator carefully; furthermore, the spatial discretization $\mathcal{D}_{j-1/2}$ needs to compute the derivative of third-degree polynomials exactly, as required by (\ref{eq:fvm_s_gcl}).

\medskip

The rest of this section focuses on constructing finite volume methods that satisfy both the GCL and TCL in a discrete sense, which is yet to be made precise. 
To this end, we follow the same notations as before and denote discrete cell numbers by $G_{j-1/2}^n$ and $M_{j-1/2}^n$, where $1\le j\le N_{\eta}$, and they represent:
\begin{align}
\label{eq:fvm_cell_g}
&G_{j-1/2}^n\approx\frac{1}{\Delta\eta\eta_{j-1/2}^2(R^n)^2}\int_{\eta_{j-1}}^{\eta_j}\eta^2(R^n)^2G(\eta,\tau^n)d\eta\;,\\
\label{eq:fvm_cell_m}
&M_{j-1/2}^n\approx\frac{1}{\Delta\eta\eta_{j-1/2}^2(R^n)^2}\int_{\eta_{j-1}}^{\eta_j}\eta^2(R^n)^2M(\eta,\tau^n)d\eta\;,\quad 1\le j\le N_{\eta}\;;
\end{align}
the discrete velocities are given by:
\begin{align}
\label{eq:fvm_node_v}
&V_j^n\approx V(\eta_j,\tau^n)\;,\quad 0\le j\le N_{\eta}\;,\\
\label{eq:fvm_node_u}
&u_j^n = u(\eta_j,\tau^n)\;,\quad 0\le j\le N_{\eta}\;,
\end{align}
where no special approximation is needed for $u$ since it can be evaluated explicitly, c.f.~(\ref{eq:model_pde_presc}).

The remainder of this section is organized as follows. 
A general finite volume formulation is provided in Section~\ref{sec:fvm_gen} and our main result is given in Section~\ref{sec:fvm_thm}, where both DGCL and TGCL are defined and sufficient conditions for numerical methods to satisfy these conditions are provided.
The subsequent sections then focus on various numerical fluxes that obey these conditions.

\subsection{A general finite volume formulation}
\label{sec:fvm_gen}
The explicit first-order time-accurate finite volume formulation of (\ref{eq:model_norm_pde_g}) and (\ref{eq:model_norm_pde_m}) is obtained by integrating these equations over each interval $[\eta_{j-1},\; \eta_j]$ and then discretizing the time-derivative by forward-Euler method:
\begin{align}
\label{eq:fvm_gen_g}
&\ \frac{\eta_{j-1/2}^2[(R^{n+1})^2G_{j-1/2}^{n+1}-(R^n)^2G_{j-1/2}^n]}{\Delta\tau^n} + \frac{1}{\Delta\eta}\left[F_j^{G,n}-F_{j-1}^{G,n}\right] \\
\notag
=&\ \eta_{j-1/2}^2(R^n)^2f_{j-1/2}^n-\eta_{j-1/2}^2R'^nR^nG_{j-1/2}^n\;, \\
\label{eq:fvm_gen_m}
&\ \frac{\eta_{j-1/2}^2[(R^{n+1})^2M_{j-1/2}^{n+1}-(R^n)^2M_{j-1/2}^n]}{\Delta\tau^n} + \frac{1}{\Delta\eta}\left[F_j^{M,n}-F_{j-1}^{M,n}\right] \\
\notag
=&\ \eta_{j-1/2}^2(R^n)^2g_{j-1/2}^n-\eta_{j-1/2}^2R'^nR^nM_{j-1/2}^n\;,\quad 1\le i\le N_{\eta}\;.
\end{align}
$F^{G,n}_j$ and $F^{M,n}_j$ are numerical fluxes for $G$ and $M$ at $\eta_j$, respectively:
\begin{align}
  \label{eq:fvm_gen_flux_g}
  &F^{G,n}_j \approx \left(\frac{V}{R}-\frac{\eta R'}{R}\right)\eta^2R^2G\Big|_{\eta=\eta_j,\;\tau=\tau^n}\;, \\
  \label{eq:fvm_gen_flux_m}
  &F^{M,n}_j \approx \left(\frac{V}{R}-\frac{\eta R'}{R}+\frac{u}{R}\right)\eta^2R^2M\Big|_{\eta=\eta_j,\;\tau=\tau^n}\;.
\end{align}
In Section~\ref{sec:rev_fvm}, the velocities $V/R-\eta R'/R$ and $V/R-\eta R'/R+u/R$ are used to compute the two fluxes $F^G_j$ and $F^M_j$, respectively.
For our problems, however, it is advantageous to consider each component of the velocity separately; namely, we segregate the numerical fluxes as:
\begin{align}
\label{eq:fvm_gen_flux_f_seg}
  F_j^{G,n} &= F_{V,j}^{G,n} + F_{R',j}^{G,n}\;, \\
\notag
  &F_{V,j}^{G,n}\approx V\eta^2RG\;,\ F_{R',j}^{G,n}\approx -\eta^3R'RG\;;\\
\label{eq:fvm_gen_flux_g_seg}
  F_j^{M,n} &= F_{V,j}^{M,n} + F_{R',j}^{M,n} + F_{u,j}^{M,n}\;, \\
\notag
  &F_{V,j}^{M,n}\approx V\eta^2RM\;,\ F_{R',j}^{M,n}\approx -\eta^3R'RM\;,\ F_{u,j}^{M,n}\approx u\eta^2RM\;.
\end{align}

The velocity equation is obtained similarly as before, but we keep the approximation to $u\eta^2R^2M$ as unspecified: 
\begin{equation}\label{eq:fvm_gen_v}
  \eta_j^2R^nV_j^n+\mathscr{F}_{u,j}^{M,n} = \sum_{k=1}^j\Delta\eta\left(\eta_{k-1/2}^2(R^n)^2f_{k-1/2}^n+\eta_{k-1/2}^2(R^n)^2g_{k-1/2}^n\right)\;.
\end{equation}
Here $\mathscr{F}_{u,j}^{M,n}$ approximates $u\eta^2R^2M$ at $(\eta_j,\;\tau^n)$; and the source terms on the right hand side are computed the same way as those in (\ref{eq:fvm_gen_g}) and (\ref{eq:fvm_gen_m}).
In Section~\ref{sec:rev_fvm}, $\mathscr{F}_{u,j}^{M,n}$ is approximated by averaging $M_{j-1/2}^n$ and $M_{j+1/2}^n$; as we will see soon, this is a good choice for our problem.

\subsection{Sufficient conditions for DTCL and DGCL}
\label{sec:fvm_thm}
It is fair to assume that we use the same flux function to compute the numerical fluxes associated with the same velocity, such as $F_{V,j}^{G,n}$ and $F_{V,j}^{M,n}$; to this end we suppose:
\begin{align}
  \label{eq:fvm_thm_flux_v}
  &F_{V,j}^{X,n} = \mathcal{F}_j^n(\{X_{j-1/2+k}^n\;:\ -l\le k\le r\},\;\mathscr{P})\;, \\
  \label{eq:fvm_thm_flux_r}
  &F_{R',j}^{X,n} = \hat{\mathcal{F}}_j^n(\{X_{j-1/2+k}^n\;:\ -l\le k\le r\},\;\hat{\mathscr{P}})\;, 
\end{align}
where $l\ge0$ and $r\ge1$ are fixed numbers denoting the stencil of the flux function, $X$ represents either species, and the parameter sets $\mathscr{P}$ and $\hat{\mathscr{P}}$ are placeholders for high-resolution fluxes that are described later.

We distinguish the flux functions $\mathcal{F}_j^n$ and $\hat{\mathcal{F}}_j^n$ because the former approximates the fluxes due to a spatially varying velocity $V$ whereas the latter can be interpreted as fluxes due to a spatially constant velocity $R'$;
furthermore, we maintain the subscript $j$ and the superscript $n$ in these generic functions to indicate their dependence on the spatial coordinates $\eta$, domain size $R^n$, as well as $R'^n$, which are determined independently from the finite volume discretizations.

For our next purpose, we note that both flux functions are in the form $\mathcal{F}(\{X_{j-1/2+k}^n\;:\ -l\le k\le r\},\;\cdots)$, where the omitted quantities represent the parameters that are the same when the flux function is applied to compute fluxes for different species, such as $G$ and $M$, respectively.

\begin{definition}\label{def:fvm_addflux}
  The flux function $\mathcal{F}(\{X_{j-1/2+k}^n\;:\ -l\le k\le r\},\;\cdots)$ is called {\it additive} if for all $X$, $Y$ and $Z=X+Y$:
  \begin{align}
    \notag
    &\mathcal{F}(\{X_{j-1/2+k}^n\;:\ -l\le k\le r\},\;\cdots) + 
    \mathcal{F}(\{Y_{j-1/2+k}^n\;:\ -l\le k\le r\},\;\cdots) \\
    \label{eq:fvm_addflux}
    =\ &\mathcal{F}(\;\{Z_{j-1/2+k}^n\;:\ -l\le k\le r\},\;\cdots)\;.
  \end{align}
  where the omitted inputs are kept the same in all the three function evaluations.
\end{definition}

Furthermore, we define the $V$-consistency for the flux function $\mathcal{F}_j^n$ of (\ref{eq:fvm_thm_flux_v}) and cubic-preserving for the flux function $\hat{\mathcal{F}}_j^n$ of (\ref{eq:fvm_thm_flux_r}) as follows.

\begin{definition}\label{def:fvm_vc}
  The numerical flux function $\mathcal{F}_j^n$ of (\ref{eq:fvm_thm_flux_v}) is {\it $V$-consistent} if for all $V_j^n$:
  \begin{equation}\label{eq:fvm_vc}
    F_{V,j}^{1,n} = \eta_j^2R^nV_j^n\;,
  \end{equation}
  that is, setting $X_{j-1/2+k}^n=1,\;\forall k$ in the right hand side of (\ref{eq:fvm_thm_flux_v}) yields $\eta_j^2R^nV_j^n$.
\end{definition}

\begin{definition}\label{def:fvm_cubic}
  The numerical flux function $\hat{\mathcal{F}}_j^n$ of (\ref{eq:fvm_thm_flux_r}) is {\it cubic-preserving} if 
  \begin{equation}\label{eq:fvm_cubic}
    \frac{1}{\Delta\eta}\left(F_{R',j}^{1,n} - F_{R',j-1}^{1,n}\right) = -3\eta_{j-1/2}^2R'^nR^n\;.
  \end{equation}
\end{definition}
Note that $\mathcal{F}_{R',j}^{X,n}$ can be treated as the flux for an advection equation with the spatially constant velocity $-R'^nR^n$ and convected variable $\eta^3X$, this will be our basis to construct a cubic-preserving flux function, see the further discussions in Section~\ref{sec:fvm_rflux}. 

\medskip

The purpose of this section is to derive sufficient conditions such that our method satisfies GCL and TCL discretely.
To this end, we have the following definitions:
\begin{definition}\label{def:fvm_gcl}
  The method given by (\ref{eq:fvm_gen_g}), (\ref{eq:fvm_gen_m}) and (\ref{eq:fvm_gen_v}) satisfies the DGCL provided that:
  Suppose $M_{j-1/2}^m+G_{j-1/2}^m=1$ for all $j$ and $m=n,n+1$, then we can derive (\ref{eq:fvm_gen_v}) from (\ref{eq:fvm_gen_g}) and (\ref{eq:fvm_gen_m}).
\end{definition}

\begin{definition}\label{def:fvm_tcl}
  The method given by (\ref{eq:fvm_gen_g}), (\ref{eq:fvm_gen_m}) and (\ref{eq:fvm_gen_v}) satisfies the DTCL if they lead to a conservative discretization of (\ref{eq:model_tclr}).
\end{definition}

\medskip

Now we state the main theorem that will eventually guide us in the construction of the enhanced numerical methods.
\begin{theorem}\label{thm:fvm_gtcl}
  The numerical method given by (\ref{eq:fvm_gen_g}), (\ref{eq:fvm_gen_m}) and (\ref{eq:fvm_gen_v}) satisfies both DGCL and DTCL if: (1) $\mathcal{F}_j^n$ is additive and $V$-consistent, (2) $\hat{\mathcal{F}}_j^n$ is additive and cubic-preserving, (3) $\mathscr{F}_{u,j}^{M,n}=F_{u,j}^{M,n}$, and (4) $R'^n$ equals the right hand side of (\ref{eq:fvm_t_gcl}).
\end{theorem}
\begin{proof}
Adding (\ref{eq:fvm_gen_g}) and (\ref{eq:fvm_gen_m}) then incorporating (\ref{eq:fvm_gen_v}), we have:
\begin{align}
  \notag
  &\ \frac{\eta_{j-1/2}^2}{\Delta\tau^n}\left[(R^{n+1})^2(G_{j-1/2}^{n+1}+M_{j-1/2}^{n+1})-(R^n)^2(G_{j-1/2}^n+M_{j-1/2}^n)\right] + \frac{1}{\Delta\eta}\left(F_{u,j}^{M,n}-F_{u,j-1}^{M,n}\right) \\
  \notag 
  &\ \frac{1}{\Delta\eta}\left(F_{V,j}^{G,n}+F_{V,j}^{M,n}-F_{V,j-1}^{G,n}-F_{V,j-1}^{M,n}\right) + \frac{1}{\Delta\eta}\left(F_{R',j}^{G,n}+F_{R',j}^{M,n}-F_{R',j-1}^{G,n}-F_{R',j-1}^{M,n}\right) \\
  \label{eq:fvm_thm_gtcl_1}
  &=\ \eta_{j-1/2}^2(R^n)^2f_{j-1/2}^n + \eta_{j-1/2}^2(R^n)^2g_{j-1/2}^n - \eta_{j-1/2}^2R'^nR^n(G_{j-1/2}^n+M_{j-1/2}^n) \\
  \notag
  &=\ \frac{R^n(\eta_j^2V_j^n-\eta_{j-1}^2V_{j-1}^n)}{\Delta\eta} + \frac{1}{\Delta\eta}\left(\mathscr{F}_{u,j}^{M,n}-\mathscr{F}_{u,j-1}^{M,n}\right)-\eta_{j-1/2}^2R'^nR^n(G_{j-1/2}^n+M_{j-1/2}^n)\;.
\end{align}
Define $\Theta_{j-1/2}^n = G_{j-1/2}^n + M_{j-1/2}^n$ and $\Theta_{j-1/2}^{n+1} = G_{j-1/2}^{n+1} + M_{j-1/2}^{n+1}$ as before; following the additivity of the fluxes $\mathcal{F}_j^n$ and $\hat{\mathcal{F}}_j^n$ we obtain:
\begin{displaymath}
  F_{V,j}^{G,n} + F_{V,j}^{M,n} = F_{V,j}^{\Theta,n}\quad\textrm{ and }\quad
  F_{R',j}^{G,n} + F_{R',j}^{M,n} = F_{R',j}^{\Theta,n}\;.
\end{displaymath}
Invoking in addition the assumption that $\mathscr{F}_{u,j}^{M,n}=F_{u,j}^{M,n}$, we obtain from (\ref{eq:fvm_thm_gtcl_1}):
\begin{align}
  \label{eq:fvm_thm_gtcl_2}
  &\ \frac{\eta_{j-1/2}^2}{\Delta\tau^n}\left[(R^{n+1})^2\Theta_{j-1/2}^{n+1}-(R^n)^2\Theta_{j-1/2}^n\right] + \frac{1}{\Delta\eta}\left(F_{V,j}^{\Theta,n}-F_{V,j-1}^{\Theta,n}\right) +\frac{1}{\Delta\eta}\left(F_{R',j}^{\Theta,n}-F_{R',j-1}^{\Theta,n}\right) \\
  \notag
  =&\ \frac{R^n(\eta_j^2V_j^n-\eta_{j-1}^2V_{j-1}^n)}{\Delta\eta} - \eta_{j-1/2}^2R'^nR^n\Theta_{j-1/2}^n\;.
\end{align}
Clearly (\ref{eq:fvm_thm_gtcl_2}) represents a conservative finite volume discretization of the continuous totality conservation law~(\ref{eq:model_tclr}) using the same flux functions $\mathcal{F}_j^n$ and $\hat{\mathcal{F}}_j^n$; hence the method satisfies DTCL.

\smallskip

Now we move on to show DGCL and to this end assume $\Theta_j^n\equiv1$ and $\Theta_j^{n+1}\equiv1$, then (\ref{eq:fvm_thm_gtcl_2}) reduce to:
\begin{align}
  \label{eq:fvm_thm_gtcl_3}
  &\ \frac{\eta_{j-1/2}^2}{\Delta\tau^n}\left[(R^{n+1})^2-(R^n)^2\right] + \frac{1}{\Delta\eta}\left(F_{V,j}^{1,n}-F_{V,j-1}^{1,n}\right) +\frac{1}{\Delta\eta}\left(F_{R',j}^{1,n}-F_{R',j-1}^{1,n}\right) \\
  \notag
  =&\ \frac{R^n(\eta_j^2V_j^n-\eta_{j-1}^2V_{j-1}^n)}{\Delta\eta} - \eta_{j-1/2}^2R'^nR^n\;.
\end{align}
Since $R'^n=((R^{n+1})^2-(R^n)^2)/(2\Delta\tau^nR^n)$, (\ref{eq:fvm_thm_gtcl_3}) is equivalent to:
\begin{equation}
  \frac{1}{\Delta\eta}\left(F_{V,j}^{1,n}-F_{V,j-1}^{1,n}\right) +\frac{1}{\Delta\eta}\left(F_{R',j}^{1,n}-F_{R',j-1}^{1,n}\right) 
  = \frac{R^n(\eta_j^2V_j^n-\eta_{j-1}^2V_{j-1}^n)}{\Delta\eta} - 3\eta_{j-1/2}^2R'^nR^n\;.
\end{equation}
This equality is trivial to prove following the $V$-consistency of $\mathcal{F}_j^n$ and the cubic-preserving of $\hat{\mathcal{F}}_j^n$.
Hence we conclude that given all the assumptions as stated, and that $M_j^m+G_j^m=\Theta_j^m=1,\;\forall j$ and $m=n,n+1$, (\ref{eq:fvm_gen_g}) and (\ref{eq:fvm_gen_m}) gives rise to (\ref{eq:fvm_gen_v}).
Thus the method satisfies DGCL.
\end{proof}
In the theorem and its proof, we only considered the radius update condition~(\ref{eq:fvm_t_gcl}).
On the one hand, the theorem only requires $R^n$, $R'^n$, and $R^{n+1}$ to be related by (\ref{eq:fvm_t_gcl}); and it does not pose any restriction on how $R'^n$ is to be computed.
On the other hand, biologically people do not expect any $G$ to flow across the moving boundary, which translates to:
\begin{equation}\label{eq:fvm_thm_noflux_end}
F_{V,N_{\eta}}^{1,n} + F_{R',N_{\eta}}^{1,n} = 0\;,
\end{equation}
and no geometrical flux at $\eta=0$:
\begin{equation}\label{eq:fvm_thm_noflux_orig}
F_{R',0}^{1,n} = 0\;.
\end{equation}
However, the $V$-consistency condition requires that:
\begin{displaymath}
  F_{V,N_{\eta}}^{1,n} = \eta_{N_{\eta}}^2 R^n V_{N_{\eta}}^n = R^n V_{N_{\eta}}^n\;,
\end{displaymath}
and incorporating~(\ref{eq:fvm_thm_noflux_orig}), the cubic-preserving condition requires:
\begin{displaymath}
  F_{R',N_{\eta}}^{1,n} = F_{R',0}^{1,n} + \sum_{j=1}^{N_{\eta}}(F_{R',j}^{1,n}-F_{R',j-1}^{1,n}) = -\Delta\eta\sum_{j=1}^{N_{\eta}}3\eta_{j-1/2}^2R'^nR^n = \left(1-\frac{1}{4}\Delta\eta^2\right)R'^nR^n\;.
\end{displaymath}
Hence the no-flux condition~(\ref{eq:fvm_thm_noflux_end}) indicates $V_{N_{\eta}}^n = \left(1-\frac{1}{4}\Delta\eta^2\right)R'^nR^n$, or equivalently~(\ref{eq:fvm_t_gcl_rp}) as proposed before.

\subsection{A review of the conventional flux functions}
\label{sec:fvm_conv}
We briefly review the conventional finite volume method in the context of (\ref{eq:model_norm_pde}); particularly we consider the spherically symmetric conservation law for a generic species $X$ in spherical coordinates and radial advective velocity $W$:
\begin{equation}\label{eq:fvm_conv_eqn}
\pp{(\eta^2R^2X)}{\tau} + \po{\eta}\left[W(\eta^2R^2X)\right] = 0\;,
\end{equation}
where we omitted any source terms on the right hand side since their approximation is generally independent of the finite volume discretizations.

The conservative variable of (\ref{eq:fvm_conv_eqn}) is $\tilde{X}\eqdef\eta^2R^2X$ rather than $X$, particularly the variable for the interval $[\eta_{j-1},\;\eta_j]$ is $\tilde{X}_{j-1/2} = \eta_{j-1/2}^2R^2X_{j-1/2}$. 
Hence (\ref{eq:fvm_conv_eqn}) is simply the conservative advection equation for $\tilde{X}$ by the velocity $W$:
\begin{equation}\label{eq:fvm_conv_prob}
\pp{\tilde{X}}{\tau} + \po{\eta}\left(W\tilde{X}\right) = 0\;,
\end{equation}
whose finite volume discretization (at the semi-discretized level) reads:
\begin{displaymath}
\frac{d\tilde{X}_{j-1/2}}{d\tau} + \frac{1}{\Delta\eta}\left(F_j-F_{j-1}\right) = 0\;,
\end{displaymath}
where $F_j\approx W\tilde{X}\big|_{\eta=\eta_j}$ and $F_{j-1}\approx W\tilde{W}\big|_{\eta=\eta_{j-1}}$.

If the conventional first-order upwind flux is used (see Section~\ref{sec:rev_fvm}), there is:
\begin{equation}\label{eq:fvm_conv_1st}
  F_j = \mathcal{F}^{\upw}(W_j;\;\tilde{X}_{j-1/2},\;\tilde{X}_{j+1/2})\;,
\end{equation}
where $W_j$ is the nodal velocity at $\eta_j$ and $\mathcal{F}^{\upw}$ is given by (\ref{eq:rev_fvm_upw}).

Extension to higher accuracy is achieved by the limited polynomial reconstruction.
One of the most widely used second-order extension is given by the high-resolution MUSCL method~\cite{BvanLeer:1979a}:
\begin{equation}\label{eq:fvm_conv_2nd}
F_j = \mathcal{F}^{\muscl}(W_j;\;\tilde{X}_{j-3/2},\;\tilde{X}_{j-1/2},\;\tilde{X}_{j+1/2},\;\tilde{X}_{j+3/2},\;\phi_{j-1/2},\;\phi_{j+1/2})\;,
\end{equation}
where $\phi_{j-1/2}$ and $\phi_{j+1/2}$ are slope limiters and the MUSCL flux function is:
\begin{align}
\label{eq:fvm_conv_muscl}
&\ \mathcal{F}^{\muscl}(W_j;\;Z_{j-3/2},\;Z_{j-1/2},\;Z_{j+1/2},\;Z_{j+3/2},\;\phi_{j-1/2},\;\phi_{j+1/2}) \\
\notag
\eqdef&\ \mathcal{F}^{\upw}\left(W_j;\;Z_{j-1/2}+\frac{1}{2}\phi_{j-1/2}\Delta Z_j,\;Z_{j+1/2}-\frac{1}{2}\phi_{j+1/2}\Delta Z_{j+1}\right)\;,
\end{align}
where $\Delta Z_k = Z_{k+1/2}-Z_{k-1/2}$, $k=j, j+1$ and $Z$ is a generic variable that equals $\tilde{X}$ in the case of (\ref{eq:fvm_conv_2nd}).
The slope limiter $\phi_{j-1/2}\in[0,\; 1]$ usually depends on the solutions, but only weakly in the following sense.
Slope limiters are introduced to reduce the magnitude of the slope such that the reconstruction will not create any new local extremum -- a property called monotone preserving.
Hence if setting $\phi_{j-1/2} = c$ satisfies the monotone preserving property for some particular value $c$, so is all slope limiters $\phi_{j-1/2}\in[0,\;c]$.
For this reason the slope limiters are introduced as free (or more precisely semi-free) parameters.

The bounds for slope limiters are nonlinear functions of the discrete solutions, for example, the minmod limiter computes:
\begin{equation}\label{eq:fvm_conv_minmod}
  \phi_{j-1/2} = \varphi^{\minmod}(\Delta Z_{j-1},\;\Delta Z_j) = \left\{\begin{array}{lcl}
      0\;, & & \Delta Z_{j-1}\Delta Z_j\le 0\;, \\ \vspace{-.1in} \\
      \min\left(\frac{\Delta Z_{j-1}}{\Delta Z_j},\; 1\right)\;, & & \Delta Z_{j-1}\Delta Z_j > 0\;.
    \end{array}\right.
\end{equation}
Other widely used limiter functions can be found in~\cite{KWMorton:1987a, RLeVeque:2002a, XZeng:2016a}.

\smallskip

Meanwhile, we show that these conventional flux functions are neither $V$-consistent nor cubic-preserving.
The latter is easy to verify; indeed, the upwind flux is only first-order accurate and the MUSCL flux is at most second-order; whereas cubic-preserving requires a third-order flux for advection equations.

Let us focus on the $V$-consistency and consider, for example, the upwind flux $\mathcal{F}^{\upw}$.
Then $V$-consistency requires that $\mathcal{F}^{\upw}(V_j/R;\;\eta_{j-1/2}^2R^2,\;\eta_{j+1/2}^2R^2) = \eta_j^2RV_j$; however, this equality does not hold either when $V_j\ge0$, in which case according to (\ref{eq:rev_fvm_upw}):
\begin{displaymath}
  \mathcal{F}^{\upw}\left(\frac{V_j}{R};\;\eta_{j-1/2}^2R^2,\;\eta_{j+1/2}^2R^2\right) = \eta_{j-1/2}^2RV_j \ne \eta_j^2RV_j\;;
\end{displaymath}
or when $V_j>0$, in which case: 
\begin{displaymath}
  \mathcal{F}^{\upw}\left(\frac{V_j}{R};\;\eta_{j-1/2}^2R^2,\;\eta_{j+1/2}^2R^2\right) = \eta_{j+1/2}^2RV_j \ne \eta_j^2RV_j\;.
\end{displaymath}

In the next sub-sections, we focus on designing numerical methods such that they lead to a method that satisfies both DGCL and DTCL, following the results of Section~\ref{sec:fvm_thm}.

\subsection{Modified fluxes: Part I}
\label{sec:fvm_vflux}
In this section, we construct $V$-consistent fluxes $\mathcal{F}_j^n$ by modifying the conventional upwind or MUSCL fluxes; in the latter case a synchronized limiter is introduced to ensure the additivity property as required by Theorem~\ref{thm:fvm_gtcl}.
The fluxes $F_{V,j}^{X,n}$ and $F_{u,j}^{X,n}$ will subsequently be constructed accordingly.

\smallskip

To construct a $V$-consistent flux $\mathcal{F}_j^n$, instead of applying the conventional flux functions to the conservative variables $\tilde{X}$, we consider the primitive ones $X$.
Particularly, a first-order upwind method for (\ref{eq:fvm_thm_flux_v}) can be constructed by setting $l=0$, $r=1$, and $\mathcal{P}=\emptyset$:
\begin{equation}\label{eq:fvm_vflux_v_upw}
  \mathcal{F}_j^n(\{X_{j-1/2}^n,\;X_{j+1/2}^n\}) = \eta_j^2(R^n)^2\mathcal{F}^{\upw}\left(\frac{V_j^n}{R^n};\;X_{j-1/2}^n,\;X_{j+1/2}^n\right)\;.
\end{equation}
Because $\mathcal{F}^{\upw}(V_j^n/R^n;\;1,\;1) \equiv V_j^n/R^n$, the flux (\ref{eq:fvm_vflux_v_upw}) is $V$-consistent.

Similarly, extension to higher-order accuracy can make use of the MUSCL flux~(\ref{eq:fvm_conv_muscl}):
\begin{align}
  \label{eq:fvm_vflux_v_muscl}
  &\ \mathcal{F}_j^n(\{X_{j-3/2}^n,\;X_{j-1/2}^n,\;X_{j+1/2}^n,\;X_{j+3/2}^n\},\{\phi_{j-1/2}^{X,n},\;\phi_{j+1/2}^{X,n}\}) \\
  \notag
  =&\ \eta_j^2(R^n)^2\mathcal{F}^{\muscl}\left(\frac{V_j^n}{R^n};\;X_{j-3/2}^n,\;X_{j-1/2}^n,\;X_{j+1/2}^n,\;X_{j+3/2}^n\;,\phi_{j-1/2}^{X,n},\;\phi_{j+1/2}^{X,n}\right)\;,\\
  \label{eq:fvm_vflux_muscl_lim}
  &\ \phi_{k-1/2}^{X,n} = \varphi^{\minmod}(\Delta X_{k-1}^n,\;\Delta X_k^n)\;,\quad k=j,j+1\;.
\end{align}
Here $\mathcal{P} = \{\phi_{j-1/2}^{X,n},\;\phi_{j+1/2}^{X,n}\}$ and the minmod limiter can be replaced by any other limiter of choice.
It is not difficult to verify that if $X_{k-1/2}^n\equiv1$, the MUSCL flux $\mathcal{F}^{\muscl}$ gives rise to $V_j^n/R^n$ regardless of the values of the limiters; hence the flux function~(\ref{eq:fvm_vflux_v_muscl}) is $V$-consistent, no matter what limiter we will choose.

\medskip

Next the additivity of these fluxes is considered, which is essentially requiring that the fluxes are linear in the inputs $X_{j-1/2}$.
Hence the upwind fluxes are by nature additive; for example let us consider $\mathcal{F}_j^n$ given by (\ref{eq:fvm_vflux_v_upw}) and suppose $V_j^n\ge0$, then:
\begin{align*}
  &\ F_{V,j}^{G,n} + F_{V,j}^{M,n} = \mathcal{F}_j^n(\{G_{j-1/2}^n,\;G_{j+1/2}^n\}) + \mathcal{F}_j^n(\{M_{j-1/2}^n,\;M_{j+1/2}^n\}) \\
  =&\ \eta_j^2(R^n)^2\mathcal{F}^{\upw}\left(\frac{V_j^n}{R^n},\;G_{j-1/2}^n,\;G_{j+1/2}^n\right) + 
    \eta_j^2(R^n)^2\mathcal{F}^{\upw}\left(\frac{V_j^n}{R^n},\;M_{j-1/2}^n,\;M_{j+1/2}^n\right) \\
    =&\ \eta_j^2(R^n)^2\cdot\frac{V_j^n}{R^n}G_{j-1/2}^n + \eta_j^2(R^n)^2\cdot\frac{V_j^n}{R^n}M_{j-1/2}^n = \eta_j^2(R^n)^2\cdot\frac{V_j^n}{R^n}\Theta_{j-1/2}^n \\
    =&\ \eta_j^2(R^n)^2\mathcal{F}^{\upw}\left(\frac{V_j^n}{R^n},\;\Theta_{j-1/2}^n,\;\Theta_{j+1/2}^n\right) = F_{V,j}^{\Theta,n}\;.
\end{align*}
The argument for the case $V_j^n<0$ is similar; hence $\mathcal{F}_j^n$ is additive.

Extension to the MUSCL-based fluxes~(\ref{eq:fvm_vflux_v_muscl}) is not straightforward, as the limiter function is generally nonlinear.
Following the discussion below Equation~(\ref{eq:fvm_conv_muscl}), we can circumvent this difficulty by synchronizing the limiters for $G$ and $M$, that is:
\begin{align}
  \label{eq:fvm_vflux_vg_muscl_add}
  F_{V,j}^{G,n} &= \eta_j^2(R^n)^2\mathcal{F}^{\muscl}\left(\frac{V_j^n}{R^n},\;G_{j-3/2}^n,\;G_{j-1/2}^n,\;G_{j+1/2}^n,\;G_{j+3/2}^n,\;\phi_{j-1/2}^n,\;\phi_{j+1/2}^n\right) \\
  \label{eq:fvm_vflux_vn_muscl_add}
  F_{V,j}^{M,n} &= \eta_j^2(R^n)^2\mathcal{F}^{\muscl}\left(\frac{V_j^n}{R^n},\;M_{j-3/2}^n,\;M_{j-1/2}^n,\;M_{j+1/2}^n,\;M_{j+3/2}^n,\;\phi_{j-1/2}^n,\;\phi_{j+1/2}^n\right) \\
  \label{eq:fvm_vflux_v_muscl_lim}
  &\textrm{ where }\  
  \phi_{k-1/2}^n = \min\left(\phi_{k-1/2}^{G,n},\;\phi_{k-1/2}^{M,n}\right),\quad k=j,j+1\;,
\end{align}
here $\phi_{k-1/2}^{G,n}$ and $\phi_{k-1/2}^{M,n}$ are obtained by applying (\ref{eq:fvm_vflux_muscl_lim}) to $X=G$ and $X=M$, respectively.
Note that the same limiters are used to compute the two fluxes.
To show the additivity, we assume again without loss of generality that $V_j^n\ge0$, then:
\begin{align*}
  &F_{V,j}^{G,n} = \eta_j^2(R^n)^2\cdot\frac{V_j^n}{R^n}\left(G_{j-1/2}^n+\frac{1}{2}\phi_{j-1/2}^n(G_{j+1/2}^n-G_{j-1/2}^n)\right),\; \\
  &F_{V,j}^{M,n} = \eta_j^2(R^n)^2\cdot\frac{V_j^n}{R^n}\left(M_{j-1/2}^n+\frac{1}{2}\phi_{j-1/2}^n(M_{j+1/2}^n-M_{j-1/2}^n)\right) \\
  \Rightarrow\ &F_{V,j}^{G,n} + F_{V,j}^{M,n} = \eta_j^2(R^n)^2\cdot\frac{V_j^n}{R^n}\left(\Theta_{j-1/2}^n+\frac{1}{2}\phi_{j-1/2}^n(\Theta_{j+1/2}^n-\Theta_{j-1/2}^n)\right) \\
  &\qquad = \eta_j^2(R^n)^2\mathcal{F}^{\muscl}\left(\frac{V_j^n}{R^n};\;\Theta_{j-3/2}^n,\;\Theta_{j-1/2}^n,\;\Theta_{j+1/2}^n,\;\Theta_{j+3/2}^n,\;\phi_{j-1/2}^n,\;\phi_{j+1/2}^n\right)\;.
\end{align*}
The latest flux is in general not a monotone flux for $\Theta$, since the selected limiter may be too large.
Nevertheless, this is not an issue since $\Theta$ is not our numerical solution; and as long as $G$ and $M$ are computed using monotone fluxes, we won't run into stability issues.
Nevertheless, if one wishes to ensure monotone flux for $\Theta$ as well, all that needs to be done is to compute the limiter $\phi_{k-1/2}^{\Theta,n}$ by applying (\ref{eq:fvm_vflux_muscl_lim}) to $X=G+N$, and include this $\phi_{k-1/2}^{\Theta,n}$ in the minimum of the right hand side of (\ref{eq:fvm_vflux_v_muscl_lim}).

\medskip

Finally, we compute the $u$-fluxes for $M$ by:
\begin{equation}\label{eq:fvm_vflux_u_upw}
  F_{u,j}^{M,n} = \eta_j^2(R^n)^2\mathcal{F}^{\upw}\left(\frac{u_j^n}{R^n};\;M_{j-1/2}^n,\;M_{j+1/2}^n\right)
\end{equation}
for first-order accuracy or:
\begin{equation}\label{eq:fvm_vflux_u_muscl}
  F_{u,j}^{M,n} = \eta_j^2(R^n)^2\mathcal{F}^{\muscl}\left(\frac{u_j^n}{R^n};\;M_{j-3/2}^n,\;M_{j-1/2}^n,\;M_{j+1/2}^n,\;M_{j+3/2}^n,\;\phi_{j-1/2}^n,\;\phi_{j+1/2}^n\right)
\end{equation}
for second-order accuracy, where the limiters are the same ones computed by (\ref{eq:fvm_vflux_v_muscl_lim}).

\subsection{Modified fluxes: Part II}
\label{sec:fvm_rflux}
To construct a cubic-preserving flux $F_{R',j}^{X,n}$, however, we cannot follow the same strategy as in the previous section. 
Indeed, if this flux is defined as:
\begin{displaymath}
  F_{R',j}^{X,n} = \eta_j^2(R^n)^2\mathcal{F}^{\upw}\left(-\frac{\eta_jR'^n}{R^n},\;X_{j-1/2}^n,\;X_{j+1/2}^n\right)\;,
\end{displaymath}
supposing $R'^n\ge0$ and setting $X_{k-1/2}^n\equiv1$ we have:
\begin{align*}
  &\ F_{R',j-1}^{1,n} = -\eta_{j-1}^3R'^nR^n\;,\quad
  F_{R',j}^{1,n} = -\eta_j^3R'^nR^n \\
  \Rightarrow&\ \frac{1}{\Delta\eta}\left(F_{R',j}^{1,n}-F_{R',j-1}^{1,n}\right) = - \left(3\eta_{j-1/2}^2+\frac{1}{4}\Delta\eta^2\right)R'^nR^n\;,
\end{align*}
which is different from $-3\eta_{j-1/2}^2R'^nR^n$, as required by the cubic-preserving property.

To proceed, we recognize that a higher-order and nonlinearly stable flux can be obtained by a polynomial reconstruction of the solutions on each interval such that the total variation does not increase, and apply the upwind flux to the two reconstructed values on both sides of the interval face.
In the MUSCL scheme, the reconstruction is achieved by limiting the slope of a linear function that preserves the interval average; in this section, we adopt the average-preserving and monotone cubic reconstruction of the Piecewise Parabolic Method (PPM)~\cite{PColella:1984a}, but construct the flux differently.

Let $X$ be a generic variable as before, the PPM reconstruction in normalized coordinates $\xi=(\eta-\eta_{j-1})/\Delta\eta$ on the interval $[\eta_{j-1},\;\eta_j]$ reads:
\begin{align}
\label{eq:fvm_rflux_ppm_rec}
  X(\xi) &= X_{j-1/2,-} + \xi\left(X_{j-1/2,+}-X_{j-1/2,-}+X_{6,j-1/2}(1-\xi)\right)\;,\\
\label{eq:fvm_rflux_ppm_x6}
  &X_{6,j-1/2} \eqdef 6X_{j-1/2}-3(X_{j-1/2,-}+X_{j-1/2,+})\;.
\end{align}
Here $X_{j-1/2,-}$ and $X_{j-1/2,+}$ are the two end values that are defined as:
\begin{align}
  \label{eq:fvm_rflux_rec_l}
  &X_{j-1/2,-} = X_{j-1/2}+\phi_{j-1/2,-}^X(X_{j-1}-X_{j-1/2})\;, \\ %\quad X_{j-1} \eqdef \frac{7}{12}(X_{j-3/2}+X_{j-1/2})-\frac{1}{12}(X_{j-5/2}+X_{j+1/2})\;, \\
  \label{eq:fvm_rflux_rec_r}
  &X_{j-1/2,+} = X_{j-1/2}+\phi_{j-1/2,+}^X(X_{j+1}-X_{j-1/2})\;, \\ %\quad X_{j+1/2} \eqdef \frac{7}{12}(X_j+X_{j+1})-\frac{1}{12}(X_{j-1}+X_{j+2})\;.
  \label{eq:fvm_rflux_rec_f} 
  &\quad\textrm{ where }\ X_k \eqdef \frac{7}{12}(X_{k-1/2}+X_{k+1/2})-\frac{1}{12}(X_{k-3/2}+X_{k+3/2})\;,\ k=j-1,j\;.
\end{align}
The value $X_k,\;k=j-1,j$ are third-order reconstructions of the face values when the data is smooth; and the two limiters $\phi_{j-1/2,\pm}^X$ are decides as follows:
\begin{enumerate}[(a)]
  \item If $(X_j-X_{j-1/2})(X_{j-1}-X_{j-1/2})\ge0$, we have a local extrema and set:
    \begin{equation}\label{eq:fvm_rflux_lim_loc}
      \phi_{j-1/2,-}^X=\phi_{j-1/2,+}^X=0\;.
    \end{equation}
  \item If (a) is not true, and if $\abs{X_j-X_{j-1/2}}>2\abs{X_{j-1}-X_{j-1/2}}$ or $\abs{X_{j-1}-X_{j-1/2}}>2\abs{X_j-X_{j-1/2}}$, the corresponding reconstructed profile is not monotone on the interval $[\eta_{j-1},\;\eta_j]$ and we compute:
    \begin{equation}\label{eq:fvm_rflux_lim_mono_p}
      \phi_{j-1/2,+}^X = -\frac{2(X_{j-1}-X_{j-1/2})}{X_j-X_{j-1/2}}
    \end{equation}
    in the former case, and
    \begin{equation}\label{eq:fvm_rflux_lim_mono_m}
      \phi_{j-1/2,-}^X = -\frac{2(X_j-X_{j-1/2})}{X_{j-1}-X_{j-1/2}}
    \end{equation}
    in the latter case.
  \item Otherwise, set the remaining limiter, which is $\phi_{j,-}^X$, or $\phi_{j,+}^X$, or both, to one.
\end{enumerate}
Finally, denoting $\overline{X}_{j-1/2}^n=\eta_{j-1/2}^3X_{j-1/2}^n$ we compute the flux $\hat{\mathcal{F}}_j^n$ of (\ref{eq:fvm_thm_flux_r}) as:
\begin{align}
  \label{eq:fvm_rflux_gen}
  &\ \hat{\mathcal{F}}_j^n\left(\{X_{j-5/2}^n,\;X_{j-3/2}^n,\;X_{j-1/2}^n,\;X_{j+1/2}^n,\;X_{j+3/2}^n,\;X_{j+5/2}^n\},\;\{\phi_{j-1/2,+}^{X,n},\;\phi_{j+1/2,-}^{X,n}\}\right) \\
  \notag
  \eqdef&\ \mathcal{F}^{\upw}\left(-R'^nR^n;\;\overline{X}_{j-1/2,+}^n,\;\overline{X}_{j+1/2,-}^n\right)\;.
\end{align}
Here $\overline{X}_{k-1/2,\pm}^n$ are computed according to (\ref{eq:fvm_rflux_rec_l}) and (\ref{eq:fvm_rflux_rec_r}) with data $\overline{X}_{k-1/2}^n$; and the limiters $\phi_{j-1/2,\pm}^{X,n}$ are computed as $\phi_{j-1/2,\pm}^{\overline{X}}$ according to (a-c) given previously.

\begin{theorem}\label{thm:fvm_rflux_cub}
  The flux $F_{R',j}^{X,n}=\hat{\mathcal{F}}^n_j$, which is given by (\ref{eq:fvm_rflux_gen}), satisfies (\ref{eq:fvm_cubic}), i.e.,:
  \begin{displaymath}
    \frac{1}{\Delta\eta}\left(F_{R',j}^{1,n}-F_{R',j-1}^{1,n}\right)=-3\eta_{j-1/2}^2R'^nR^n
  \end{displaymath} 
  for $j\ge 4$.
\end{theorem}
\begin{proof}
  We suppress the superscript $n$ for simplicity and let $X_{k-1/2}\equiv1,\;\forall k$, then the reconstructed values $\overline{X}_k$, $k\ge2$ are given by:
  \begin{equation}\label{eq:fvm_rflux_cub_rec}
    \overline{X}_k = \frac{7}{12}\left(\eta_{k-1/2}^3+\eta_{k+1/2}^3\right) - \frac{1}{12}\left(\eta_{k-3/2}^3+\eta_{k+3/2}^3\right)
    = \eta_k^3-\frac{1}{4}\eta_k\Delta\eta^2\;.
  \end{equation}
  Now we compute the limiters $\phi^{X}_{k-1/2,\pm}$ for $k\ge3$.
  First of all:
  \begin{align*}
    \overline{X}_k-\overline{X}_{k-1/2} &= \eta_k^3-\frac{1}{4}\eta_k\Delta\eta^2-\eta_{k-1/2}^3 = \frac{3}{2}\eta_{k-1/2}^2\Delta\eta + \frac{1}{2}\eta_{k-1/2}\Delta\eta^2 > 0\;. \\
    \overline{X}_{k-1}-\overline{X}_{k-1/2} &= \eta_{k-1}^3-\frac{1}{4}\eta_{k-1}\Delta\eta^2-\eta_{k-1/2}^3 = -\frac{3}{2}\eta_{k-1/2}^2\Delta\eta + \frac{1}{2}\eta_{k-1/2}\Delta\eta^2 < 0 \;,
  \end{align*}
  thus the condition in (a) does not hold.
  Continuing to check the conditions in (b), we have:
  \begin{align*}
    2\abs{\overline{X}_k-\overline{X}_{k-1/2}} - \abs{\overline{X}_{k-1}-\overline{X}_{k-1/2}}
    = \frac{3}{2}\eta_{k-1/2}^2\Delta\eta + \frac{3}{2}\eta_{k-1/2}\Delta\eta^2 > 0 \;, \\
    2\abs{\overline{X}_{k-1}-\overline{X}_{k-1/2}} - \abs{\overline{X}_k-\overline{X}_{k-1/2}}
    = \frac{3}{2}\eta_{k-1/2}^2\Delta\eta - \frac{3}{2}\eta_{k-1/2}\Delta\eta^2 > 0 \;, 
  \end{align*}
  hence neither condition in (b) is true.
  Thus we conclude that $\phi_{k-1/2,\pm}^{X}=1$ for all $k\ge3$; consequently $\overline{X}_{k-1/2,-} = \overline{X}_{k-1}$ and $\overline{X}_{k-1/2,+} = \overline{X}_k$.

  Finally we can calculate the flux $F_{R',j}^{1,n}$ for all $j\ge3$.
  Because $\overline{X}_{j+1/2,-}=\overline{X}_j=\overline{X}_{j-1/2,+}$, regardless of the sign of $R'^n$, there is:
  \begin{equation}\label{eq:fvm_rflux_cub_flux}
    F_{R',j}^{1,n} = -R'^nR^n\overline{X}_j = -R'^nR^n\left(\eta_j^3-\frac{1}{4}\eta_j\Delta\eta^2\right)\;,
  \end{equation}
  thus for $j,\;j-1\ge3$:
  \begin{align*}
    &\ \frac{1}{\Delta\eta}\left(F_{R',j}^{1,n}-F_{R',j-1}^{1,n}\right) = -\frac{R'^nR^n}{\Delta\eta}\left[
      \left(\eta_j^3-\frac{1}{4}\eta_j\Delta\eta^2\right) - \left(\eta_{j-1}^3-\frac{1}{4}\eta_{j-1}\Delta\eta^2\right)
    \right] \\
    = &\ -\frac{R'^nR^n}{\Delta\eta}\left[\left((\eta_{j-1/2}+\frac{1}{2}\Delta\eta)^3-(\eta_{j-1/2}-\frac{1}{2}\Delta\eta)^3\right) - \frac{1}{4}(\eta_j-\eta_{j-1})\Delta\eta^2\right] \\
    = &\ -\frac{R'^nR^n}{\Delta\eta}\left(3\eta_{j-1/2}^2\Delta\eta + \frac{1}{4}\Delta\eta^3-\frac{1}{4}\Delta\eta^3\right) = -3\eta_{j-1/2}^2R'^nR^n\;.
  \end{align*}
  This concludes that the constructed flux is cubic-preserving.
\end{proof}

Theorem~\ref{thm:fvm_rflux_cub} addresses the cubic-preserving for intervals far away from the boundaries; next we consider boundary intervals and focus on those near the origin first.
More specifically, we need to consider the cubic-preserving on the first three intervals $[0,\; \Delta\eta]$, $[\Delta\eta,\; 2\Delta\eta]$, and $[2\Delta\eta,\; 3\Delta\eta]$.
Following a similar procedure in the preceding proof, cubic-preserving on these three intervals amounts to:
\begin{subequations}\label{eq:fvm_rflux_bc}
\begin{align}
  \label{eq:fvm_rflux_bc_2}
  F_{R',2}^{1,n} &= -R'^nR^n\left(\eta_2^3-\frac{1}{4}\eta_2\Delta\eta^2\right)\;,\\
  \label{eq:fvm_rflux_bc_1}
  F_{R',1}^{1,n} &= -R'^nR^n\left(\eta_1^3-\frac{1}{4}\eta_1\Delta\eta^2\right)\;,\\
  \label{eq:fvm_rflux_bc_0}
  F_{R',0}^{1,n} &= 0\;.
\end{align}
\end{subequations}
Note that (\ref{eq:fvm_rflux_bc_0}) is enforced automatically as the boundary condition at the origin; hence we focus on the first two, which are guaranteed if:
\begin{displaymath}
  \overline{X}_2 = \frac{15}{2}\Delta\eta^3\;,\quad
  \overline{X}_1 = \frac{3}{4}\Delta\eta^3\;,
\end{displaymath}
and that $\phi_{3/2,\pm}^{X,n}=\phi_{1/2,\pm}^{X,n}=1$ when all $X$'s are $1$.
The required $\overline{X}_2$ is precisely given by the formula in (\ref{eq:fvm_rflux_rec_l}) or (\ref{eq:fvm_rflux_rec_l}); thus we just need to look at $\overline{X}_1$ and the limiters.
To this end, utilizing the knowledge that when $\overline{X}$ is defined as $\eta^3X$, we expect $\overline{X}=0$ at $\eta=0$ and define $\overline{X}_1$ by the modified formula:
\begin{equation}\label{eq:fvm_rflux_rec_1}
  \overline{X}_1 = \frac{7}{12}(\overline{X}_{1/2}+\overline{X}_{3/2}) - \frac{1}{12}(\overline{X}_{5/2}-\overline{X}_{1/2})\;.
\end{equation}
According to this definition, $\phi_{3/2,\pm}^{X,n}=1$ when $X\equiv1$ as desired; but on the first interval we have $\phi_{1/2,-}^{X,n}=1$ and $\phi_{1/2,+}^{X,n}=2/5$ following the criterion before, particularly (\ref{eq:fvm_rflux_lim_mono_p}).
Hence we need to relax it.
A simple fix could be derived by the following reconstruction on the interval $[0,\; \Delta\eta]$ utilizing the knowledge that $\overline{X}\sim\eta^3$ near $\eta=0$:
\begin{align*}
  \overline{X}(\xi) &= \xi^3(\overline{X}_{1/2,+}+\overline{X}_{6,1/2}(1-\xi^2))\;,\quad \overline{X}_{6,1/2}\eqdef 12 \overline{X}_{1/2}-3\overline{X}_{1/2,+}\;,\\
  \overline{X}_{1/2,+} &= \overline{X}_{1/2} + \phi_{1/2,+}^{\overline{X}}(\overline{X}_1-\overline{X}_{1/2})\;.
\end{align*}
As before $\overline{X}_{6,1/2}$ is defined such that the mean of $\overline{X}(\xi)$ is $\overline{X}_{1/2}$ for any right end value $\overline{X}_{1/2,+}$.
Because $\overline{X}_{1/2,-}\equiv \overline{X}_0\equiv0$, there is no way to design the limiter such that $\overline{X}(\xi)$ is monotone if $\overline{X}_{1/2}\ne0$ and $\overline{X}_1$ is too close to $\overline{X}_{1/2}$.
To this end, we design $\phi_{1/2,+}^{\overline{X}}$ as follows instead of the generic construction for other intervals:
\begin{enumerate}[(a')]
  \item If $\overline{X}_{1/2}\overline{X}_1\le0$ or if $3\abs{\overline{X}_1}\le 8\abs{\overline{X}_{1/2}}$, we set $\phi_{1/2,+}^{\overline{X}}=0$.
  \item If (a') is not true, and if $\abs{\overline{X}_1}>6\abs{\overline{X}_{1/2}}$, we define
  \begin{equation}\label{eq:fvm_rflux_lim_mono_1}
    \phi_{1/2,+}^{\overline{X}} =
    \frac{5}{\overline{X}_1/\overline{X}_{1/2}-1}\;. 
  \end{equation}
\item Otherwise, $\phi_{1/2,+}^{\overline{X}}=1$.
\end{enumerate}
Following this modified definition, when $\overline{X}_{1/2} = \eta_{1/2}^3 = \frac{1}{8}\Delta\eta^3$ and $\overline{X}_1=\frac{3}{4}\Delta\eta^3$, the limiter is precisely $\phi_{1/2,+}^{\overline{X},n}=1$.

\medskip

Finally we ensure the additivity of the fluxes $F_{R',j}^{G,n}$ and $F_{R',j}^{M,n}$ by constructing the limiter $\phi_{j,\pm}^n$ for both species as follows.
For $j\ge2$:
\begin{enumerate}[(A)]
  \item If $(\overline{M}_j-\overline{M}_{j-1/2})(\overline{M}_{j-1}-\overline{M}_{j-1/2})\ge0$ or $(\overline{G}_j-\overline{G}_{j-1/2})(\overline{G}_{j-1}-\overline{G}_{j-1/2})\ge0$, we set:
\begin{equation}\label{eq:fvm_rflux_slim_loc}
  \phi_{j-1/2,-}^n=\phi_{j-1/2,+}^n = 0\;.
\end{equation}
\item Otherwise, we compute:
\begin{displaymath}
\alpha_1 = \min\left(
  \frac{2\abs{\overline{M}_{j-1}-\overline{M}_{j-1/2}}}{\abs{\overline{M}_j-\overline{M}_{j-1/2}}},\;
  \frac{2\abs{\overline{G}_{j-1}-\overline{G}_{j-1/2}}}{\abs{\overline{G}_j-\overline{G}_{j-1/2}}}%,\;
    %\frac{2\abs{\overline{\Theta}_{j-1/2}-\overline{\Theta}_j}}{\abs{\overline{\Theta}_{j+1/2}-\overline{\Theta}_j}}
    \right)\;,
\end{displaymath}
and
\begin{displaymath}
\alpha_2 = \max\left(
  \frac{\abs{\overline{M}_{j-1}-\overline{M}_{j-1/2}}}{2\abs{\overline{M}_j-\overline{M}_{j-1/2}}},\;
  \frac{\abs{\overline{G}_{j-1}-\overline{G}_{j-1/2}}}{2\abs{\overline{G}_j-\overline{G}_{j-1/2}}}%,\;
    %\frac{\abs{\overline{\Theta}_{j-1/2}-\overline{\Theta}_j}}{2\abs{\overline{\Theta}_{j+1/2}-\overline{\Theta}_j}}
    \right)\;.
\end{displaymath}
  \begin{enumerate}[(B1)]
  \item If $\alpha_2>\alpha_1$, use (\ref{eq:fvm_rflux_slim_loc}).
  \item Otherwise if $\alpha_1<1$, set:
  \begin{equation}\label{eq:fvm_rflux_slim_mono_p}
    \phi_{j-1/2,-}^n = 1\;,\quad\phi_{j-1/2,+}^n=\alpha_1\;;
  \end{equation}
  and if $\alpha_2>1$, set:
  \begin{equation}\label{eq:fvm_rflux_slim_mono_m}
    \phi_{j-1/2,-}^n = \alpha_2^{-1}\;,\quad\phi_{j-1/2,+}^n=1\;.
  \end{equation}
  \end{enumerate}
\end{enumerate}
Similarly for the first interval, we have:
\begin{enumerate}[(A')]
  \item If $\overline{X}_{1/2}\overline{X}_1\le0$ or $3\abs{\overline{X}_1}\le8\abs{\overline{X}_{1/2}}$ is true for either $X=G$ or $X=M$, we set $\phi_{1/2,+}^n=0$.
\item Otherwise set:
\begin{equation}\label{eq:fvm_rflux_slim_1}
  \phi_{1/2,+}^n = \min\left(1,\;\frac{5}{\overline{G}_1/\overline{G}_{1/2}-1},\;\frac{5}{\overline{M}_1/\overline{M}_{1/2}-1}\right)\;.%,\;\frac{5}{\overline{\Theta}_{1+2}/\overline{\Theta}_1-1}\right)\;.
\end{equation}
\end{enumerate}

\subsection{Modified fluxes: Part III}
\label{sec:fvm_bflux}
In the last part of the modified fluxes, we consider the intervals near the right boundary $\eta=1$.
Particularly, these are the intervals whose flux calculation requires solutions beyond the computational domain.

\medskip

At $\eta_{N_{\eta}}$, we notice that the velocity $V/R-\eta R'/R\equiv0$ due to (\ref{eq:model_norm_pde_r}). 
Hence the numerical fluxes need to satisfy the following identity:
\begin{equation}\label{eq:fvm_bflux_m0}
  F_{V,N_{\eta}}^{X,n}+F_{R',N_{\eta}}^{X,n} = 0\;,
\end{equation}
where $X=G$ or $X=M$.
In fact, practically we set both fluxes $F_{V,N_{\eta}}^{X,n}$ and $F_{R',N_{\eta}}^{X,n}$ to zero for convenience.

The remaining flux $F_{u,N_\eta}^{M,n}$ is computed as:
\begin{equation}\label{eq:fvm_bflux_u_m0}
  F_{u,N_\eta}^{M,n} = \eta_{N_\eta}^2(R^n)^2\mathcal{F}^{\upw}\left(\frac{u_{N_\eta}^n}{R^n};\;M_{N_\eta-1/2}^n,\;M_{\bc}(t^n)\right)\;,
\end{equation}
no matter which flux function is used for interior nodes.

\medskip

At $\eta_{N_{\eta}-1}$, calculating $F_{V,N_{\eta}-1}^{X,n}$ using the upwind flux~(\ref{eq:fvm_vflux_v_upw}) does not require any intervals beyond $\eta=1$ hence no modification is needed.
The MUSCL flux~(\ref{eq:fvm_vflux_v_muscl}), however, requires the phantom variable $X_{N_{\eta}+1/2}$.
To this end, we avoid the linear reconstruction on the last interval and modify the numerical flux as:
\begin{equation}\label{eq:fvm_bflux_v_m1_muscl_add}
  F_{V,N_{\eta}-1}^{X,n} = \eta_{N_{\eta}-1}^2(R^n)^2\mathcal{F}^{\upw}\left(\frac{V_{N_{\eta}-1}^n}{R^n};\;X_{N_\eta-3/2}^n+\frac{1}{2}\phi_{N_\eta-3/2}^n\Delta X_{N_\eta-1}^n,\;X_{N_\eta-1/2}^n\right)\;,
\end{equation}
where $X=G$ or $X=M$ and $\phi_{N_\eta-3/2}^n$ is computed according to (\ref{eq:fvm_vflux_v_muscl_lim}).
Similarly, the flux $F_{u,N_\eta-1}^{M,n}$ is given by:
\begin{equation}\label{eq:fvm_bflux_u_m1_muscl_add}
  F_{u,N_{\eta}-1}^{M,n} = \eta_{N_{\eta}-1}^2(R^n)^2\mathcal{F}^{\upw}\left(\frac{u_{N_{\eta}-1}^n}{R^n};\;M_{N_\eta-3/2}^n+\frac{1}{2}\phi_{N_\eta-3/2}^n\Delta M_{N_\eta-1}^n,\;M_{N_\eta-1/2}^n\right)\;,
\end{equation}
where $\phi_{N_\eta-3/2}^n$ is the same limiter used in (\ref{eq:fvm_bflux_v_m1_muscl_add}).

In order to compute $F_{R',N_\eta-1}^{X,n}$ using the modified PPM method in Section~\ref{sec:fvm_rflux}, we need to use a biased stencil to interpolate $\overline{X}$ to obtain:
\begin{equation}\label{eq:fvm_bflux_ppm_rec_m1}
  \overline{X}_{N_{\eta}-1} = \frac{1}{12}(3\overline{X}_{N_{\eta}-1/2}+13\overline{X}_{N_{\eta}-3/2}-5\overline{X}_{N_{\eta}-5/2}+\overline{X}_{N_{\eta}-7/2})\;.
\end{equation}
Using this definition, when $X\equiv1$ we have $\overline{X}_{N_\eta-1}=\eta_{N_\eta-1}^3-\eta_{N_\eta-1}\Delta\eta^2/4$, c.f.~(\ref{eq:fvm_rflux_cub_rec}).
Following the same procedure in the proof of Theorem~\ref{thm:fvm_rflux_cub}, we have $\phi_{N_\eta-3/2,\pm}^{X}=1$ and thusly the cubic-preserving property holds for the interval $[\eta_{N_\eta-3},\;\eta_{N_\eta-2}]$.

Finally, to ensure cubic-preserving on the next interval $[\eta_{N_\eta-2},\;\eta_{N_\eta-1}]$, all that needs to be done is to design $\overline{X}_{N_{\eta}}$ properly such that the limiter $\phi_{N_\eta-1/2,-}^X$ takes the value $1$ when $X\equiv1$.
This can be achieved by the extrapolating formula:
\begin{equation}\label{eq:fvm_bflux_ppm_rec_r}
  \overline{X}_{N_{\eta}} = \frac{1}{12}(25\overline{X}_{N_{\eta}-1/2}-23\overline{X}_{N_{\eta}-3/2}+13\overline{X}_{N_{\eta}-5/2}-3\overline{X}_{N_{\eta}-7/2})\;.
\end{equation}
Note that $\overline{X}_{N_\eta}$ is only used for computing the limiter $\phi_{N_\eta-1/2,-}^X$.

\medskip

\subsection{Higher-order accuracy in time}
\label{sec:fvm_time}
The preceding sections fully specify the discretization with first-order and second-order accuracy in space, and first-order accuracy in time.
Here the temporal integration is achieved by the forward Euler (FE) method; hence extension to higher-order accuracy in time can be easily achieved by using Total Variation Diminishing (TVD) Runge-Kutta methods~\cite{CWShu:1988a, SGottlieb:1998a}.
In particular, we consider the second-order TVD Runge-Kutta, denoted by TVD-RK2 in the rest of the paper, to match the spatial order of accuracy when the MUSCL fluxes are used. 
For this purpose, we denote the solution as $\mathcal{S} = \{G,\;M,\;V,\;R\}$ and let the method proposed before with FE integrator be summarized as:
\begin{equation}\label{eq:fvm_time_fe}
\mathcal{S}^{n+1} = \mathcal{M}_{\Delta\tau^n}(\mathcal{S}^n)\;,
\end{equation}
here the subscript $\Delta\tau^n$ denotes the time-step size; then the method using TVD RK2 reads:
\begin{align}
\notag
&\mathcal{S}^{(1)} = \mathcal{M}_{\Delta\tau^n}(\mathcal{S}^n)\;, \\
\label{eq:fvm_time_rk2}
&\mathcal{S}^{(2)} = \mathcal{M}_{\Delta\tau^n}(\mathcal{S}^{(1)})\;, \\
\notag
&\mathcal{S}^{n+1} = \frac{1}{2}\left(\mathcal{S}^n+\mathcal{S}^{(2)}\right)\;.
\end{align}
Here the average is defined in the natural way, for example, $G_j^{n+1} = (G_j^n+G_j^{(2)})/2$ and $R^{n+1}=(R^n+R^{(2)})/2$, etc.

\medskip

Before concluding this section, we have three remarks.

{\bf Remark 1}.
The computation of the time step size $\Delta\tau^n$ is according to the classical Courant condition for linear stability.
However, the original formula needs to be adjusted since segregated advection velocities are used in the enhanced methods.
In the case of the upwind flux combining with first-order explicit time-integrator, the method remains conditionally stable and the analysis as well as the formula for the corresponding Courant condition are provided in Appendix~\ref{app:split}.

{\bf Remark 2}.
The methodology extends naturally to the original tumor growth problem.
In particular, the fluxes for $G$, $N$, $M$ are segregated similarly; and in extension to the MUSCL flux or PPM flux, the limiter synchronization needs to take into account all species.

{\bf Remark 3}.
When the tumor growth model~(\ref{eq:rev_eqn}) is considered, one typically requires an implicit time-integrator for updating the chemoattractant concentration $A$ to avoid tiny time step sizes.
It is then very natural to ask how the enhancement can be applied with implicit solvers.
We briefly address this issue in Appendix~\ref{app:imp} in the case of the backward-Euler method and the class of Diagonally Implicit Runge-Kutta (DIRK) methods, and provide brief numerical results for comparison.

%\section{Application to the tumor growth model}
%\label{sec:tumor}
%\input{section_tumor.tex}

\section{Numerical Assessment}
\label{sec:num}
We assess the performance of the numerical methods of Section~\ref{sec:fvm} using various benchmark tests.
First, we consider the model problem~(\ref{eq:model_norm_pde}) and verify the DTCL and DGCL properties of the proposed method.

\subsection{The model problem}
\label{sec:num_mod}
To assess the numerical performance, we consider a series of tests that are characterized by spatially constant solutions, ``prescribed'' growth, and non-monotonic radius change, respectively.
For all the tests, we set the initial radius as $R(0)=1.0$.
The purpose of these tests is to assess the ability of the enhanced methods to satisfy the totality conservation law and geometric conservation law discretely.
Hence for each test below, we compare the numerical results that are obtained by using four different methods:
\begin{itemize}
\item The conventional finite volume method with upwind fluxes as described in Section~\ref{sec:rev_fvm}.
\item The conventional finite volume method with MUSCL fluxes with the TVD-RK2 time-integrator as described in Section~\ref{sec:fvm_time}. 
\item The enhanced finite volume method with upwind $V$ and $u$ fluxes and cubic-preserving $R'$ fluxes, as required by Theorem~\ref{thm:fvm_gtcl}.
\item The enhanced finite volume method with MUSCL $V$ and $u$ fluxes and cubic-preserving $R'$ fluxes, as required by Theorem~\ref{thm:fvm_gtcl}, and TVD-RK2 time-integrator as described in Section~\ref{sec:fvm_time}.
\end{itemize}
The first two methods are denoted by ``Conv. Upwind'' and ``Conv. MUSCL'', respectively; and the two enhanced ones are denoted by ``Enhc. Upwind'' and ``Enhc. MUSCL'', respectively, in the subsequent tests.
For all the methods, the fixed Courant number $\alpha_{\cfl}=0.8$ is used.

\subsubsection{Test 1: Spatially constant solutions -- single species}
\label{sec:num_mod_1}
In the first two tests, the velocity field $u$ is manufactured such that (\ref{eq:model_pde}) allows solutions that are independent of the spatial coordinate.
Indeed, assuming $f\equiv h\equiv 0$ and $G(r,t) = G(t)$, (\ref{eq:model_pde_g}) indicates:
\begin{displaymath}
G'(t) + \frac{1}{r^2}\po{r}(r^2V(r,t))G(t) = 0\;.
\end{displaymath}
Thus $V$ has to vary linearly in $r$ and with abuse of notation we write $V(r,t)=rV(t)$, where $V(t)$ satisfies:
\begin{displaymath}
G'(t) + 3V(t)G(t) = 0\;\Rightarrow\;
G(t) = e^{-3\int_0^tV(s)ds}G(0)\;.
\end{displaymath}
Similarly, supposing further $M(r,t)=M(t)$, (\ref{eq:model_pde_m}) indicates $u(r,t)$ also varies linearly in $r$ and $u(r,t)=ru(t)$; hence we have:
\begin{displaymath}
M'(t) + 3(V(t)+u(t))M(t) = 0\;\Rightarrow\;
M(t) = e^{-3\int_0^t(V(s)+u(s))ds}M(0)\;.
\end{displaymath}
Lastly, (\ref{eq:model_pde_v}) is satisfied if and only if $V(t)+u(t)M(t)\equiv0$.
To this end, we let $V(t)=V_0>0$, which indicates:
\begin{equation}\label{eq:num_mod_1_sol}
G(t) = e^{-3V_0t}G(0)\;,\quad
M(t) = 1 - e^{-3V_0t}G(0)\;,\quad
u(t) = - \frac{V_0}{M(t)} = -\frac{V_0}{1-e^{-3V_0t}G(0)}\;.
\end{equation}
It is easy to check that (\ref{eq:num_mod_1_sol}) indeed solves the model problem, providing the boundary data:
\begin{equation}\label{eq:num_mod_1_bc}
M_{\bc}(t) = 1 - e^{-3V_0t}G(0)\;.
\end{equation}
In this case, the radius growth is exponential:
\begin{equation}\label{eq:num_mod_1_r}
R' = RV_0\ \Rightarrow\ 
R(t) = e^{V_0t}\;.
\end{equation}

\smallskip

In the first test, we set $V_0=0.5$ and consider the initial condition:
\begin{equation}\label{eq:num_mod_1_ic}
G(r,0) = 0.0\;,\quad
M(r,0) = 1.0\;.
\end{equation}
It is easy to tell that when the initial data $G(r,0)$ is zero, so is $G(r,t)$ for all $t>0$.
This is indeed satisfied by all our numerical solutions, whether using the conventional methods or the enhanced ones; hence we will not plot $G$ in the next results.
Solving the problem on a grid of $50$ uniform intervals until $T=2.0$, the histories of the incompressibility constraint violation index~(\ref{eq:rev_case_l1}) are plotted in Figure~\ref{fg:num_mod_1_inc}.
\begin{figure}\centering
  \includegraphics[width=.384\textwidth]{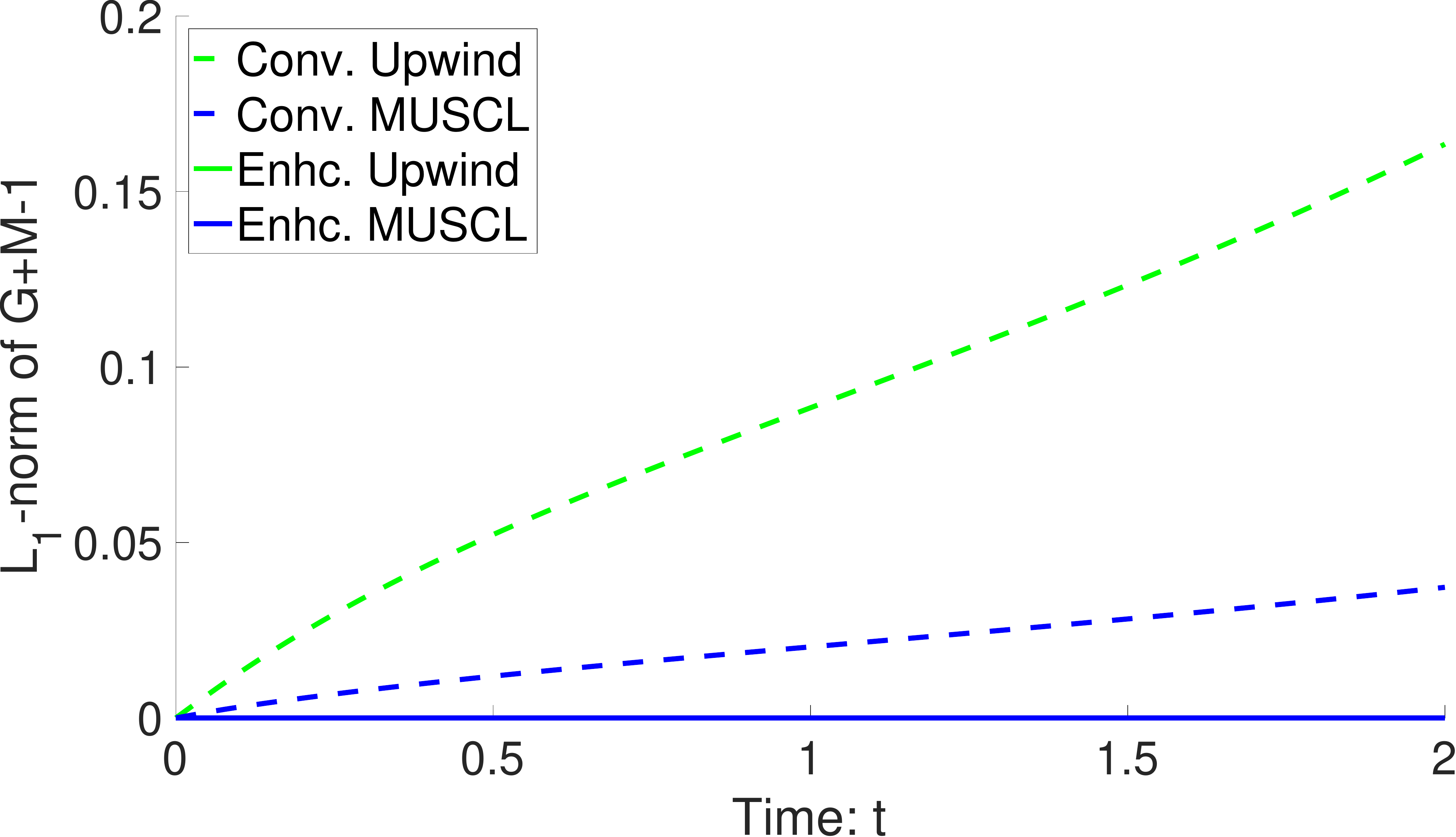}
  \vspace*{-.1in}
  \caption{Histories of $d_\theta$ on a $50$-interval grid.}
  \label{fg:num_mod_1_inc}
\end{figure}
We clearly observe that the incompressibility constraint is satisfied by both enhanced methods, whereas both conventional methods lead to increasing violation of this constraint as $t$ grows.

In Figure~\ref{fg:num_mod_1_sol}, we plot the radius histories and the profile of $M(r,T)$ in the left panel and the right panel, respectively.
\begin{figure}\centering
  \begin{subfigure}[b]{.48\textwidth}\centering
    \includegraphics[width=.8\textwidth]{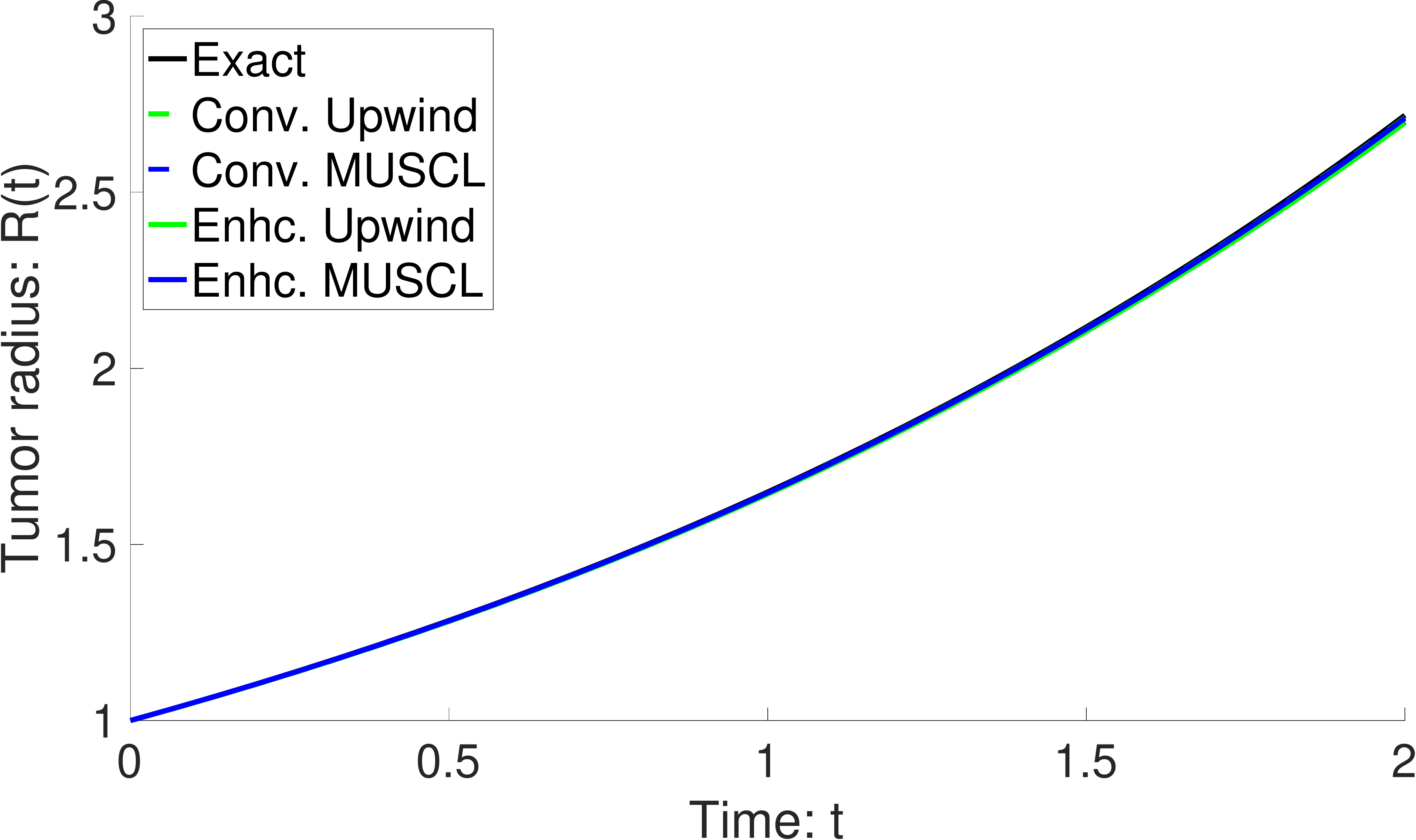}
    \caption{Radius growth history.}
    \label{fg:num_mod_1_sol_rad}
  \end{subfigure}
  \begin{subfigure}[b]{.48\textwidth}\centering
    \includegraphics[width=.8\textwidth]{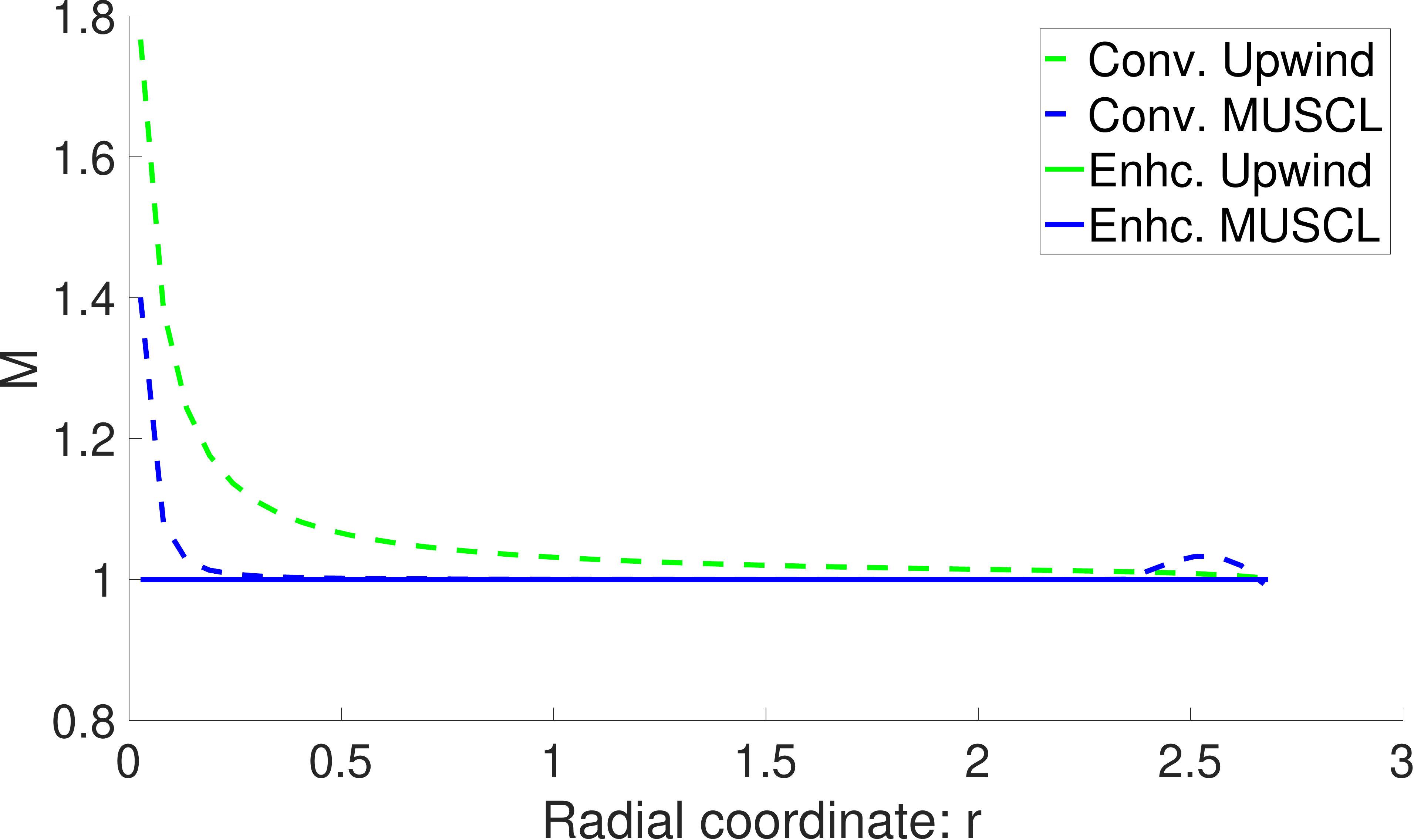}
    \caption{Cell numbers for $M$ at $T=2.0$.}
    \label{fg:num_mod_1_sol_cell}
  \end{subfigure}
  \vspace*{-.1in}
  \caption{Solutions to test 1 on a $50$-interval grid.}
  \label{fg:num_mod_1_sol}
\end{figure}
Here in Figure~\ref{fg:num_mod_1_sol_rad}, the reference radius growth curve~(\ref{eq:num_mod_1_r}) is plotted against the numerical ones; and we can see that all numerical solutions are close to the reference one, with the MUSCL fluxes provide slightly more accurate results than the upwind ones.
Figure~\ref{fg:num_mod_1_sol_cell} show that conventional methods fail to preserve constant solutions for a single species, indicating the violation of the geometrical conservation law; whereas both enhanced methods satisfy DGCL.

The same tests are performed on finer grids, with $100$, $200$, and $400$ cells, respectively; and we have very similar plots as before.
In Figure~\ref{fg:num_mod_1_conv_rad}, the final radius is plotted for each method on the sequence of four grids; and they're compared to the exact value. 
\begin{figure}\centering
  \includegraphics[width=.384\textwidth]{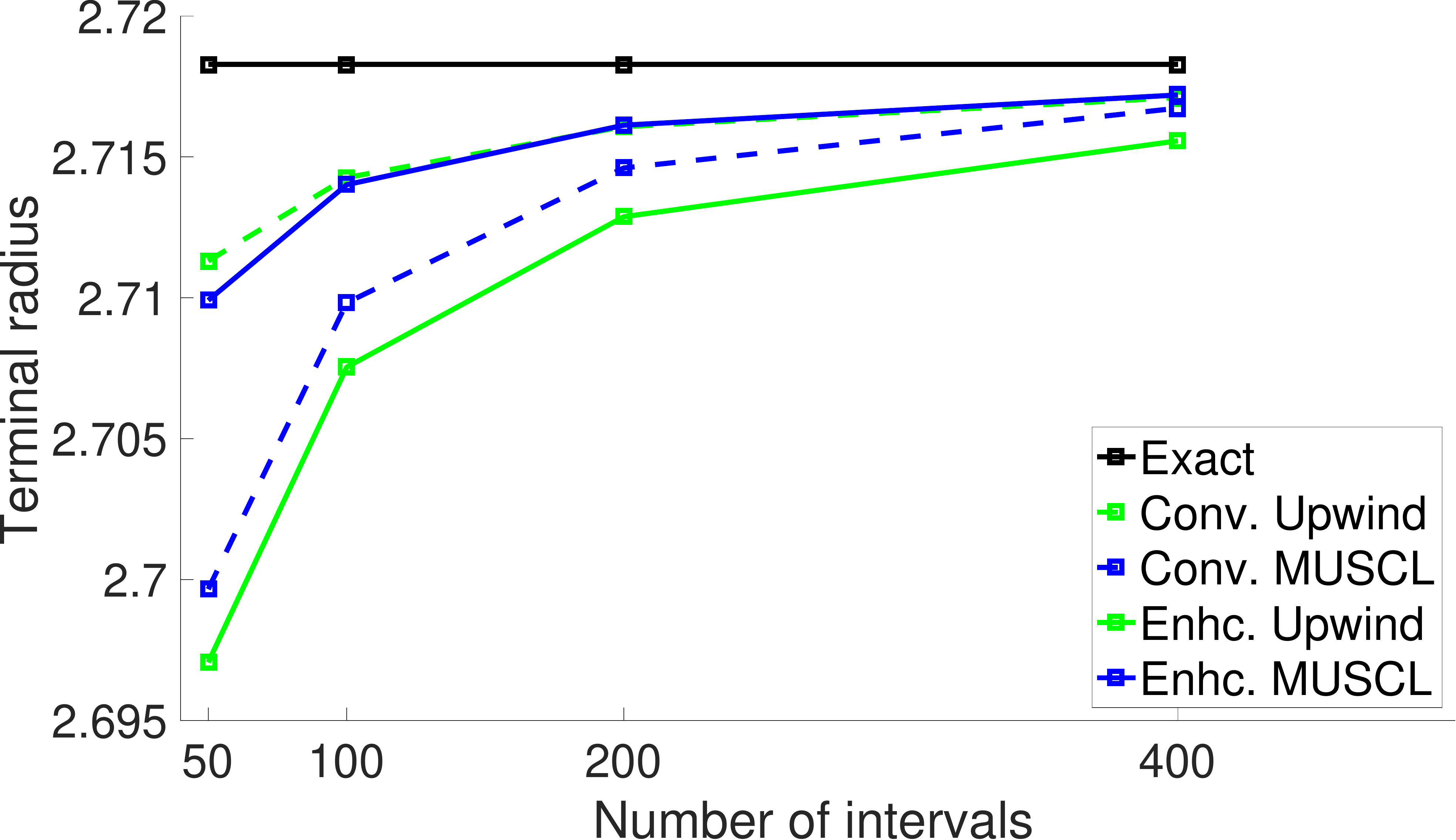}
  \vspace*{-.1in}
  \caption{Final radius (test 1) by various methods on a sequence of four grids.}
  \label{fg:num_mod_1_conv_rad}
\end{figure}
In addition, quantitative comparison is provided in Table~\ref{tb:num_mod_1_raderr} and Table~\ref{tb:num_mod_1_m_err}, which summarizes the errors in the final radius and the numerical solutions in $M$, respectively.
In order to evaluate the errors in $M$, we consider the $L_1$-error in the normalized coordinate that is defined as:
\begin{displaymath}
\sum_{i=1}^{N_{\eta}}\Delta\eta\abs{M_{k-1/2}^{N_\tau}-M^\ast(\eta_{k-1/2},T)}\;,
\end{displaymath}
where $N_\tau$ is the time step at $T$ and $M^\ast$ is the exact solution given by~(\ref{eq:num_mod_1_sol}); for this particular problem, we have $M^\ast\equiv1$.
\begin{table}\centering
\caption{Numerical errors in radius of test 1 at $T=2.0$.}
\label{tb:num_mod_1_raderr}
\vspace*{-.1in}
\begin{tabular}{|c|cc|cc|cc|cc|}
\hline
\multirow{2}{*}{$N_{\eta}$} & \multicolumn{2}{c|}{Conv. Upwind} 
                            & \multicolumn{2}{c|}{Conv. MUSCL} 
                            & \multicolumn{2}{c|}{Enhc. Upwind} 
                            & \multicolumn{2}{c|}{Enhc. MUSCL} \\ \cline{2-9}
                            & \multicolumn{1}{c|}{Error} & \multicolumn{1}{c|}{Rate} 
                            & \multicolumn{1}{c|}{Error} & \multicolumn{1}{c|}{Rate} 
                            & \multicolumn{1}{c|}{Error} & \multicolumn{1}{c|}{Rate} 
                            & \multicolumn{1}{c|}{Error} & \multicolumn{1}{c|}{Rate} \\ \hline
$50$  & -6.97e-3 &      & -1.86e-2 &      & -2.12e-2 &      & -8.37e-3 &      \\  
$100$ & -4.01e-3 & 0.80 & -8.46e-3 & 1.14 & -1.07e-2 & 0.98 & -4.27e-3 & 0.97 \\ 
$200$ & -2.21e-3 & 0.86 & -3.67e-3 & 1.20 & -5.40e-3 & 0.99 & -2.15e-3 & 0.99 \\ 
$400$ & -1.18e-3 & 0.91 & -1.55e-3 & 1.24 & -2.71e-3 & 1.00 & -1.08e-3 & 0.99 \\ \hline
\end{tabular}
\end{table}
\begin{table}\centering
\caption{$L_1$-errors in $M(\eta,T)$ of test 1 at $T=2.0$.}
\label{tb:num_mod_1_m_err}
\vspace*{-.1in}
\begin{tabular}{|c|cc|cc|cc|cc|}
\hline
\multirow{2}{*}{$N_{\eta}$} & \multicolumn{2}{c|}{Conv. Upwind} 
                            & \multicolumn{2}{c|}{Conv. MUSCL} 
                            & \multicolumn{2}{c|}{Enhc. Upwind} 
                            & \multicolumn{2}{c|}{Enhc. MUSCL} \\ \cline{2-9}
                            & \multicolumn{1}{c|}{Error} & \multicolumn{1}{c|}{Rate} 
                            & \multicolumn{1}{c|}{Error} & \multicolumn{1}{c|}{Rate} 
                            & \multicolumn{1}{c|}{Error} & \multicolumn{1}{c|}{Rate} 
                            & \multicolumn{1}{c|}{Error} & \multicolumn{1}{c|}{Rate} \\ \hline
$50$  & 6.03e-2 &      & 1.38e-2 &      & 3.89e-16 &     & 8.62e-16 &     \\  
$100$ & 3.46e-2 & 0.80 & 6.67e-3 & 1.05 & 8.18e-16 & N/A & 5.94e-16 & N/A \\ 
$200$ & 1.95e-2 & 0.83 & 3.24e-3 & 1.04 & 3.71e-15 & N/A & 7.43e-16 & N/A \\ 
$200$ & 1.09e-2 & 0.85 & 1.58e-3 & 1.04 & 9.80e-16 & N/A & 1.34e-16 & N/A \\ \hline
\end{tabular}
\end{table}
From Table~\ref{tb:num_mod_1_m_err}, we see that the numerical error in $M$ by the enhanced methods is at the scale of the machine accuracy, which indicates that they satisfy the discrete geometric conservation law; in comparison, the conventional finite volume method gives much larger errors.

\subsubsection{Test 2: Spatially constant solutions -- two species}
\label{sec:num_mod_2}
In the second test, we set again $V_0=0.5$ as in the previous problem, but consider the initial condition:
\begin{equation}\label{eq:num_mod_2_ic}
  G(r,0) = M(r,0) = 0.5\;,
\end{equation}
and modify the boundary condition accordingly, so that the exact solution is given by~(\ref{eq:num_mod_1_sol}).
On the coarsest grid with $50$ uniform cells, the numerical solutions at $T=2.0$ by all four methods are plotted in Figure~\ref{fg:num_mod_2_sol}.
\begin{figure}\centering
  \begin{subfigure}[b]{.48\textwidth}\centering
    \includegraphics[width=.8\textwidth]{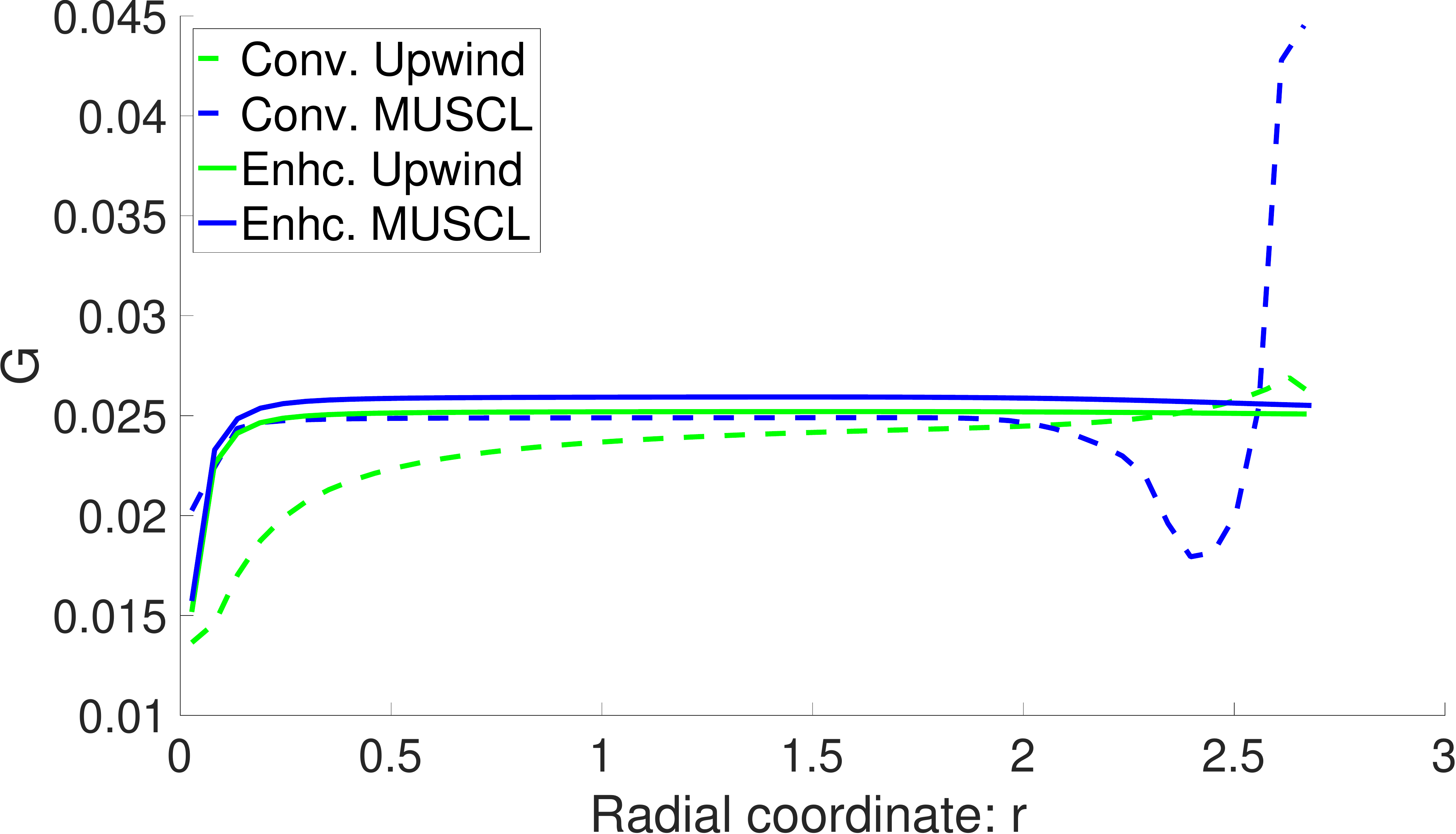}
    \caption{Cell numbers for $G$ at $T=2.0$.}
    \label{fg:num_mod_2_sol_g}
  \end{subfigure}
  \begin{subfigure}[b]{.48\textwidth}\centering
    \includegraphics[width=.8\textwidth]{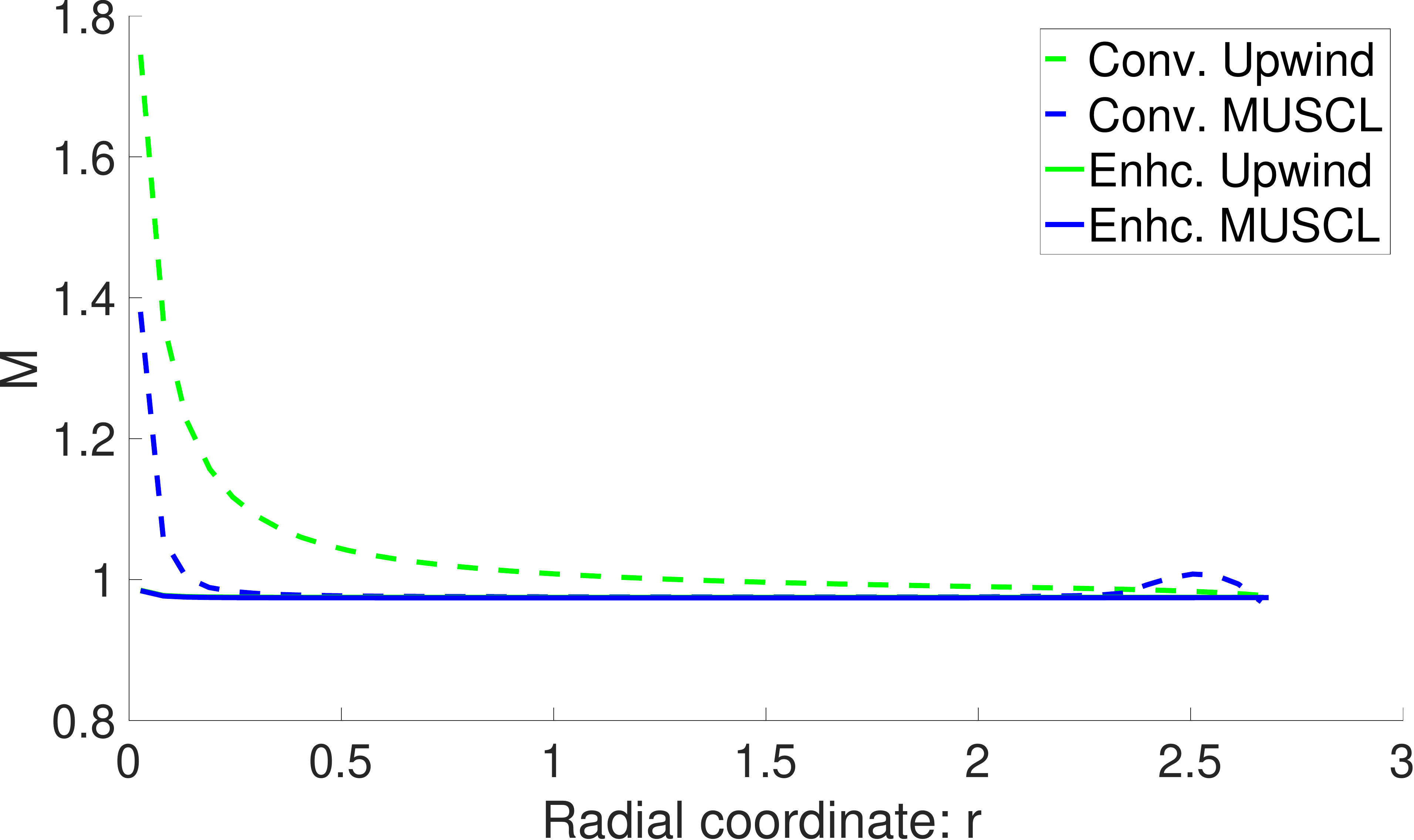}
    \caption{Cell numbers for $M$ at $T=2.0$.}
    \label{fg:num_mod_2_sol_m}
  \end{subfigure} \\
  \begin{subfigure}[b]{.48\textwidth}\centering
    \includegraphics[width=.8\textwidth]{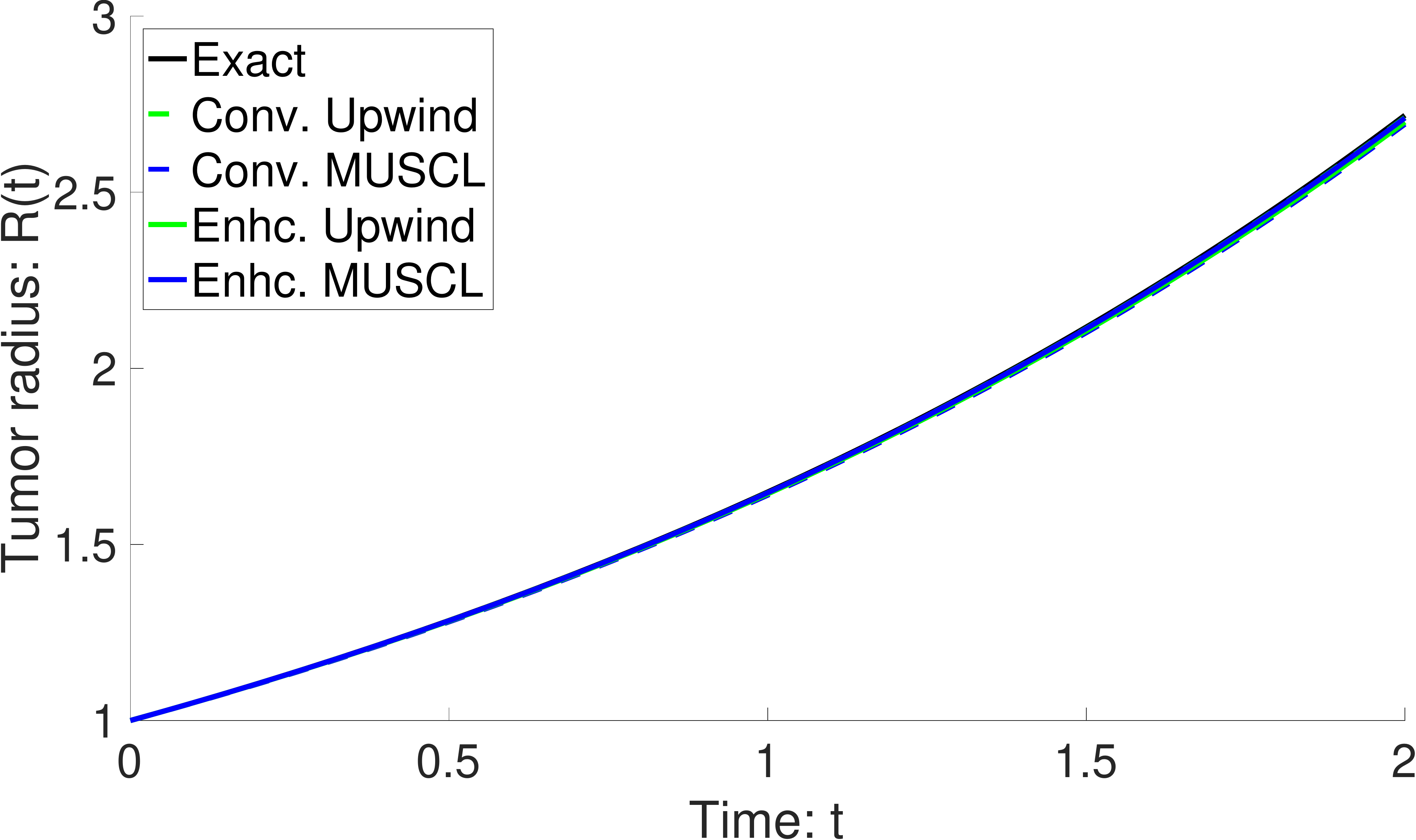}
    \caption{Radius growth history.}
    \label{fg:num_mod_2_sol_rad}
  \end{subfigure}
  \begin{subfigure}[b]{.48\textwidth}\centering
    \includegraphics[width=.8\textwidth]{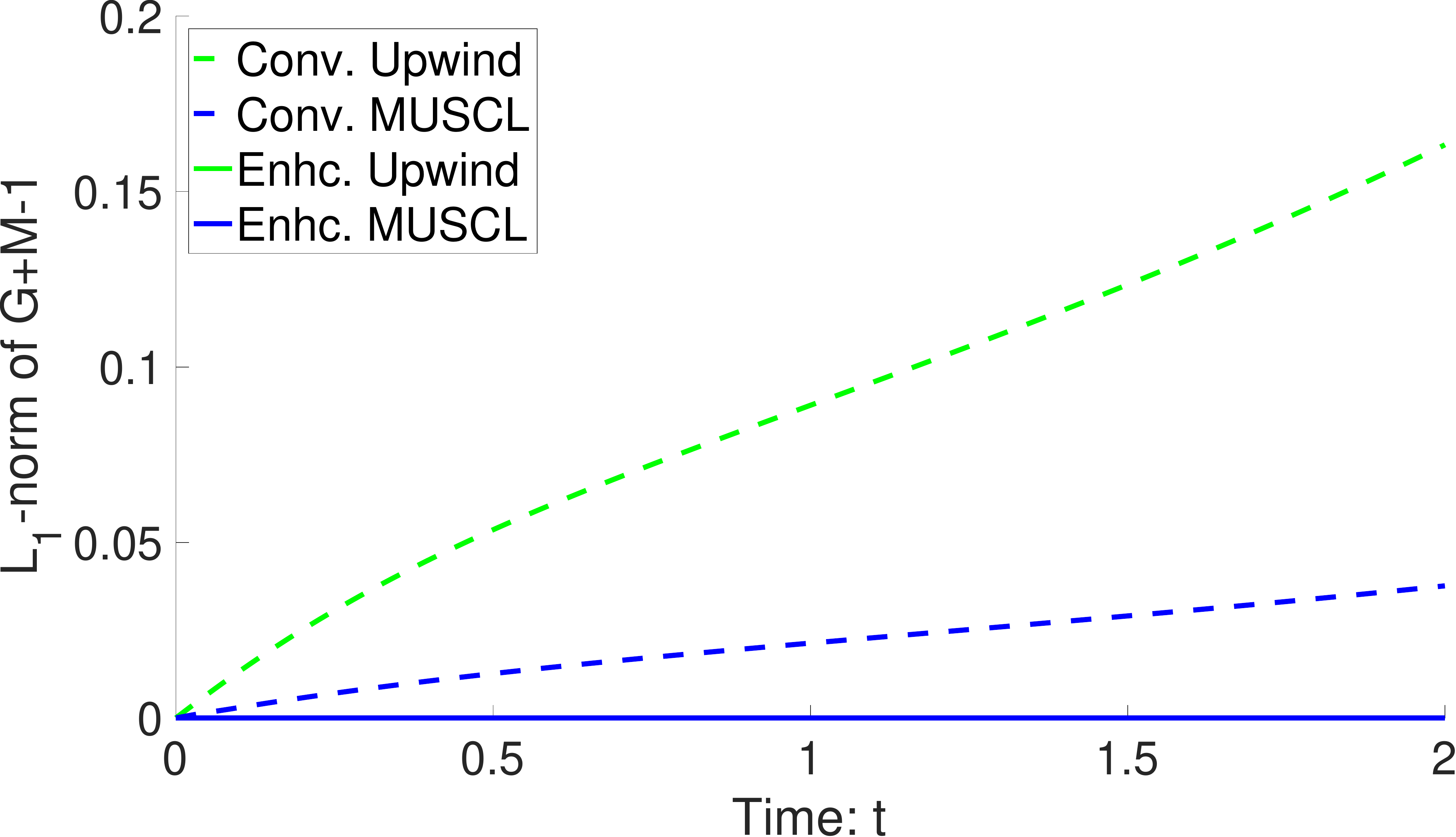}
    \caption{Histories of $d_\theta$.}
    \label{fg:num_mod_2_sol_inc}
  \end{subfigure}
  \vspace*{-.1in}
  \caption{Solutions to test 2 on a $50$-interval grid.}
  \label{fg:num_mod_2_sol}
\end{figure}
Again, the numerical radius growth agrees well with the exact one for all methods.
All four methods fail to compute spatially constant solutions in $G$ and $M$, see Figures~\ref{fg:num_mod_2_sol_g} and \ref{fg:num_mod_2_sol_m}; comparing the conventional and enhanced methods, however, we see clearly that the enhanced ones produce solutions with much less overshoots or undershoots.
In Figure~\ref{fg:num_mod_2_sol_inc}, once more we observe the satisfaction of DTCL by the enhanced methods, as the incompressibility constraint is well preserved. 

Similar as the previous test, the final radii computed by all four methods on a sequence of four meshes are plotted in Figure~\ref{fg:num_mod_2_conv_rad};
and the numerical errors are summarized in Table~\ref{tb:num_mod_2_raderr}--\ref{tb:num_mod_2_m_err} for quantitative comparison.
\begin{figure}\centering
  \includegraphics[width=.384\textwidth]{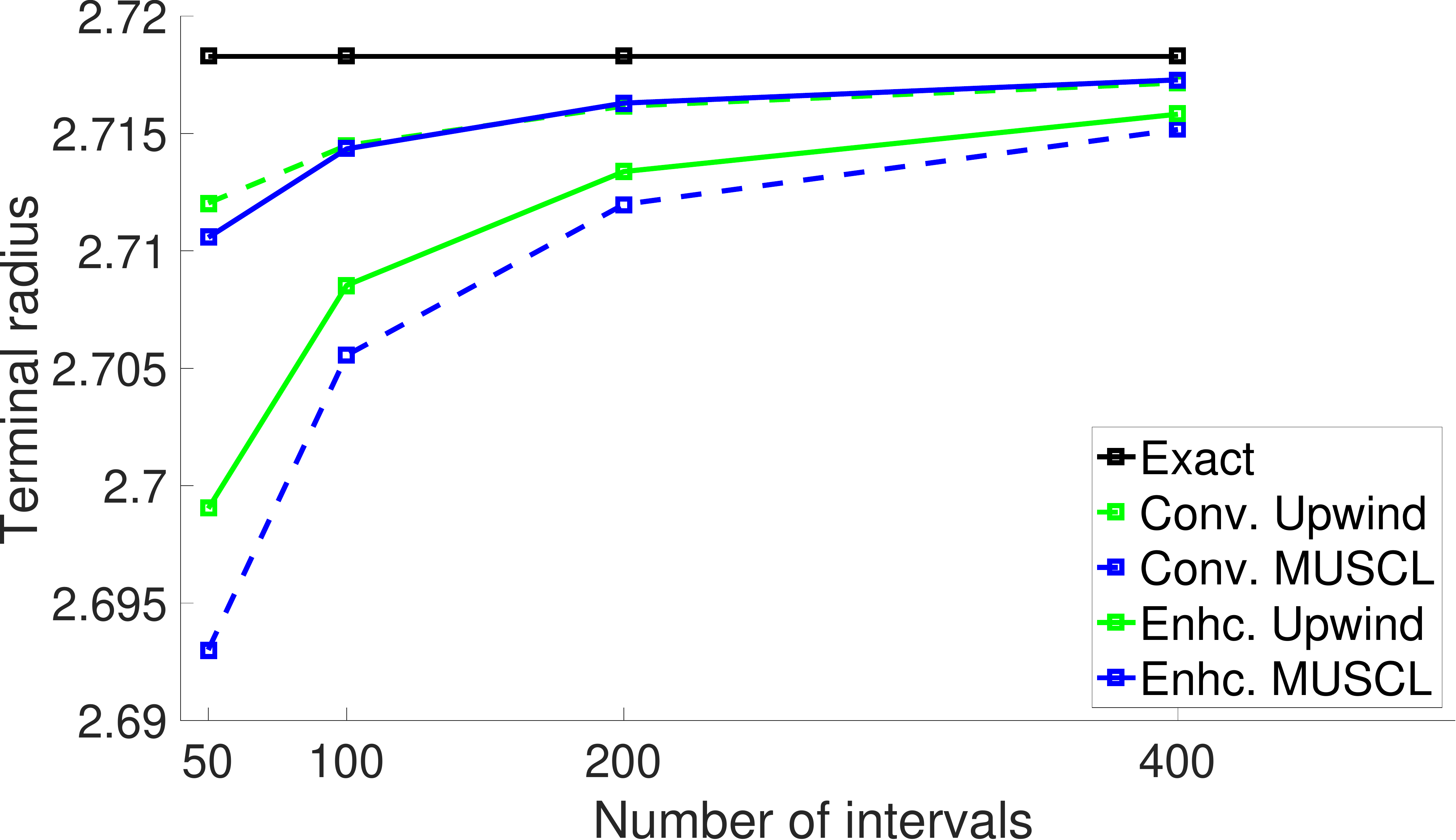}
  \vspace*{-.1in}
  \caption{Final radius (test 2) by various methods on a sequence of four grids.}
  \label{fg:num_mod_2_conv_rad}
\end{figure}
\begin{table}\centering
\caption{Numerical errors in radius of test 2 at $T=2.0$.}
\label{tb:num_mod_2_raderr}
\vspace*{-.1in}
\begin{tabular}{|c|cc|cc|cc|cc|}
\hline
\multirow{2}{*}{$N_{\eta}$} & \multicolumn{2}{c|}{Conv. Upwind} 
                            & \multicolumn{2}{c|}{Conv. MUSCL} 
                            & \multicolumn{2}{c|}{Enhc. Upwind} 
                            & \multicolumn{2}{c|}{Enhc. MUSCL} \\ \cline{2-9}
                            & \multicolumn{1}{c|}{Error} & \multicolumn{1}{c|}{Rate} 
                            & \multicolumn{1}{c|}{Error} & \multicolumn{1}{c|}{Rate} 
                            & \multicolumn{1}{c|}{Error} & \multicolumn{1}{c|}{Rate} 
                            & \multicolumn{1}{c|}{Error} & \multicolumn{1}{c|}{Rate} \\ \hline
$50$  & -6.25e-3 &      & -2.53e-2 &      & -1.93e-2 &      & -7.71e-3 &      \\  
$100$ & -3.79e-3 & 0.72 & -1.27e-2 & 0.99 & -9.75e-3 & 0.98 & -3.94e-3 & 0.97 \\ 
$200$ & -2.12e-3 & 0.84 & -6.30e-3 & 1.01 & -4.91e-3 & 0.99 & -1.99e-3 & 0.98 \\ 
$400$ & -1.13e-3 & 0.91 & -3.11e-3 & 1.02 & -2.46e-3 & 1.00 & -1.00e-3 & 0.99 \\ \hline
\end{tabular}
\end{table}
\begin{table}\centering
\caption{$L_1$-errors in $G(\eta,T)$ of test 2 at $T=2.0$.}
\label{tb:num_mod_2_g_err}
\vspace*{-.1in}
\begin{tabular}{|c|cc|cc|cc|cc|}
\hline
\multirow{2}{*}{$N_{\eta}$} & \multicolumn{2}{c|}{Conv. Upwind} 
                            & \multicolumn{2}{c|}{Conv. MUSCL} 
                            & \multicolumn{2}{c|}{Enhc. Upwind} 
                            & \multicolumn{2}{c|}{Enhc. MUSCL} \\ \cline{2-9}
                            & \multicolumn{1}{c|}{Error} & \multicolumn{1}{c|}{Rate} 
                            & \multicolumn{1}{c|}{Error} & \multicolumn{1}{c|}{Rate} 
                            & \multicolumn{1}{c|}{Error} & \multicolumn{1}{c|}{Rate} 
                            & \multicolumn{1}{c|}{Error} & \multicolumn{1}{c|}{Rate} \\ \hline
$50$  & 1.86e-3 &      & 1.57e-3 &      & 5.03e-4 &      & 1.09e-3 &      \\  
$100$ & 1.13e-3 & 0.72 & 9.79e-4 & 0.68 & 2.67e-4 & 0.91 & 5.76e-4 & 0.93 \\ 
$200$ & 6.63e-4 & 0.77 & 5.83e-4 & 0.75 & 1.40e-4 & 0.94 & 2.98e-4 & 0.95 \\
$400$ & 3.81e-4 & 0.80 & 3.39e-4 & 0.78 & 7.18e-5 & 0.96 & 1.53e-4 & 0.97 \\ \hline
\end{tabular}
\end{table}
\begin{table}\centering
\caption{$L_1$-errors in $M(\eta,T)$ of test 2 at $T=2.0$.}
\label{tb:num_mod_2_m_err}
\vspace*{-.1in}
\begin{tabular}{|c|cc|cc|cc|cc|}
\hline
\multirow{2}{*}{$N_{\eta}$} & \multicolumn{2}{c|}{Conv. Upwind} 
                            & \multicolumn{2}{c|}{Conv. MUSCL} 
                            & \multicolumn{2}{c|}{Enhc. Upwind} 
                            & \multicolumn{2}{c|}{Enhc. MUSCL} \\ \cline{2-9}
                            & \multicolumn{1}{c|}{Error} & \multicolumn{1}{c|}{Rate} 
                            & \multicolumn{1}{c|}{Error} & \multicolumn{1}{c|}{Rate} 
                            & \multicolumn{1}{c|}{Error} & \multicolumn{1}{c|}{Rate} 
                            & \multicolumn{1}{c|}{Error} & \multicolumn{1}{c|}{Rate} \\ \hline
$50$  & 6.18e-2 &      & 1.43e-2 &      & 5.03e-4 &      & 1.09e-3 &      \\  
$100$ & 3.55e-2 & 0.80 & 6.99e-3 & 1.03 & 2.67e-4 & 0.91 & 5.76e-4 & 0.93 \\ 
$200$ & 2.01e-2 & 0.82 & 3.44e-3 & 1.02 & 1.40e-4 & 0.94 & 2.98e-4 & 0.95 \\
$400$ & 1.12e-2 & 0.84 & 1.70e-3 & 1.02 & 7.18e-5 & 0.96 & 1.53e-4 & 0.97 \\ \hline
\end{tabular}
\end{table}
For this test, the interaction between the two cell numbers causes the $L_1$-errors in Table~\ref{tb:num_mod_2_g_err}--\ref{tb:num_mod_2_m_err} to be much larger than the previous case; however, the enhanced methods still produce much more accurate solutions than the conventional ones.
In addition, it is no coincidence that for both enhanced methods, the $L_1$-errors in $M$ are the same as those in $G$ on the same grids; 
this is because when DTCL and DGCL are satisfied, the two variables sum up to a constant value, whose numerical error is on the scale of machine precision, as shown in the previous test.

\subsubsection{Test 3: Monotone growth with constant boundary condition}
\label{sec:num_mod_3}
In the view of (\ref{eq:model_ode}), we can setup the velocity $u$ and boundary condition $M_{\bc}$ accordingly to obtain almost any desired monotonic growth pattern.
To be more specific, suppose a growth curve $\hat{R}(t)$ with $\hat{R}'>0$ is desired; all that we need to do is to make sure:
\begin{displaymath}
u(\hat{R}(t),t) < 0\;,\quad
u(\hat{R}(t),t)M_{\bc}(t) = -\hat{R}'(t)\;.
\end{displaymath}
Indeed for the simplified model~(\ref{eq:model_pde}), if $u(R(t),t)<0$ for all $t$, then the growth of $R(t)$ is completely determined by the boundary velocity and the boundary condition.
In this test, we consider a constant boundary condition $M_{\bc}(t) = 0.5$, and set up $u$ such that $R$ grows linearly as $R(t) = R(0) + V_0t$, where $V_0=0.5$:
\begin{equation}\label{eq:num_mod_3_u}
u(r,t) = -2V_0\;\sin\left(\frac{\pi\;r}{2(R(0)+V_0t)}\right)\;. 
\end{equation}
Here $u(r,t)$ is nonlinear in space, c.f. the previous test; and we do not expect the solutions to stay constant across the domain.

In Figure~\ref{fg:num_mod_3_sol}, we plot the numerical solutions for $G$ and $M$ at $T=2.0$ by all four methods as well as the histories of the radii and $d_{\theta}$.
\begin{figure}\centering
  \begin{subfigure}[b]{.48\textwidth}\centering
    \includegraphics[width=.8\textwidth]{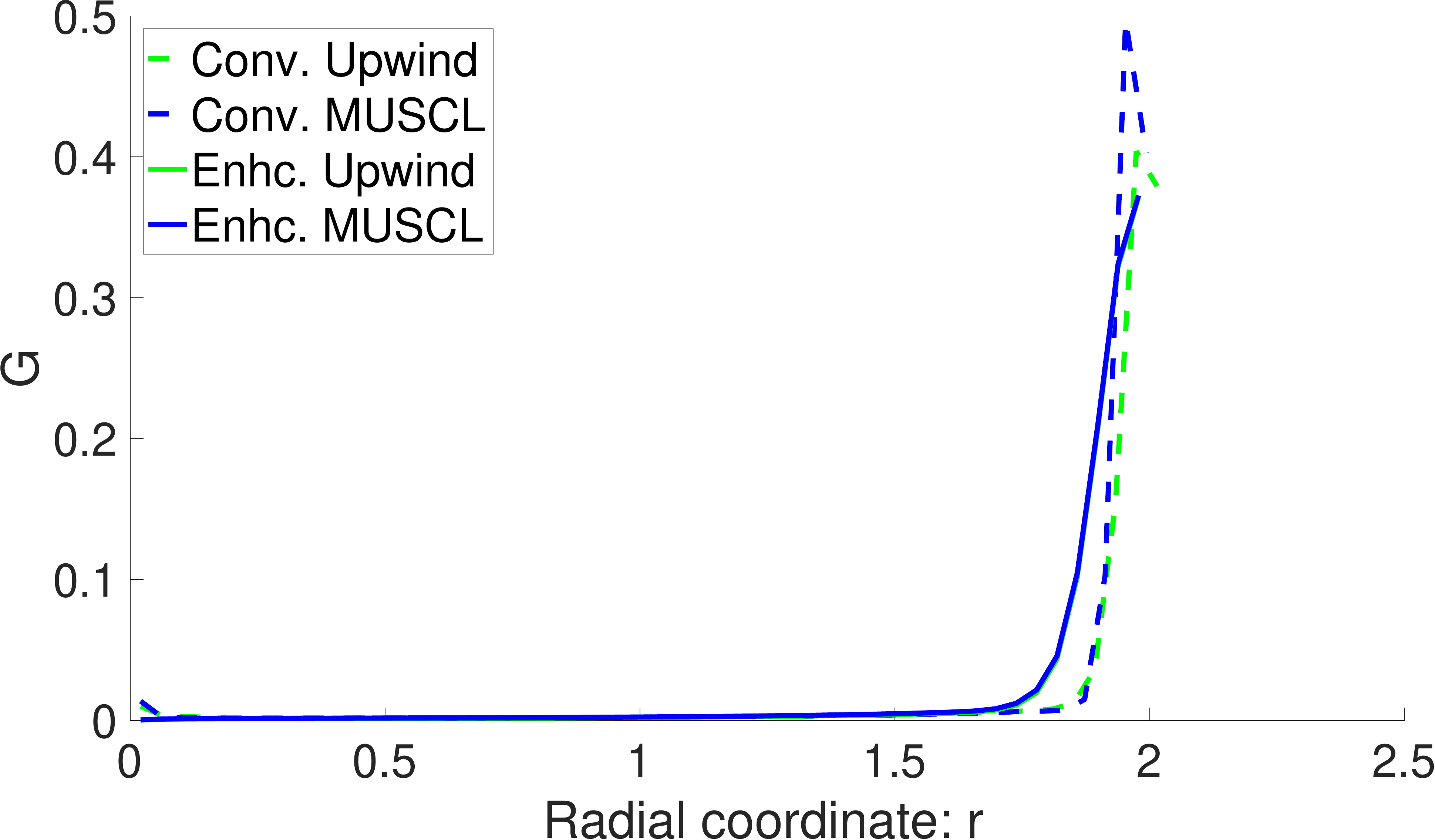}
    \caption{Cell numbers for $G$ at $T=2.0$.}
    \label{fg:num_mod_3_sol_g}
  \end{subfigure}
  \begin{subfigure}[b]{.48\textwidth}\centering
    \includegraphics[width=.8\textwidth]{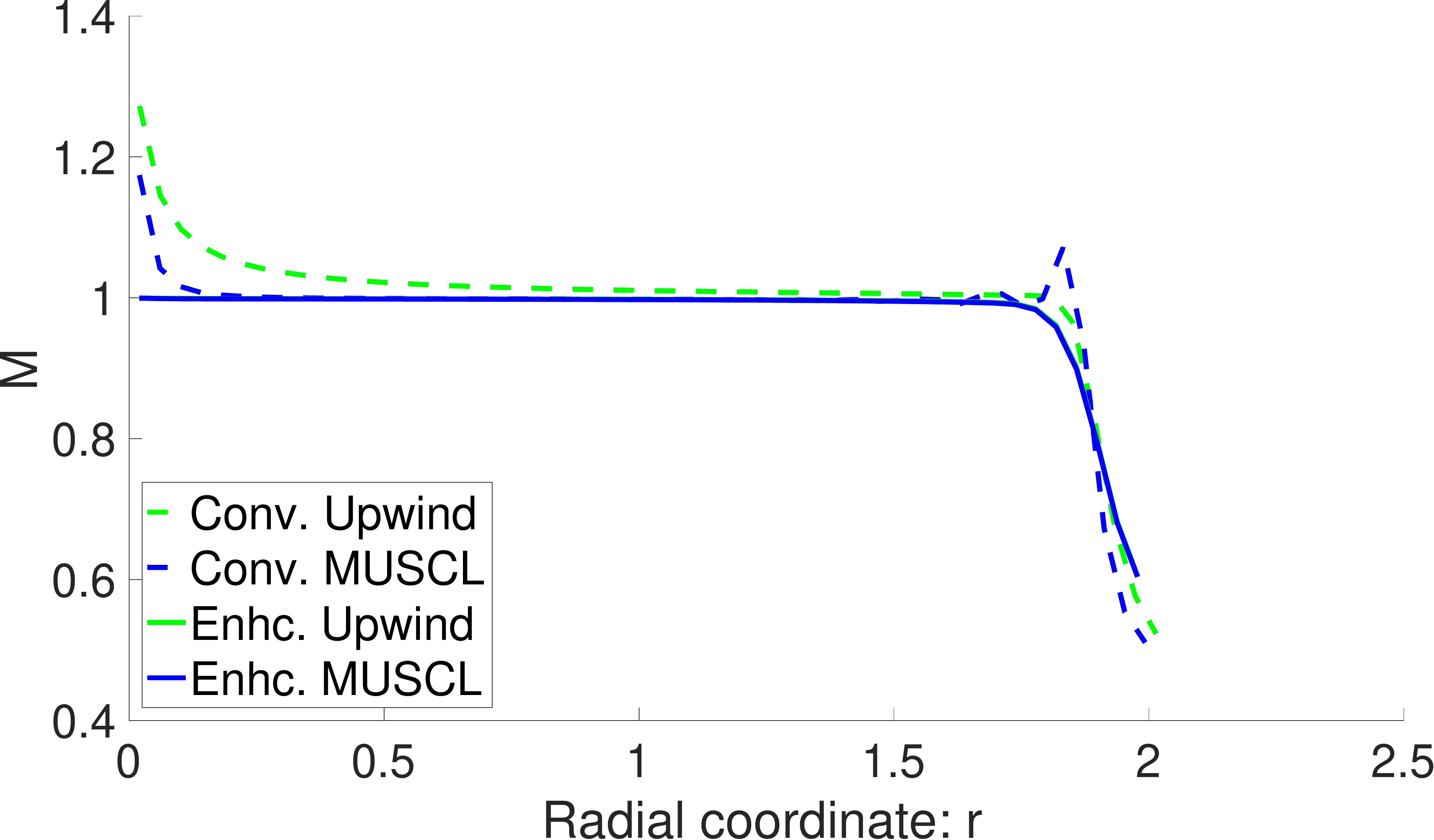}
    \caption{Cell numbers for $M$ at $T=2.0$.}
    \label{fg:num_mod_3_sol_m}
  \end{subfigure} \\
  \begin{subfigure}[b]{.48\textwidth}\centering
    \includegraphics[width=.8\textwidth]{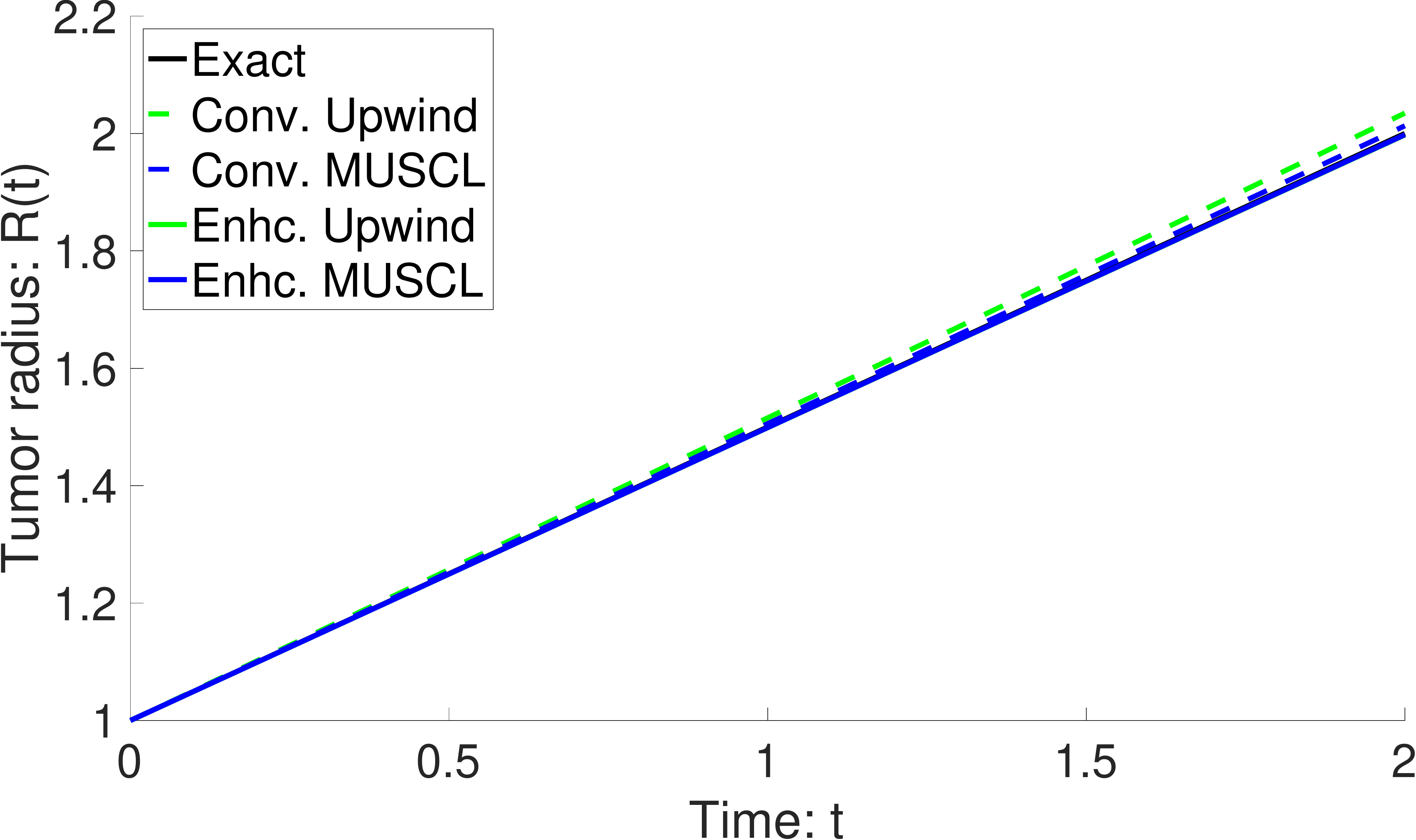}
    \caption{Radius growth history.}
    \label{fg:num_mod_3_sol_rad}
  \end{subfigure}
  \begin{subfigure}[b]{.48\textwidth}\centering
    \includegraphics[width=.8\textwidth]{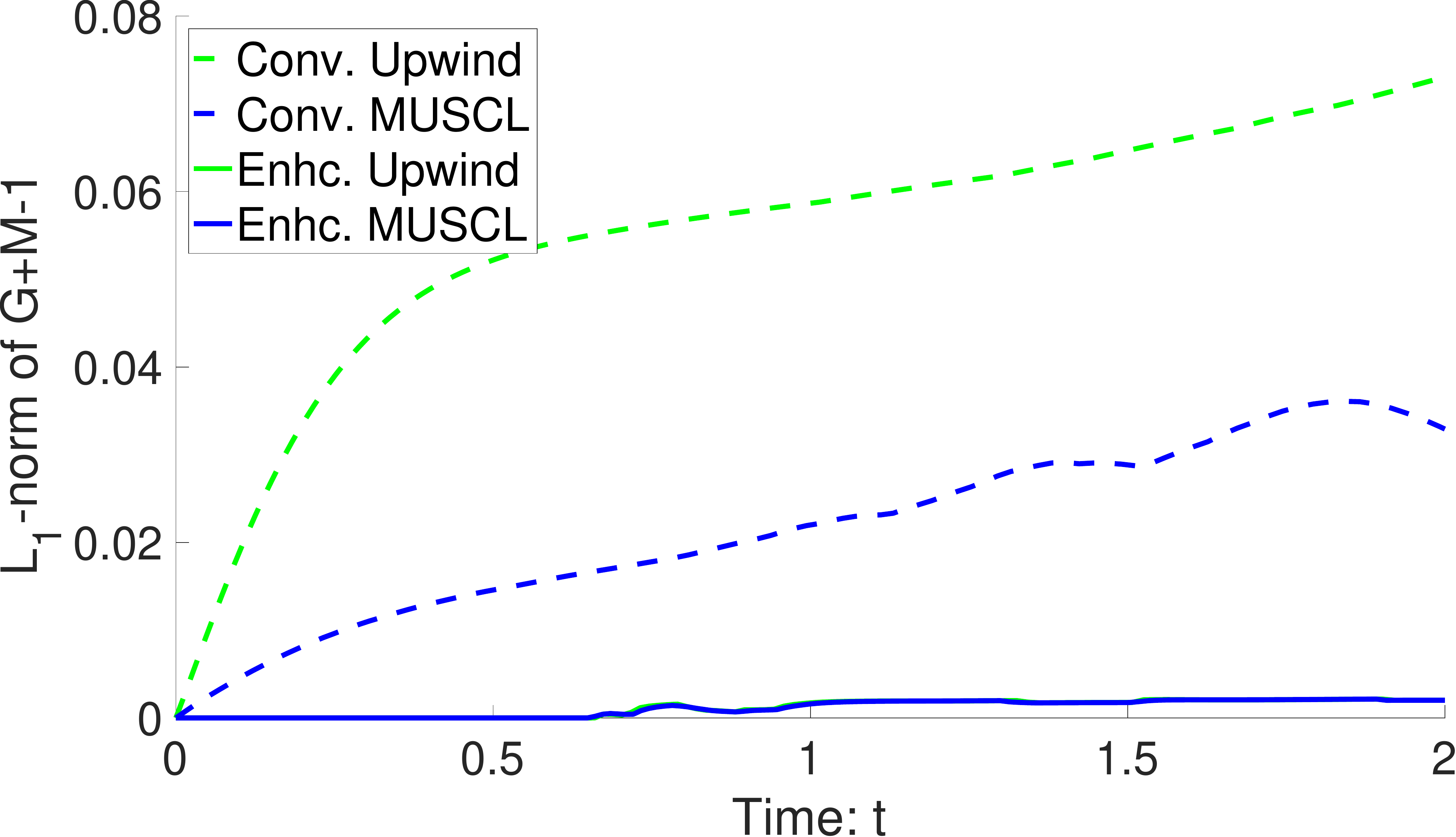}
    \caption{Histories of $d_\theta$.}
    \label{fg:num_mod_3_sol_inc}
  \end{subfigure}
  \vspace*{-.1in}
  \caption{Solutions to test 3 on a $50$-interval grid.}
  \label{fg:num_mod_3_sol}
\end{figure}
All methods predict well the linear growth of the radius.
Comparing the conventional methods and the enhanced ones, when the former are used, clear overshoots near the origin and spurious oscillations near the right boundary in both $G$ and $M$ are observed; however, both enhanced methods seem to lead to smooth solutions.
Similar patterns are observed on finer grids: the enhanced methods produce smooth solutions whereas the conventional ones lead to oscillations whose magnitudes increase as the grid is refined.
In Figure~\ref{fg:num_mod_3_sol_inc}, we see that the incompressibility constraint is much better preserved by the enhanced methods.
In comparison to the previous two tests, $d_\theta$ in this case is not at the machine precision level for the reason that the proposed methods are DGCL and DTCL for the interior nodes, whereas our theory does not address whether it is possible to satisfy these properties with arbitrary incoming data $M_{\bc}$.
This is exactly what happened here -- because of the jump in the boundary condition and the numerical solution at the last interval, small incompressibility violation is created and propagated towards the origin of the domain.
Nevertheless, the enhanced methods show significant improvement over their conventional counterparts. 

In Figure~\ref{fg:num_mod_3_conv_rad} and Table~\ref{tb:num_mod_3_raderr} we provide the convergence of the final radii by all methods on the same sequence of grids as well as quantitative comparisons.
Clearly, the enhanced methods provide much more accurate results than the conventional ones.
\begin{figure}\centering
  \includegraphics[width=.384\textwidth]{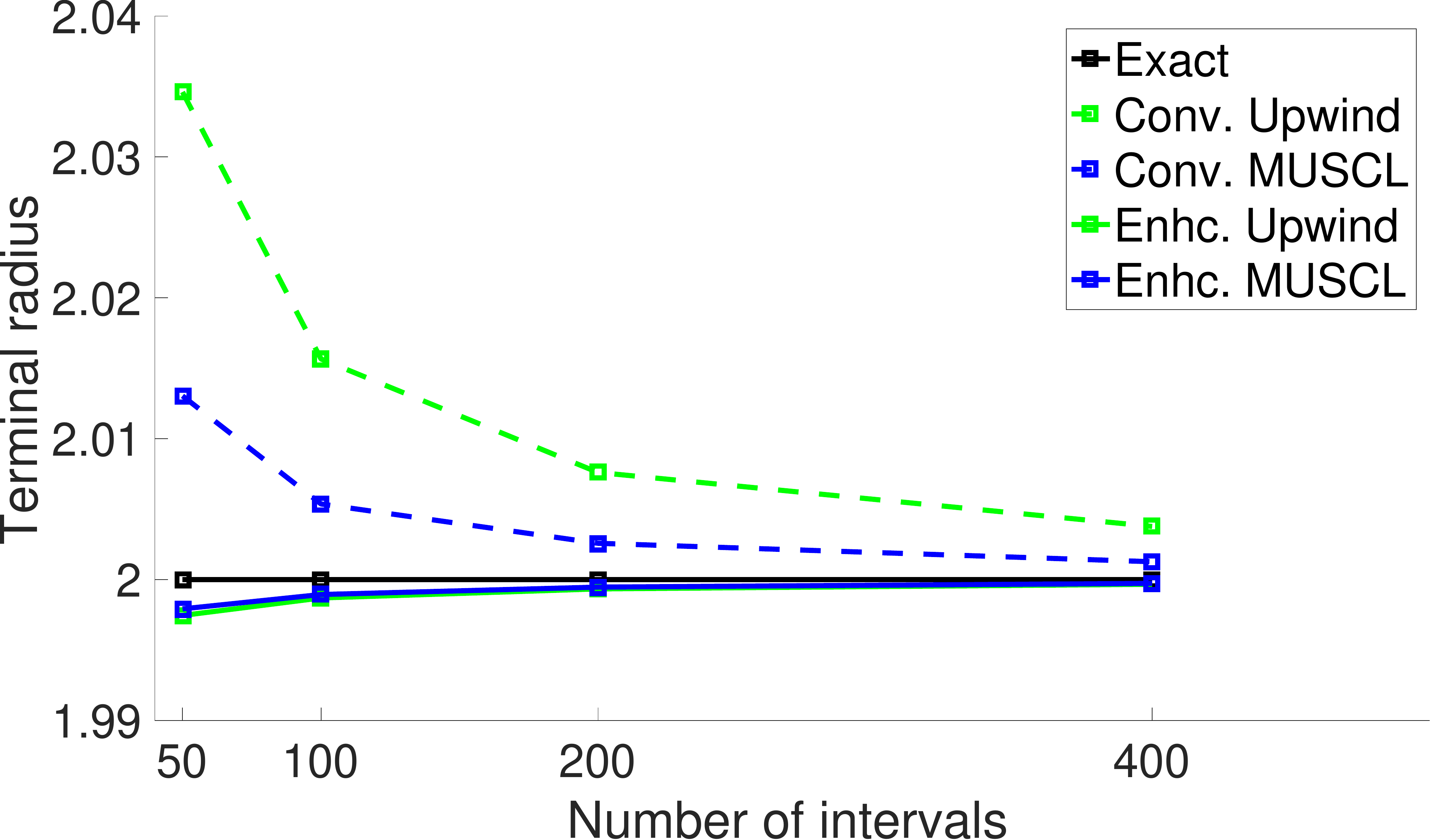}
  \vspace*{-.1in}
  \caption{Final radius (test 3) by various methods on a sequence of four grids.}
  \label{fg:num_mod_3_conv_rad}
\end{figure}
\begin{table}\centering
\caption{Numerical errors in radius of test 3 at $T=2.0$.}
\label{tb:num_mod_3_raderr}
\vspace*{-.1in}
\begin{tabular}{|c|cc|cc|cc|cc|}
\hline
\multirow{2}{*}{$N_{\eta}$} & \multicolumn{2}{c|}{Conv. Upwind} 
                            & \multicolumn{2}{c|}{Conv. MUSCL} 
                            & \multicolumn{2}{c|}{Enhc. Upwind} 
                            & \multicolumn{2}{c|}{Enhc. MUSCL} \\ \cline{2-9}
                            & \multicolumn{1}{c|}{Error} & \multicolumn{1}{c|}{Rate} 
                            & \multicolumn{1}{c|}{Error} & \multicolumn{1}{c|}{Rate} 
                            & \multicolumn{1}{c|}{Error} & \multicolumn{1}{c|}{Rate} 
                            & \multicolumn{1}{c|}{Error} & \multicolumn{1}{c|}{Rate} \\ \hline
$50$  & 3.46e-2 &      & 1.30e-2 &      & -2.54e-3 &      & -2.08e-3 &      \\  
$100$ & 1.57e-2 & 1.14 & 5.37e-3 & 1.28 & -1.30e-3 & 0.97 & -1.06e-3 & 0.97 \\ 
$200$ & 7.60e-3 & 1.04 & 2.58e-3 & 1.06 & -6.59e-4 & 0.98 & -5.33e-4 & 0.99 \\ 
$400$ & 3.77e-3 & 1.01 & 1.27e-3 & 1.02 & -3.31e-4 & 0.99 & -2.67e-4 & 1.00 \\ \hline
\end{tabular}
\end{table}

\subsubsection{Test 4: A prediction problem with non-monotone radius change}
\label{sec:num_mod_4}
Finally, we consider a test whose radius change cannot be predicted, by considering the velocity:
\begin{equation}\label{eq:num_mod_4_u}
u(r,t) = V_0\sin(r(1+t))\;,
\end{equation}
where $V_0=0.5$ is a constant that is small enough to prevent the domain from vanishing.

Numerical solutions on a grid of $50$ uniform interval are plotted in Figure~\ref{fg:num_mod_4_sol}.
\begin{figure}\centering
  \begin{subfigure}[b]{.48\textwidth}\centering
    \includegraphics[width=.8\textwidth]{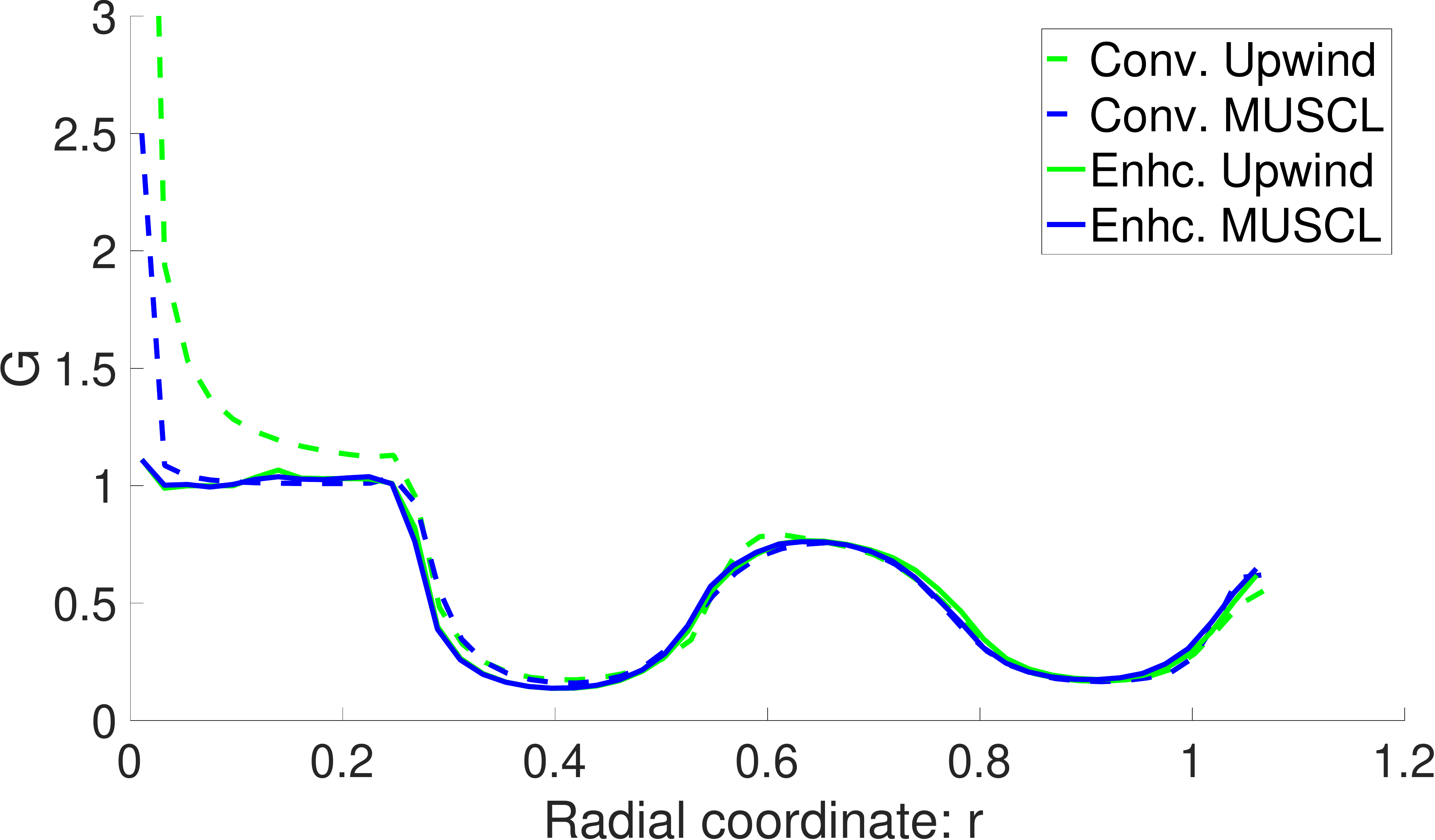}
    \caption{Cell numbers for $G$ at $T=2.0$.}
    \label{fg:num_mod_4_sol_g}
  \end{subfigure}
  \begin{subfigure}[b]{.48\textwidth}\centering
    \includegraphics[width=.8\textwidth]{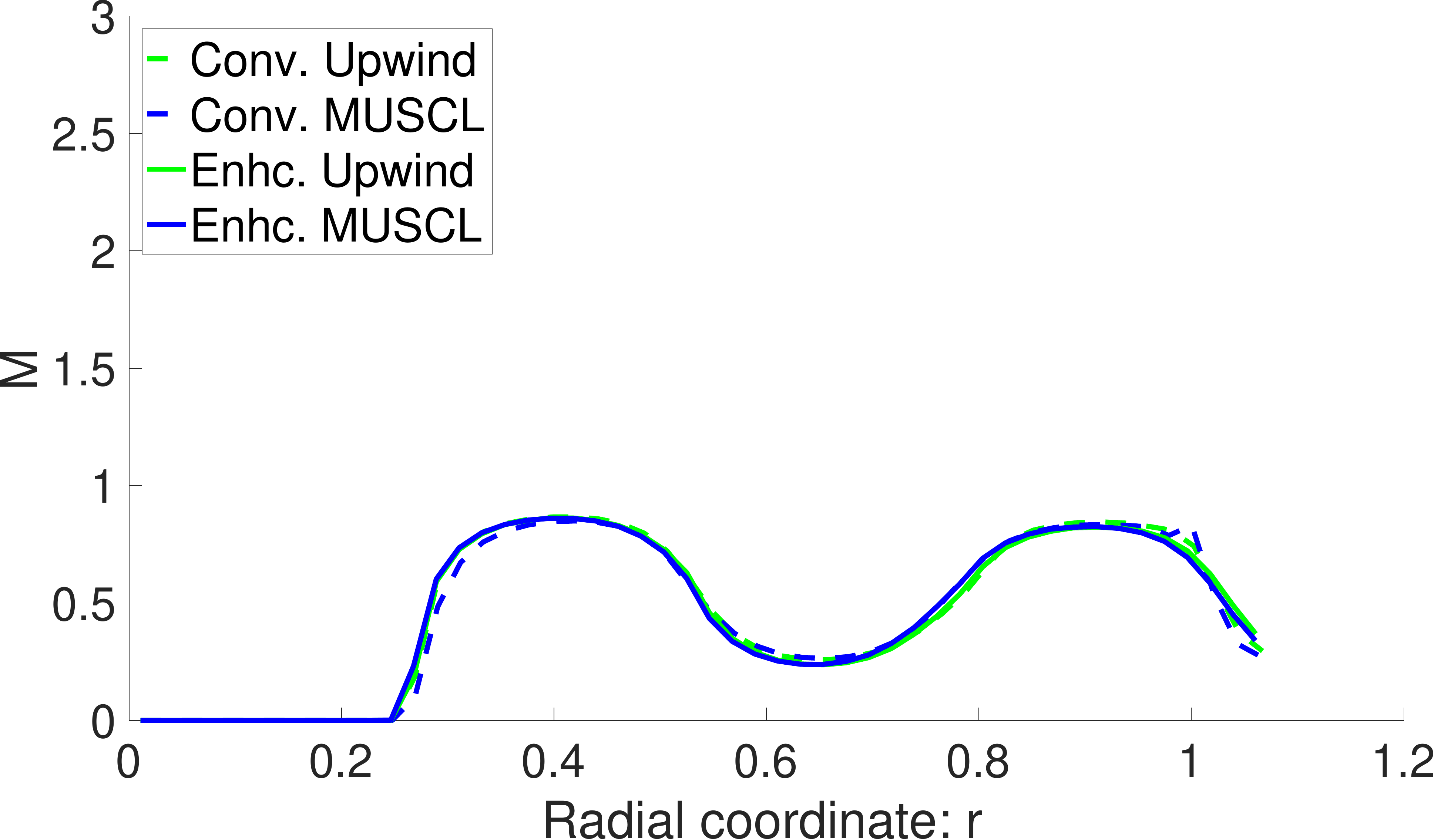}
    \caption{Cell numbers for $M$ at $T=2.0$.}
    \label{fg:num_mod_4_sol_m}
  \end{subfigure} \\
  \begin{subfigure}[b]{.48\textwidth}\centering
    \includegraphics[width=.8\textwidth]{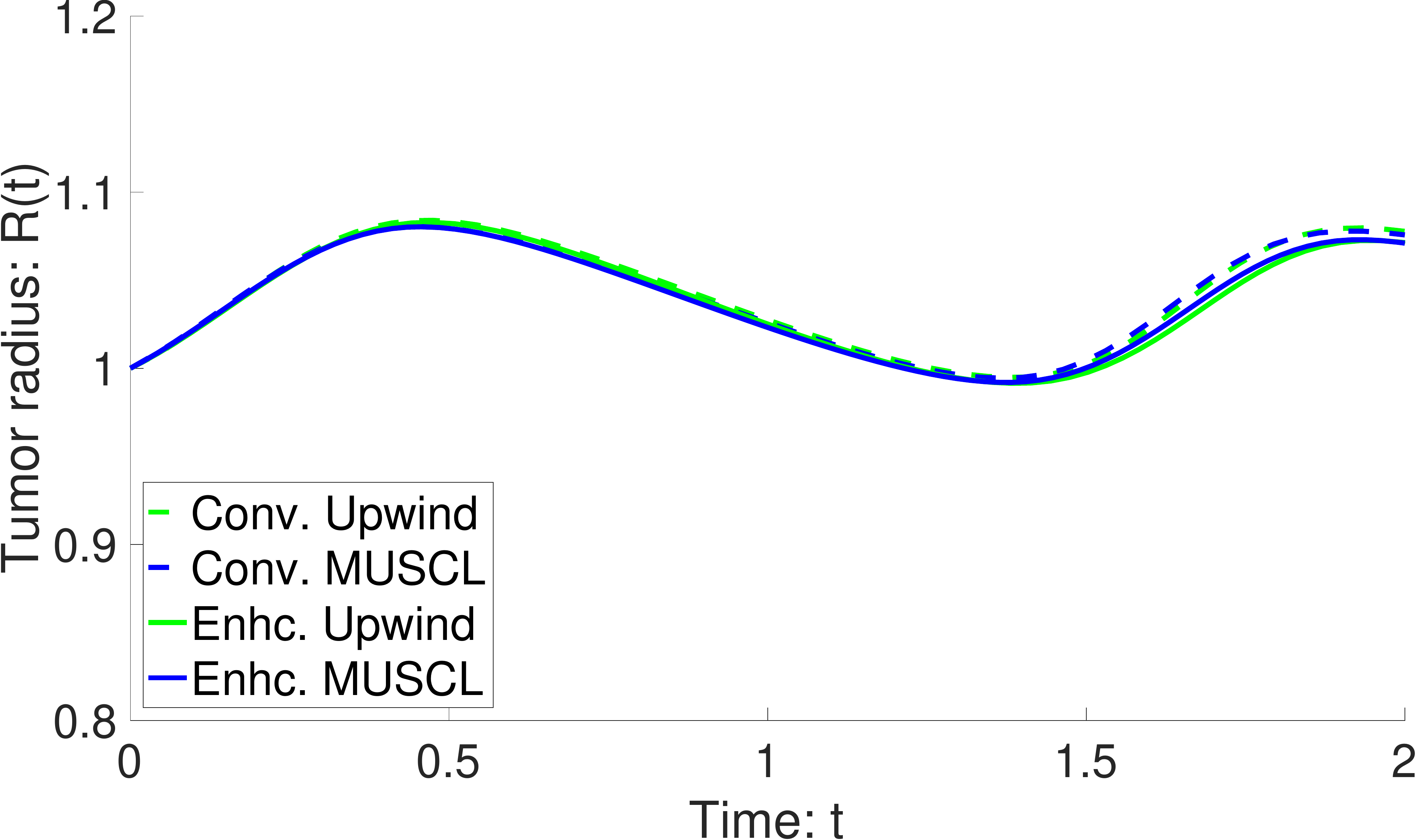}
    \caption{Radius growth history.}
    \label{fg:num_mod_4_sol_rad}
  \end{subfigure}
  \begin{subfigure}[b]{.48\textwidth}\centering
    \includegraphics[width=.8\textwidth]{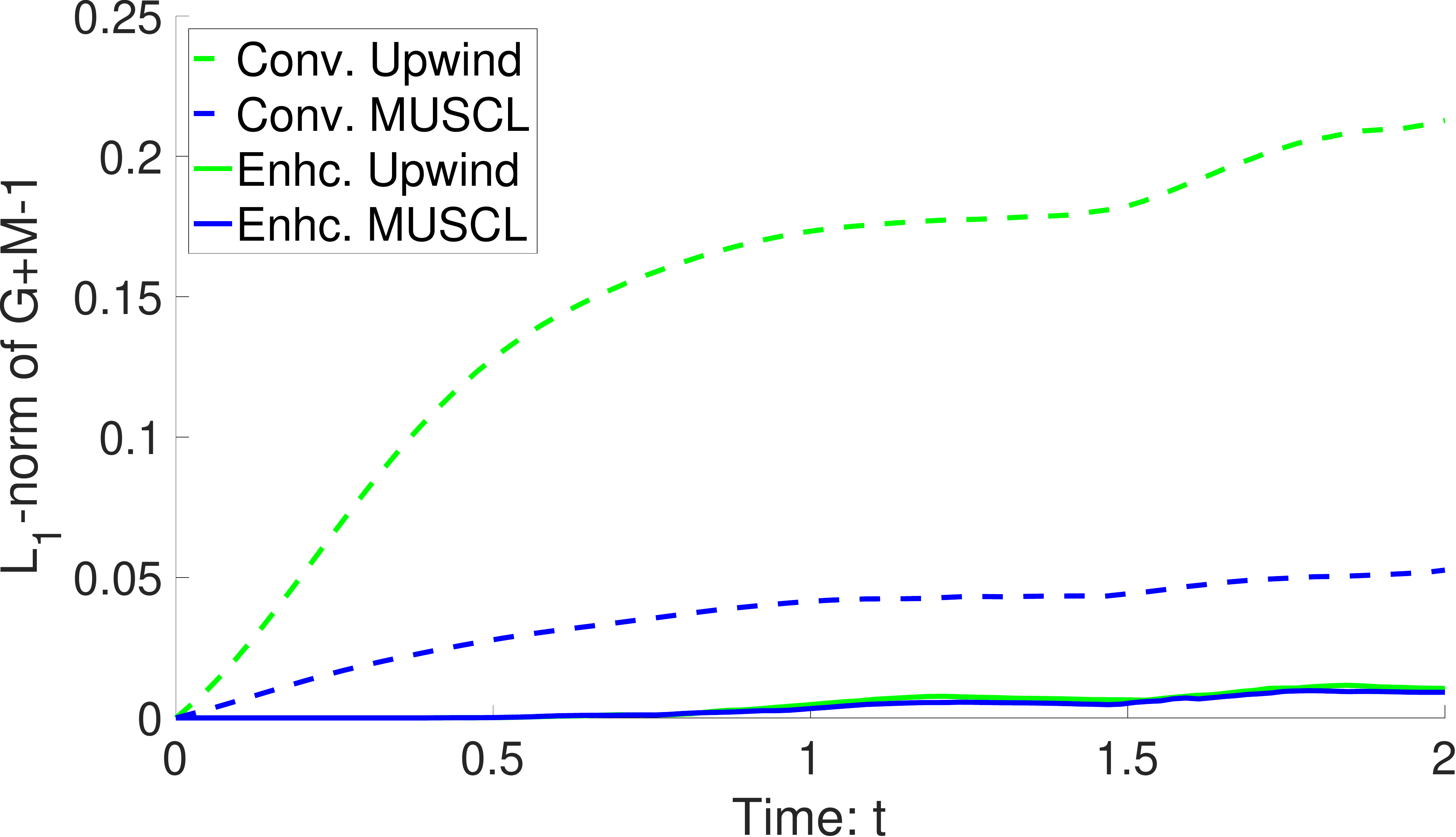}
    \caption{Histories of $d_\theta$.}
    \label{fg:num_mod_4_sol_inc}
  \end{subfigure}
  \vspace*{-.1in}
  \caption{Solutions to test 4 on a $50$-interval grid.}
  \label{fg:num_mod_4_sol}
\end{figure}
Similar as before, the enhanced methods produce much smoother solutions than the conventional ones, and the overshoots near the origin is in much smaller magnitudes.
Figure~\ref{fg:num_mod_4_sol_inc} reveals that the incompressibility constraint is much better preserved by the enhanced methods, whose small violation is due to the boundary conditions.

For this test, the overshoot in $G$ near the origin by the conventional methods increases rapidly as the mesh is refined, and eventually kill the computations when $400$ uniform intervals are used to discretize the domain.
In Figure~\ref{fg:num_mod_4_inc_fine} we plot the $d_\theta$ histories on a $100$-interval grid and a $200$-interval grid in the left panel and right panel, respectively.
Table~\ref{tb:num_mod_4_rad} summarizes the final radii by all methods on the sequence of the four grids; note that we do not compute the numerical error as before since the exact value is unknown.
\begin{figure}\centering
  \begin{subfigure}[b]{.48\textwidth}\centering
    \includegraphics[width=.8\textwidth]{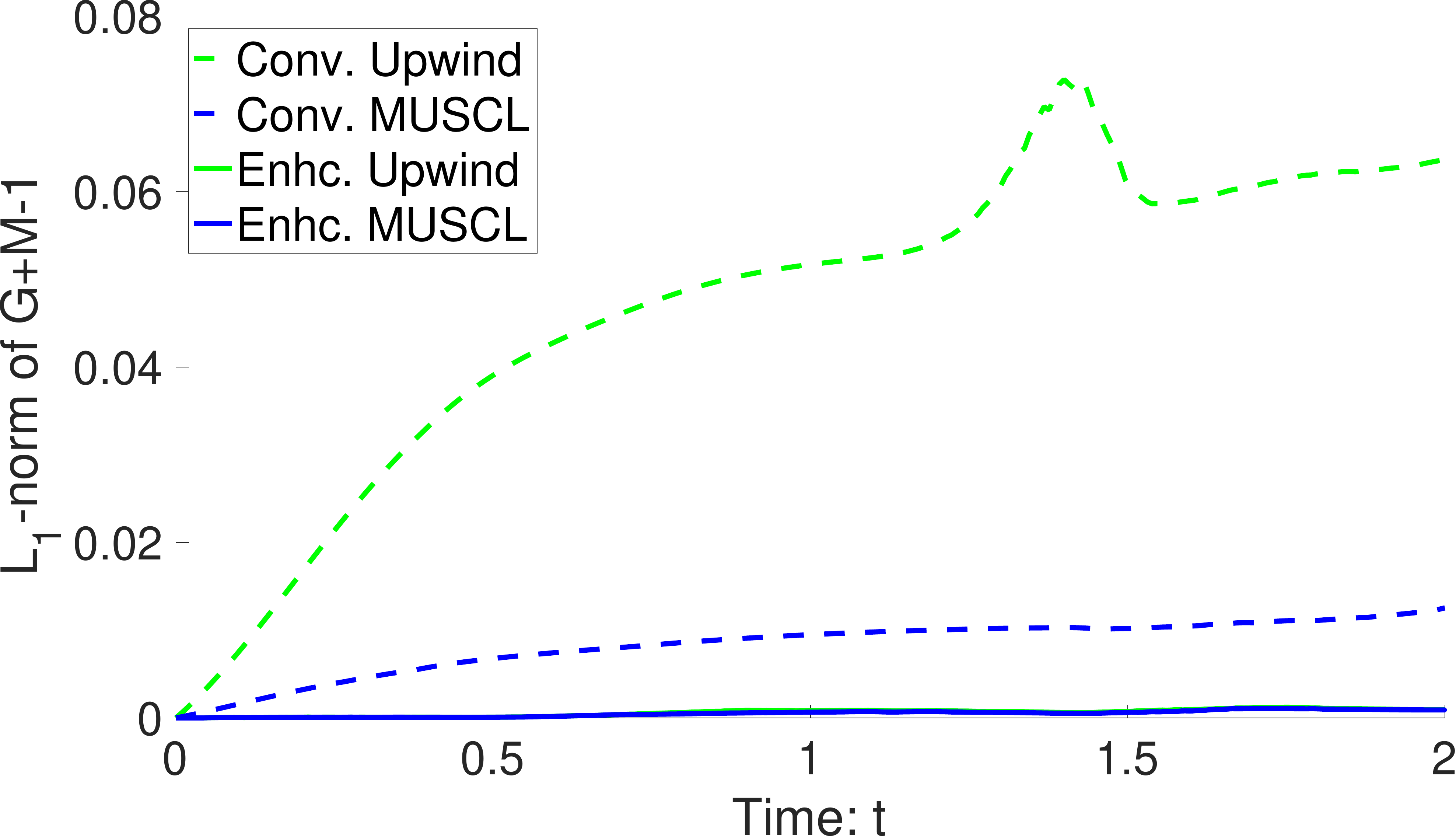}
    \caption{$d_\theta$ histories on a $200$-interval grid.}
    \label{fg:num_mod_4_inc_fine_200}
  \end{subfigure}
  \begin{subfigure}[b]{.48\textwidth}\centering
    \includegraphics[width=.8\textwidth]{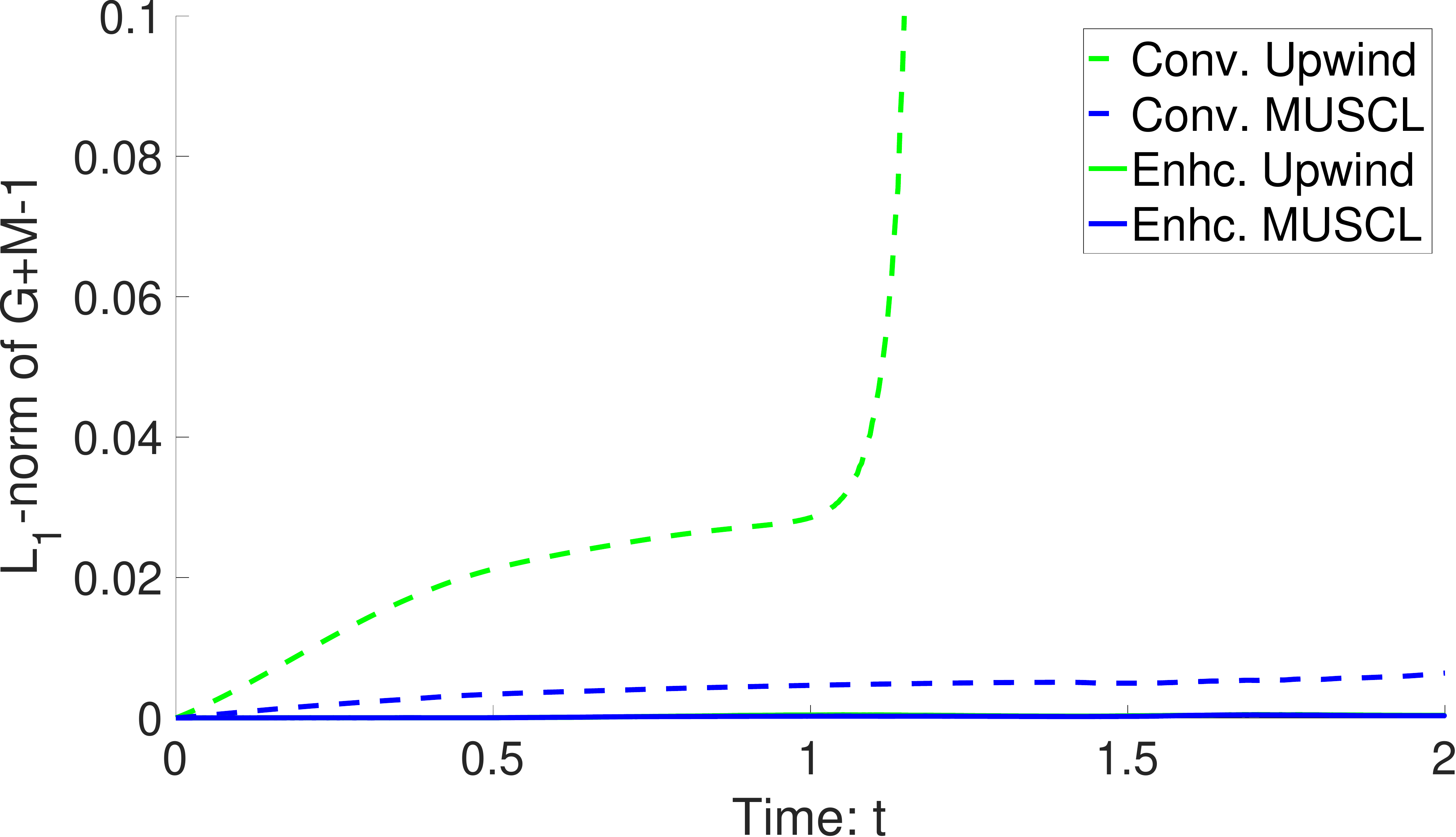}
    \caption{$d_\theta$ histories on a $400$-interval grid.}
    \label{fg:num_mod_4_inc_fine_400}
  \end{subfigure} 
  \vspace*{-.1in}
  \caption{$d_\theta$ histories of test 4 on two finer grids.}
  \label{fg:num_mod_4_inc_fine}
\end{figure}
\begin{table}\centering
\caption{Numerical solutions of the final radius of test 3 at $T=2.0$.}
\label{tb:num_mod_4_rad}
\vspace*{-.1in}
\begin{tabular}{|c|c|c|c|c|}
\hline
$N_{\eta}$ & Conv. Upwind 
           & Conv. MUSCL
           & Enhc. Upwind 
           & Enhc. MUSCL \\ \hline
$50$  & 1.0777       & 1.0757 & 1.0713 & 1.0711 \\  
$100$ & 1.0763       & 1.0747 & 1.0735 & 1.0734 \\ 
$200$ & 1.0754       & 1.0745 & 1.0745 & 1.0745 \\ 
$400$ & \textrm{nan} & 1.0747 & 1.0749 & 1.0749 \\ \hline
\end{tabular}
\end{table}
In the table, no data is reported for the conventional upwind method since the computation breaks up around $t=1.2$, c.f. Figure~\ref{fg:num_mod_4_inc_fine_400};
in addition, the conventional MUSCL method seems to produce non-monotone ``convergence'' pattern, whereas both enhanced methods seem to provide monotonic and convergent solutions.

\subsection{The tumor growth model}
\label{sec:num_tum}
Next we consider the tumor growth model~(\ref{eq:rev_eqn}) and our first test revisits the case study in Section~\ref{sec:rev_case}.
Then, a set of parameters are chosen according to the study of~\cite{BNiu:2018a} to assess the impact of using the enhanced methods in practical predictions.
The PDE system~(\ref{eq:rev_eqn}) involves one more equation for the velocity field $U$; for all the four methods including the enhanced ones, we use the same discretization method as described in Section~\ref{sec:rev_fvm} to update $A$.

\subsubsection{The case study revisited}
\label{sec:num_tum_case}
Using the same initial and boundary conditions as in Section~\ref{sec:rev_case}, the tumor growth model~(\ref{eq:rev_eqn}) is solved by the four methods until $T=1.0$.
The sample solutions on a grid of $50$ uniform intervals are plotted in Figure~\ref{fg:num_tum_case_sol}.
\begin{figure}\centering
  \begin{subfigure}[b]{.48\textwidth}\centering
    \includegraphics[width=.8\textwidth]{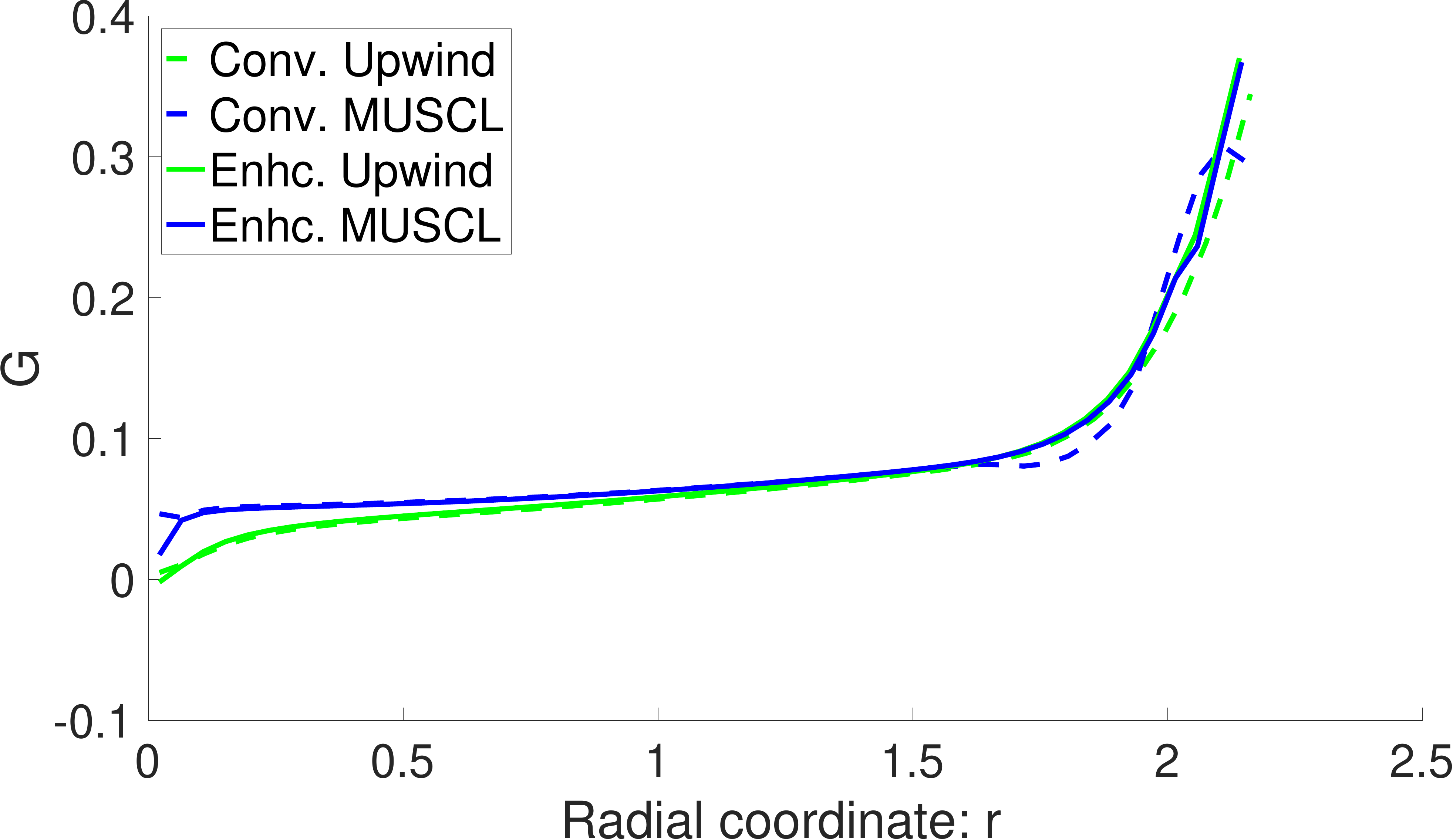}
    \caption{Cell numbers for $G$ at $T=2.0$.}
    \label{fg:num_tum_case_sol_g}
  \end{subfigure}
  \begin{subfigure}[b]{.48\textwidth}\centering
    \includegraphics[width=.8\textwidth]{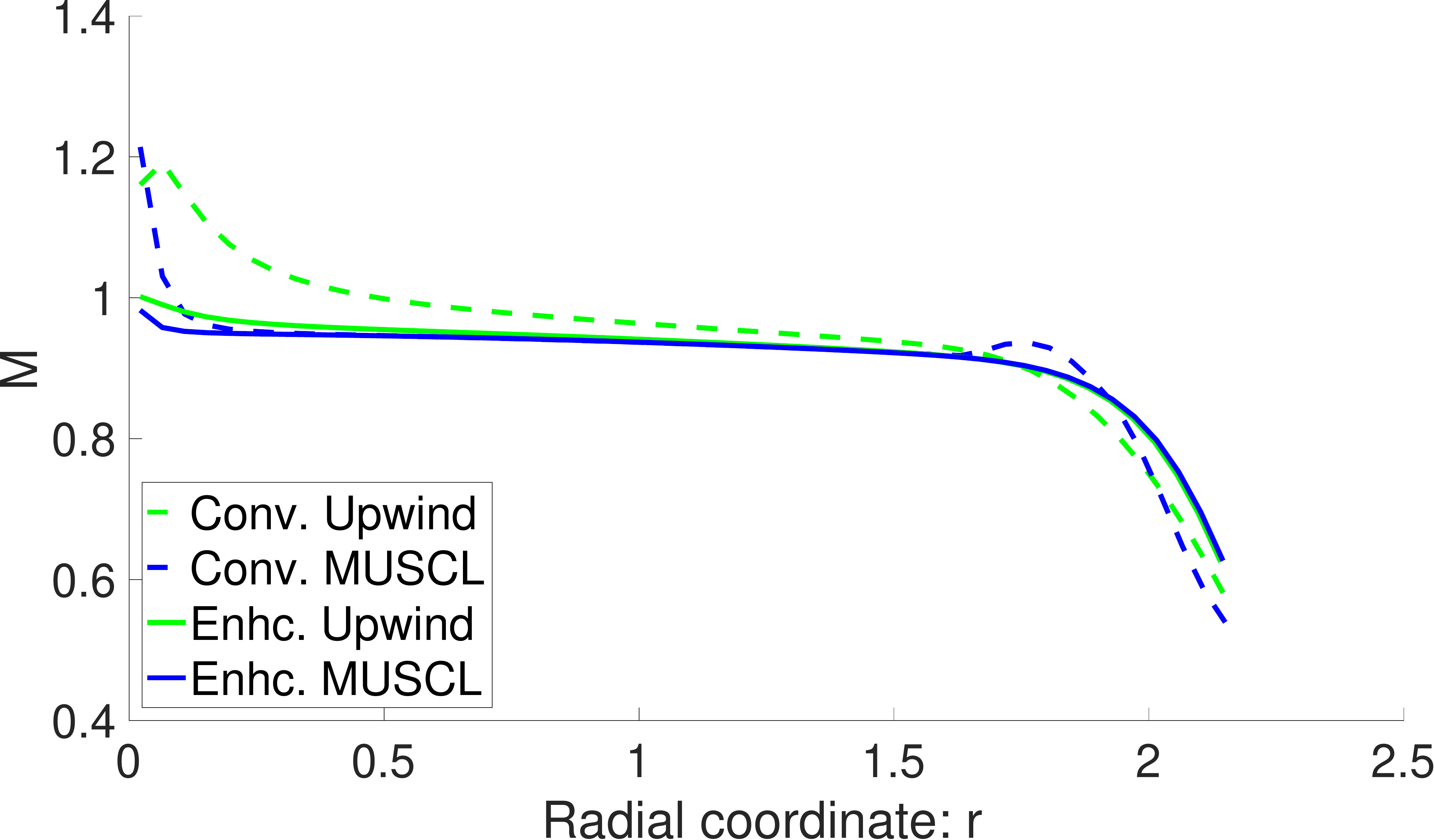}
    \caption{Cell numbers for $M$ at $T=2.0$.}
    \label{fg:num_tum_case_sol_m}
  \end{subfigure} \\
  \begin{subfigure}[b]{.48\textwidth}\centering
    \includegraphics[width=.8\textwidth]{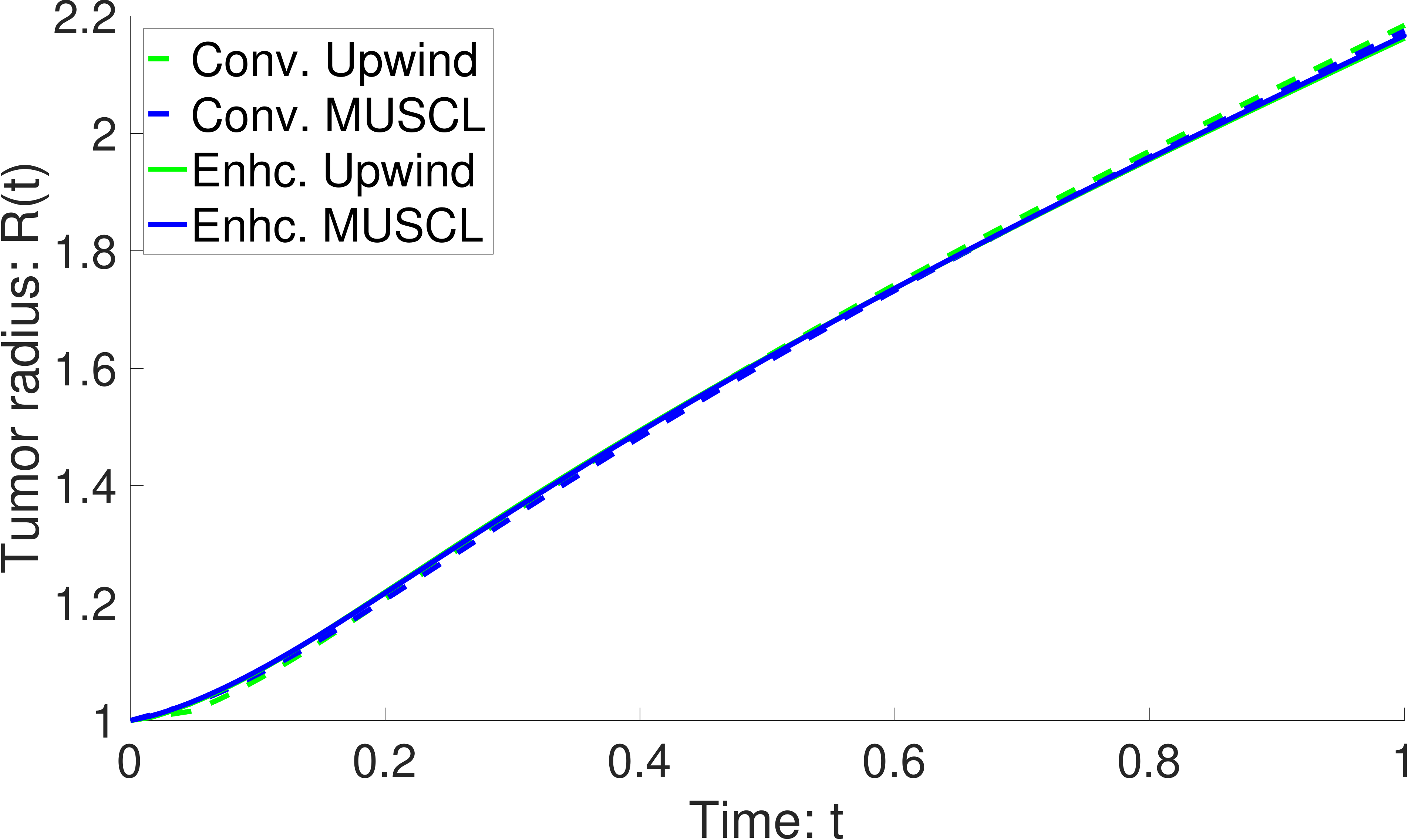}
    \caption{Radius growth history.}
    \label{fg:num_tum_case_sol_rad}
  \end{subfigure}
  \begin{subfigure}[b]{.48\textwidth}\centering
    \includegraphics[width=.8\textwidth]{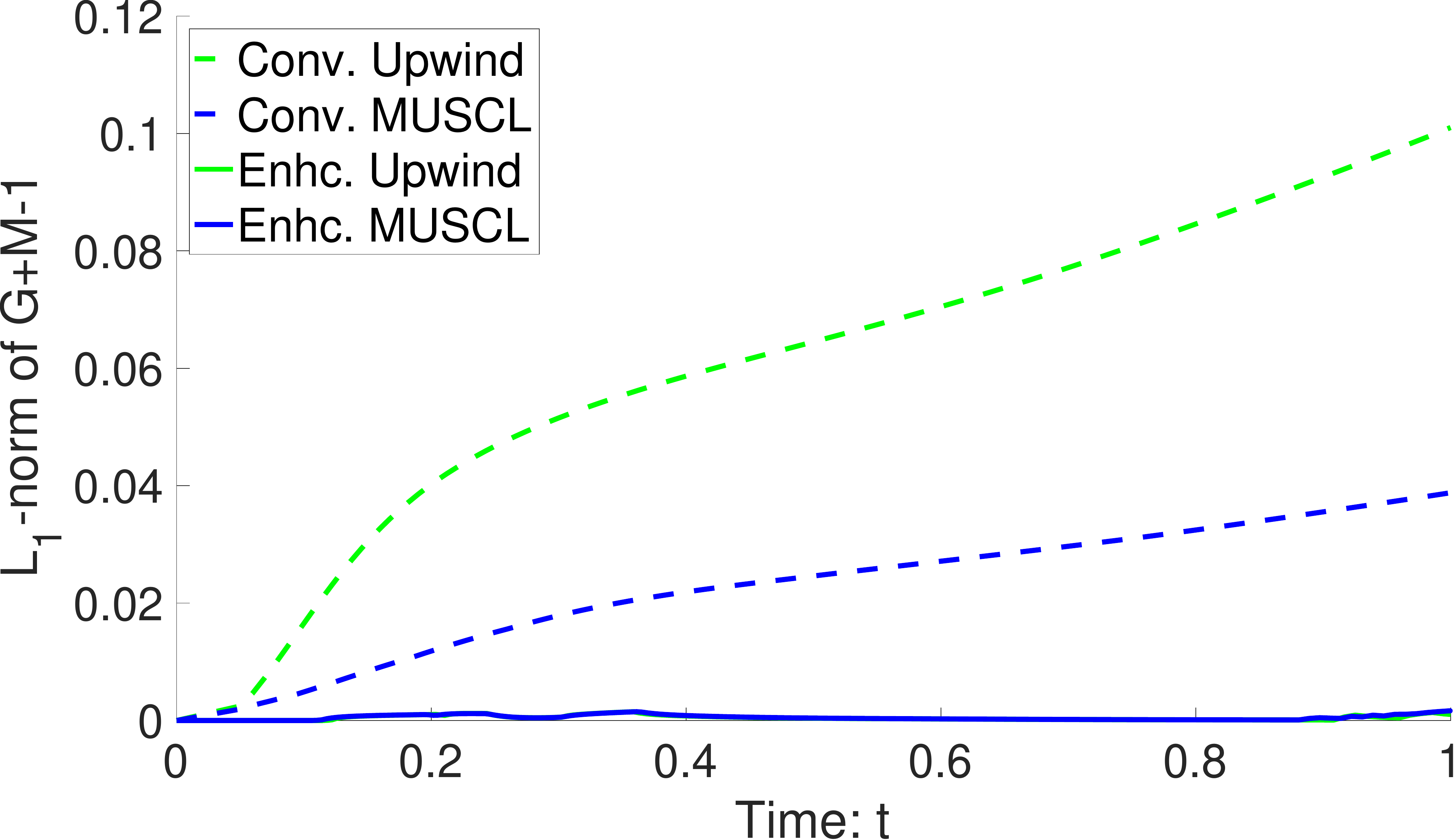}
    \caption{Histories of $d_\theta$.}
    \label{fg:num_tum_case_sol_inc}
  \end{subfigure}
  \vspace*{-.1in}
  \caption{Solutions to the tumor case study on a $50$-interval grid.}
  \label{fg:num_tum_case_sol}
\end{figure}
From Figure~\ref{fg:num_tum_case_sol_inc} we clearly see that the enhanced methods lead to much smaller incompressibility violation than the conventional methods.
The numerical solutions show similar pattern as the simpler model in Section~\ref{sec:num_mod_3}; and we observe alike oscillations in the solutions by the conventional methods.
Nevertheless, the radius growth histories seem to compare well among different methods; and this observation is made more precise by Table~\ref{tb:num_tum_case_rad}, which summarizes the terminal radii at $T=1.0$ by all four methods on a sequence of four grids.
\begin{table}\centering
\caption{Numerical final radius of the tumor case study at $T=1.0$.}
\label{tb:num_tum_case_rad}
\vspace*{-.1in}
\begin{tabular}{|c|c|c|c|c|}
\hline
$N_{\eta}$ & Conv. Upwind 
           & Conv. MUSCL
           & Enhc. Upwind 
           & Enhc. MUSCL \\ \hline
$50$  & 2.1843 & 2.1748 & 2.1625 & 2.1670 \\  
$100$ & 2.1809 & 2.1737 & 2.1666 & 2.1689 \\ 
$200$ & 2.1763 & 2.1718 & 2.1682 & 2.1693 \\ 
$400$ & 2.1730 & 2.1705 & 2.1688 & 2.1693 \\ \hline
\end{tabular}
\end{table}

\subsubsection{The tumor problem of~\cite{BNiu:2018a}}
\label{sec:num_tum_prev}
Finally, we consider the tumor model describing the PDGF-driven glioma cells in the previous work~\cite{BNiu:2018a}, where the set of parameters are chosen so that the survival length matches experimental data.
Here the survival length is defined as the time $T_{\term}$ when the radius reaches $5.0\;\texttt{mm}$.
The parameters and their dimensions are summarized in Table~\ref{tb:num_tum_prev_par}.
\begin{table}\centering
\caption{Parameters of the PDGF-driven glioma problem. For the biological interpretation of each parameter, the readers are referred to~\cite{BNiu:2018a}.}
  %-- the two different values of $m$, the maximum chemoattractant production rate in the Michaelis-Menten model, correspond to two different glioma cells ($G$), namely mulDH1 and wtlDH1.}
\label{tb:num_tum_prev_par}
\vspace*{-.1in}
\begin{tabular}{|c|c|c|}
\hline
Parameter & Value           & Dimension                              \\ \hline
$\lambda$ & 0.48            & $\textrm{day}^{-1}$                    \\
$\mu$     & 0.33            & $\textrm{day}^{-1}$                    \\
$\delta$  & 0.45            & $\textrm{day}^{-1}$                    \\
$\rho$    & 0.9             & $\textrm{day}^{-1}$                    \\
$\nu$     & 6.048           & $\textrm{mm}^2\cdot\textrm{day}^{-1}$  \\
$m$       & 1.5e+5          & $\textrm{pg}\cdot\textrm{ml}^{-1}\cdot\textrm{day}^{-1}$ \\
%\multirow{2}{*}{$m$}        & 1.5e+5 (mulDH1) 
%          & \multirow{2}{*}{$\textrm{pg}\cdot\textrm{ml}^{-1}\cdot\textrm{day}^{-1}$} \\
%          & 5.5e+5 (wtlDH1) & \\
$\beta$   & 1.0e+5          & $\textrm{cell}\cdot\textrm{mm}^{-3}$   \\
$\gamma$  & 1.0e+2          & $\textrm{day}^{-1}$                    \\
$\alpha$  & 0.6             
          & $\textrm{mm}^2\cdot\textrm{ml}\cdot\textrm{day}^{-1}\cdot\textrm{pg}^{-1}$ \\
$\theta$  & 1.0e+6          & $\textrm{cell}\cdot\textrm{mm}^{-3}$   \\ \hline
\end{tabular}
\end{table}
%In this table, the two different values for the parameter $m$ represent two different glioma cells $G$, namely mulDH1 and wtlDH1.
%Since $m$ is the maximum production rate in the Michaelis-Menten term of the chemoattractant equation~(\ref{eq:rev_nrm_eqn_a}), larger $m$ typically leads to higher $A$ inside the tumor hence steeper gradient at the tumor boundary.
%Thus one expects faster infiltration with larger $m$, hence faster tumor growth and shorter survival length.

The initial tumor size is given by $R(0)=0.2\;\textrm{mm}$ and other initial conditions are:
\begin{align*}
  &G(r,0) = 0.84\;\theta,\ 
  H(r,0) = 0.155\;\theta,\ 
  M(r,0) = 0.005\;\theta,\quad &r\in[0,\;R(0))\;; \\
  &A(r,0) = 1000\;\exp(-r^2)\;,&r\in[0,\;+\infty)\;.
\end{align*}
The boundary condition for $M$ describes the environmental number density for the immune cells: $M_{\bc}(t)=0.005\;\theta$.
Numerical solutions to this problem on a uniform grid of $50$ intervals are plotted in the left panel of Figure~\ref{fg:num_tum_prev_sol}, where the solutions of the radius, the glioma cells ($G$), and the total number of cells ($G+H+M$) at $T_{\term}$ are plotted from top to bottom.
For comparison, the solutions computed on a uniform grid with $200$ intervals are provided side-by-side in the right panel of the same figure.
\begin{figure}\centering
  \begin{subfigure}[b]{.96\textwidth}\centering
    \includegraphics[width=.4\textwidth]{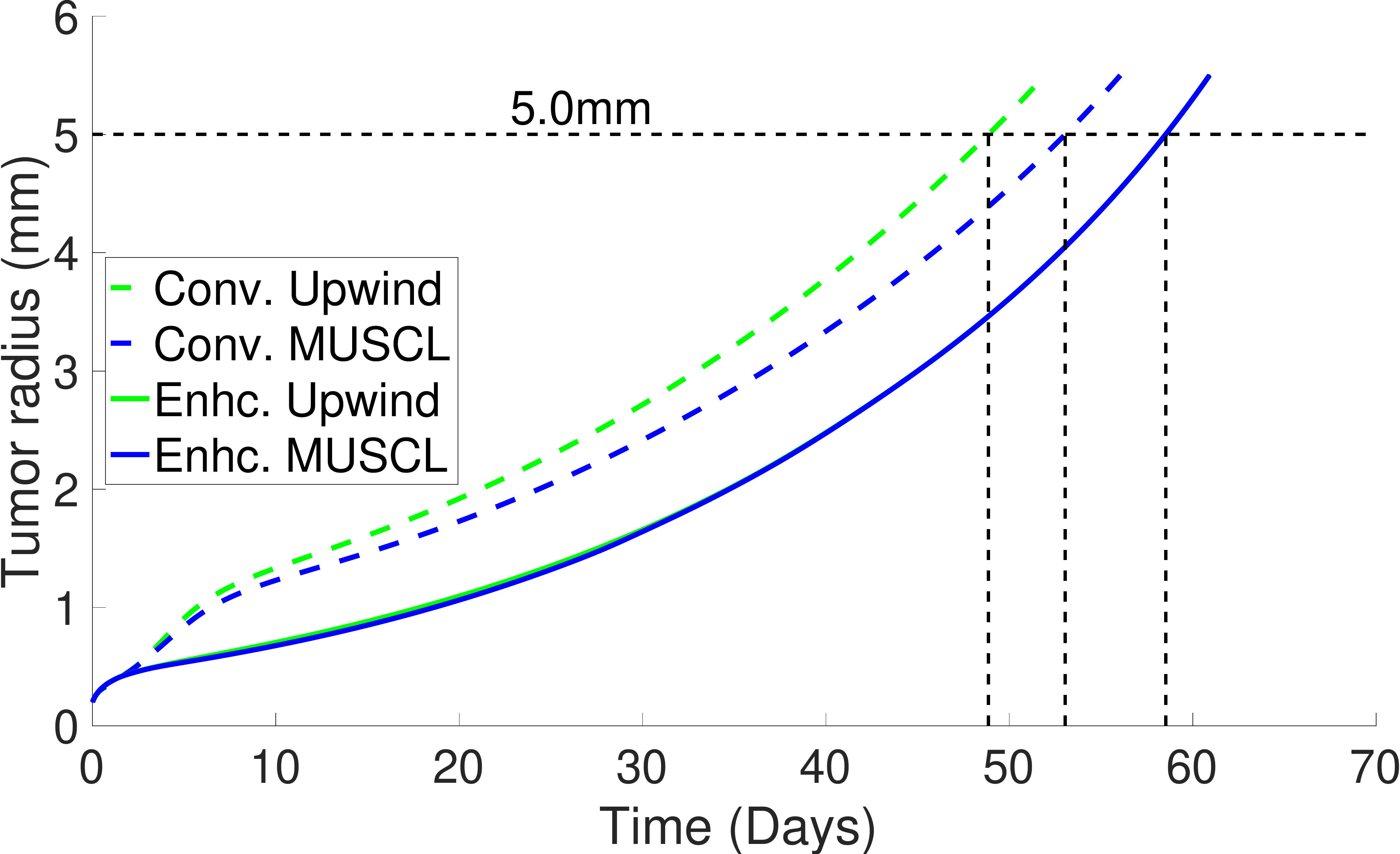}$\qquad$
    \includegraphics[width=.4\textwidth]{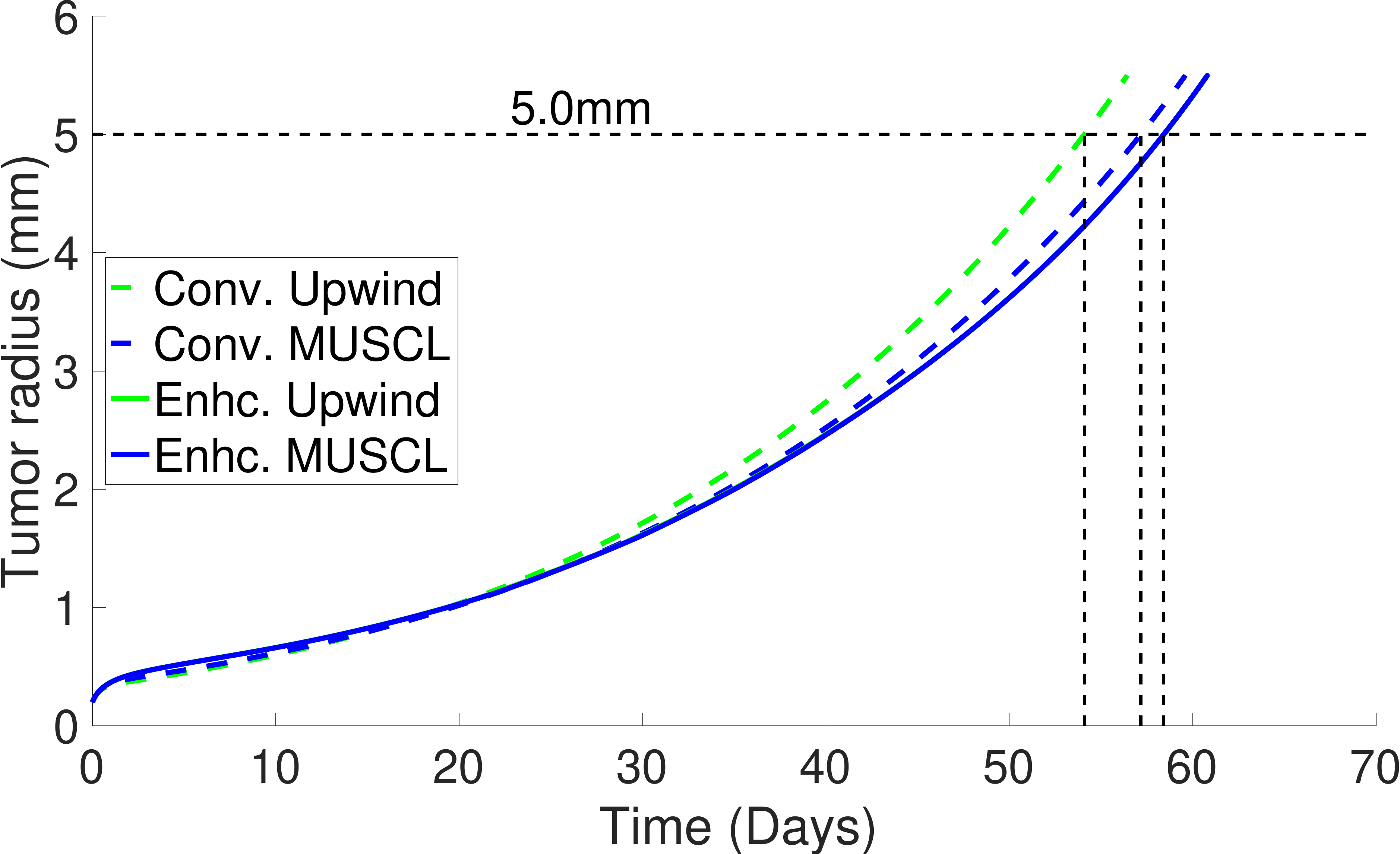}
    \caption{Radius growth histories.}
    \label{fg:num_tum_prev_sol_rad}
  \end{subfigure} \\
  \begin{subfigure}[b]{.96\textwidth}\centering
    \includegraphics[width=.4\textwidth]{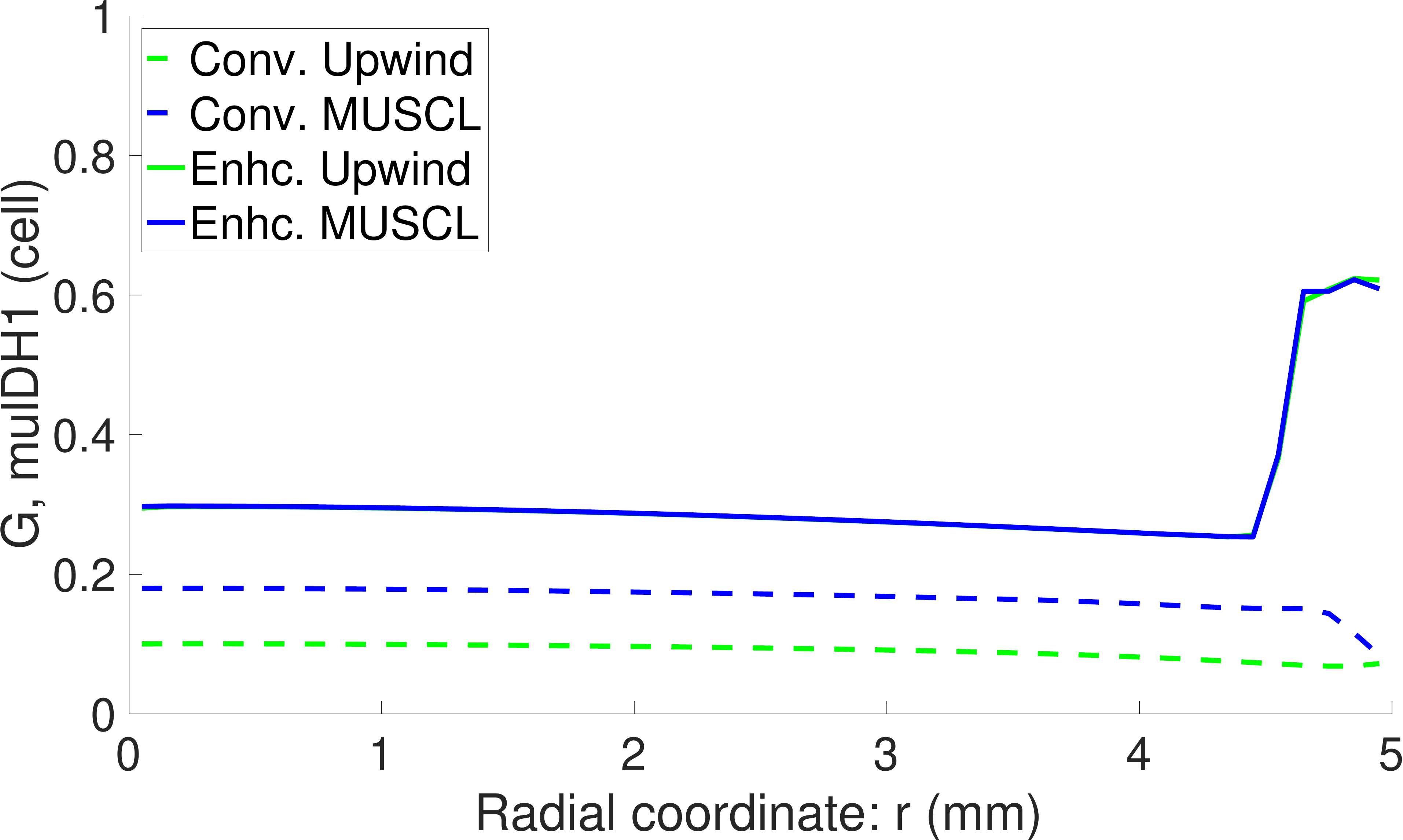}$\qquad$
    \includegraphics[width=.4\textwidth]{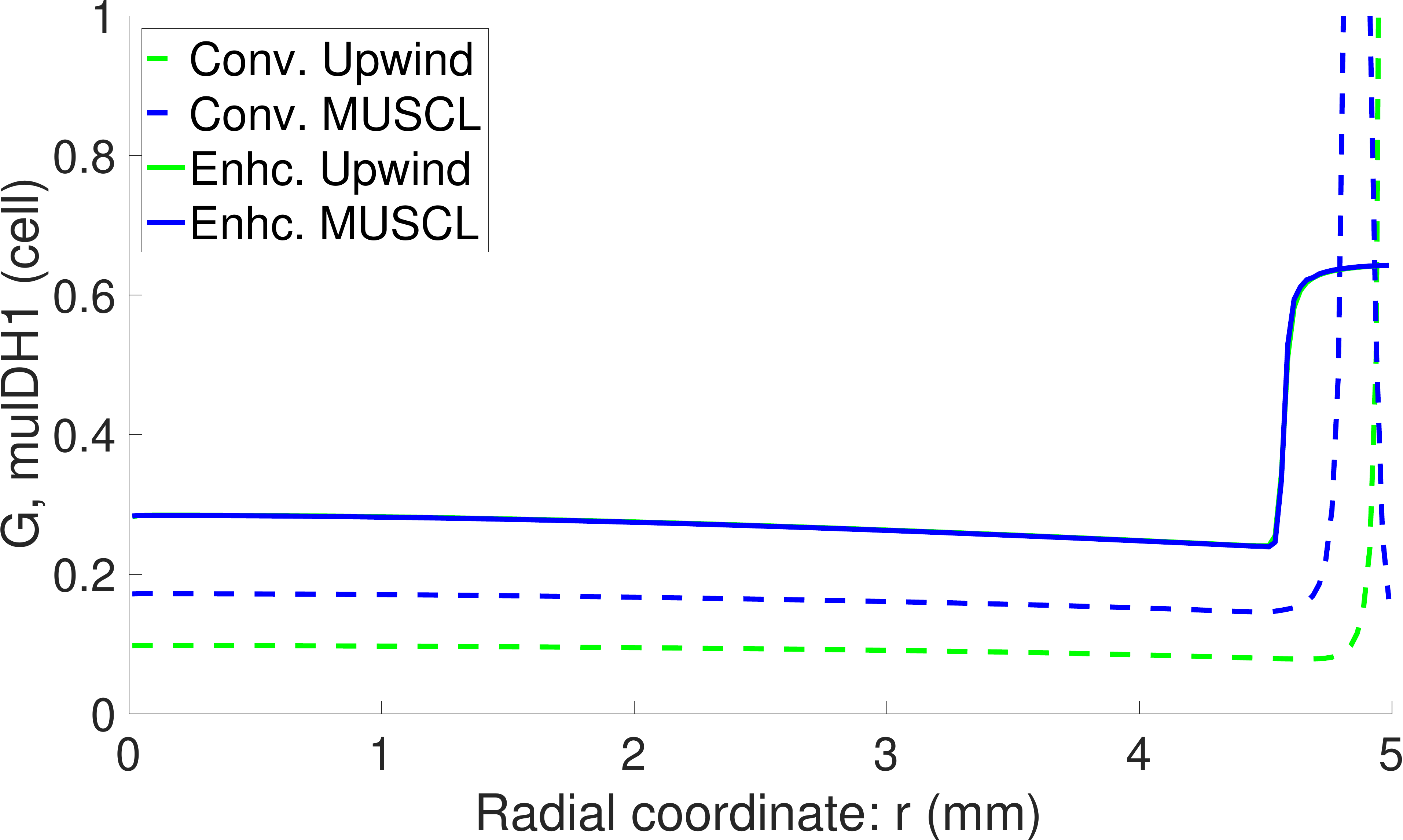}
    \caption{Glioma cell distributions at $T_{\term}$.}
    \label{fg:num_tum_prev_sol_g}
  \end{subfigure} \\
  \begin{subfigure}[b]{.96\textwidth}\centering
    \includegraphics[width=.4\textwidth]{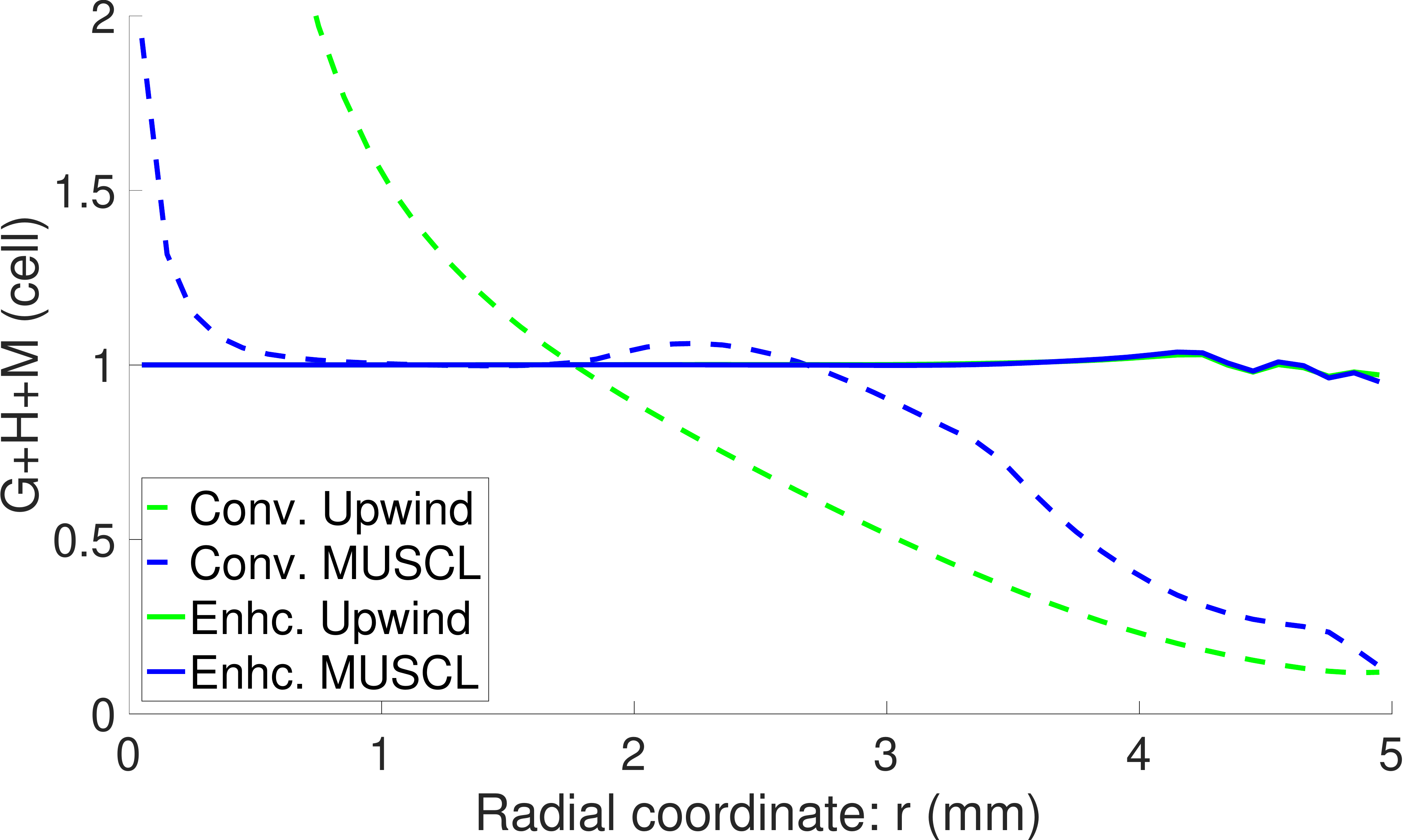}$\qquad$
    \includegraphics[width=.4\textwidth]{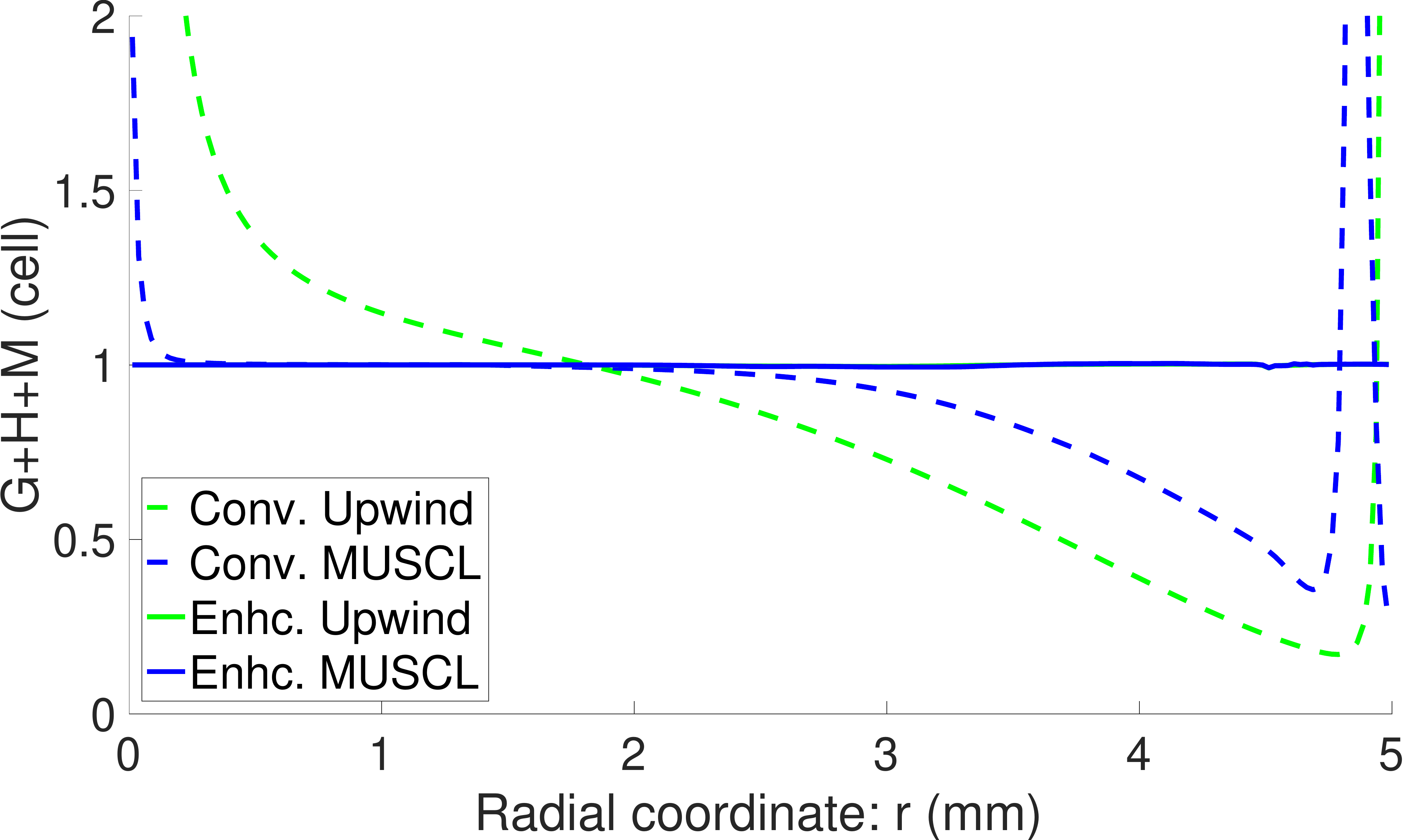}
    \caption{Total cell distributions at $T_{\term}$.}
    \label{fg:num_tum_prev_sol_all}
  \end{subfigure}
  \vspace*{-.1in}
  \caption{Numerical solutions on two grids: (left) $50$ uniform intervals, (right) $200$ uniform intervals.}
  \label{fg:num_tum_prev_sol}
\end{figure}
Although the exact solutions to this problem is unknown, we have the following observations:
\begin{itemize}
  \item The conventional methods seem to underestimate the growth rate of the tumor; and the modified methods provide much faster convergent results, c.f. Figure~\ref{fg:num_tum_prev_sol_rad}.
  \item The solutions to cell species are very different between the conventional methods and the enhanced ones -- particularly near the tumor boundary the conventional FVMs produce overshoots that grows significantly on finer grids, whereas the enhanced ones predict a flat plateau, that could possibly represent the ``rim'' that is reported in many existing studies~\cite{JJCasciari:1992a}. See Figure~\ref{fg:num_tum_prev_sol_g}.
  \item The incompressibility condition is severely violated by both conventional methods; whereas the enhanced ones respect this constraint very nicely, see Figure~\ref{fg:num_tum_prev_sol_all}.
\end{itemize}
Quantitative comparisons are provided in Table~\ref{tb:num_tum_prev_term}, which summarizes the survival length $T_{\term}$ computed by all four methods on a sequence of four grids.
\begin{table}\centering
\caption{Survival length ($\textrm{day}$) of the PDGF-driven model.}
\label{tb:num_tum_prev_term}
\vspace*{-.1in}
\begin{tabular}{|c|c|c|c|c|}
\hline
$N_{\eta}$ & Conv. Upwind 
           & Conv. MUSCL
           & Enhc. Upwind 
           & Enhc. MUSCL \\ \hline
$50$  & 48.8644 & 53.0510 & 58.5284 & 58.5276 \\  
$100$ & 50.4446 & 54.5993 & 58.4433 & 58.4408 \\ 
$200$ & 54.0926 & 57.1748 & 58.4223 & 58.4212 \\ 
$400$ & 57.6348 & 58.4293 & 58.4177 & 58.4174 \\ \hline
\end{tabular}
\end{table}
This table reveals that although the conventional methods predict $T_{\term}$ that is somewhat different from that by the enhanced methods, the predictions converge nevertheless to the same value as the grid is refined.
Hence we claim that the enhanced methods indeed improve the accuracy of the numerical simulation.

Finally, it is worth noting that although the conventional methods compute very different solutions in the cell numbers, they seem to give reasonable predictions on the tumor growth curves, which explains why the numerical results compare reasonably well to experimental data in the previous work~\cite{BNiu:2018a}.
To close this section, we provide an explanation to the phenomenon by utilizing a relation that is analogous to~(\ref{eq:model_ode}).
Particularly, when the tumor grows monotonically as in the present case, a simpler formula determining the growth pattern is:
\begin{displaymath}
  R'(t) = -\alpha\pp{A(R(t),t)}{r}M_{\bc}(t)\;.
\end{displaymath}
Since $M_{\bc}$ is a constant, the growth is determined by the chemoattractant concentration gradient at the tumor boundary.
Because $A$ is governed by the diffusion-reaction equation~(\ref{eq:rev_eqn_a}), its profile is less affected by different cell number solutions in $G$.
Particularly, we plot numerical solutions to $\pp{A(R(t),t)}{r}$ by various methods in Figure~\ref{fg:num_tum_prev_da}, and see that the difference between the conventional methods and enhanced methods is less significant than the difference in the cell number solutions.
\begin{figure}\centering
  \begin{subfigure}[b]{.48\textwidth}\centering
    \includegraphics[width=.8\textwidth]{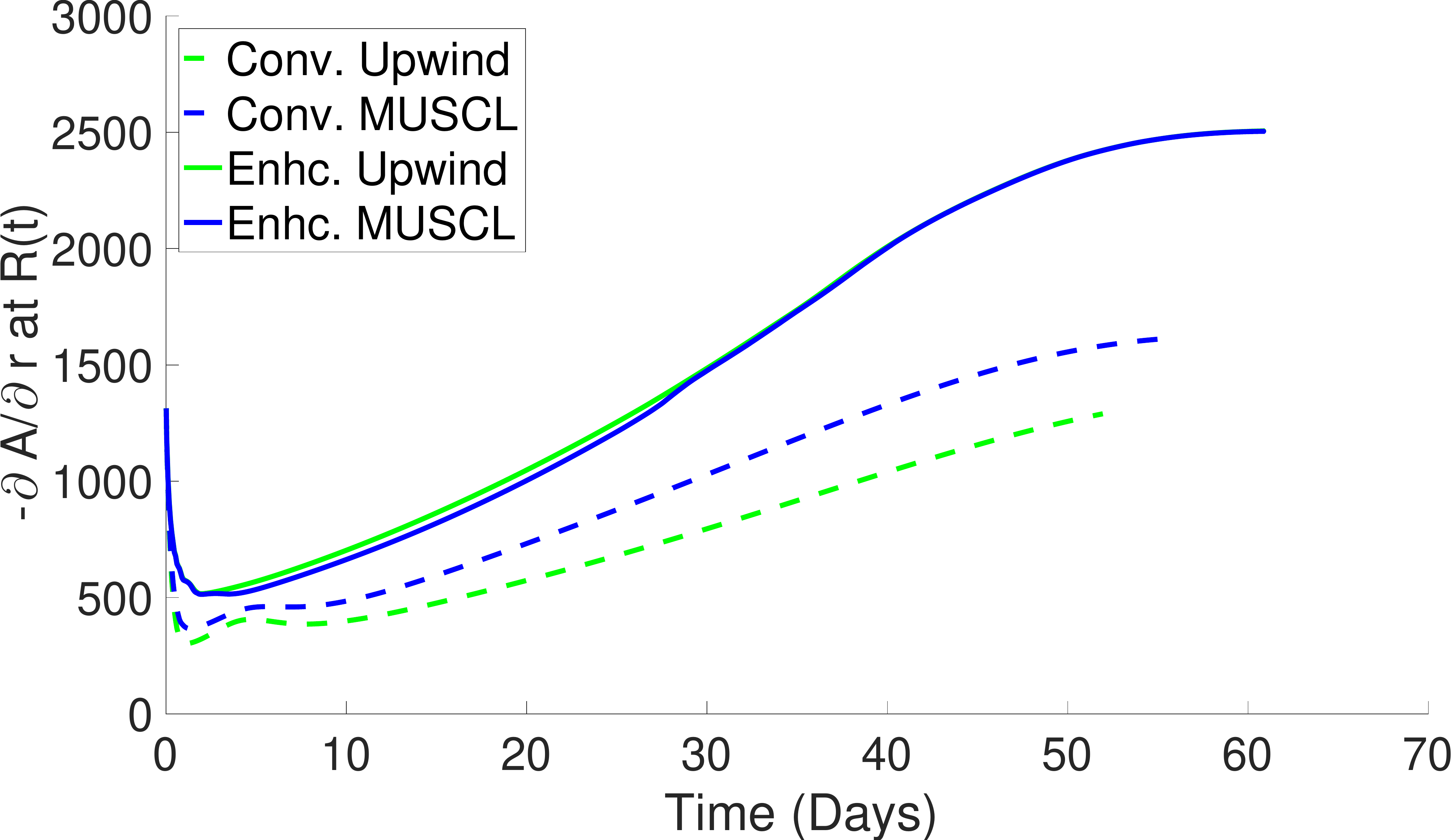}
    \caption{$50$-interval grid.}
    \label{fg:num_tum_prev_da_050}
  \end{subfigure}
  \begin{subfigure}[b]{.48\textwidth}\centering
    \includegraphics[width=.8\textwidth]{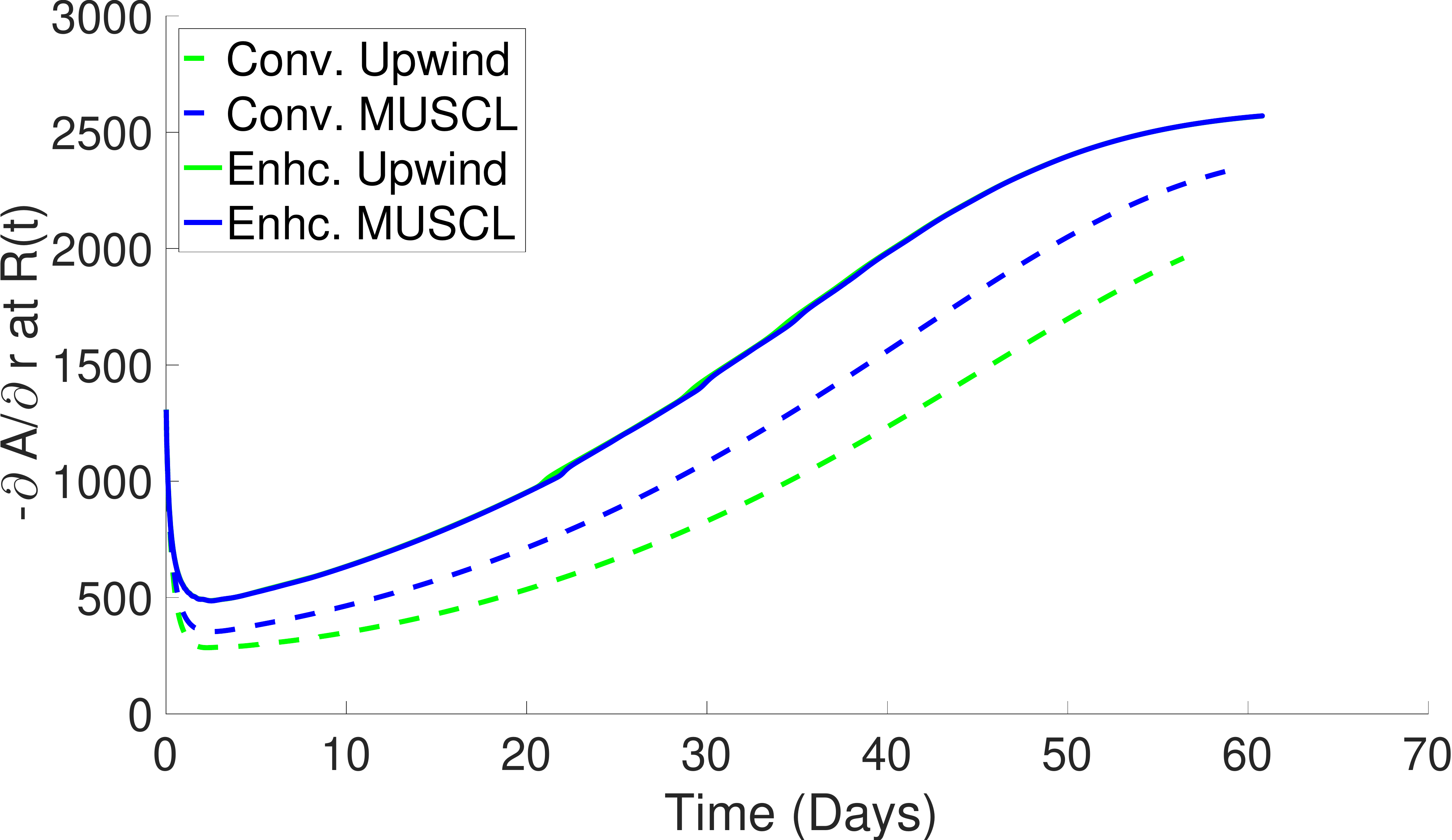}
    \caption{$200$-interval grid.}
    \label{fg:num_tum_prev_da_200}
  \end{subfigure}
  \vspace*{-.1in}
  \caption{Histories of $\pp{A(R(t),t)}{r}$ on two uniform grids.}
  \label{fg:num_tum_prev_da}
\end{figure}

\section{Conclusions}
\label{sec:concl}
We propose a finite volume framework with segregated fluxes for numerical computation of free boundary problems that model infiltration dynamics in spherically symmetric tumor growth.
Under this framework, sufficient conditions for ensuring the geometric conservation law on a moving grid and the incompressibility constraint are derived; and classical first-order and second-order finite volume methods are enhanced following these requirements.
The numerical performance of the enhanced methods are assessed by several representative tests, either for a simplified model or a full PDGF-driven tumor growth model; and their solutions exhibit significant improvements over those by conventional methods.
More importantly, the cell-incompressibility condition is well respected by the enhanced methods but not by the conventional one; and it is shown to be crucial to deliver convergent and stable solutions on refined grids.

It worth noting that, although the MUSCL-type methods generally produce more accurate solutions than the upwind ones, they do not deliver second-order convergence even when the solutions are smooth.
This is probably due to the integro-differential nature of the governing equation; and how to improve the second-order methods will be addressed in future work.
Nevertheless, the methodology to ensure the DGCL and DTCL properties for MUSCL-based methods is expected to remain the same; hence it is addressed in the current paper instead of in future publications.

\section*{Acknowledgements}
X. Zeng would like to thank University of Texas at El Paso for the start up support.
P. Tian would like to thank the National Science Foundation of US for the support in mathematical modeling under the grant number DMS-1446139.

% BibTeX users please use one of
\bibliographystyle{plain}      % mathematics and physical sciences
%\bibliography{pap}

% Non-BibTeX users please use
%   \begin{thebibliography}{}
%   %
%   % and use \bibitem to create references. Consult the Instructions
%   % for authors for reference list style.
%   %
%   \bibitem{ref.J}
%   % Format for Journal Reference
%   Author, Article title, Journal, Volume, page numbers (year)
%   % Format for books
%   \bibitem{ref.B}
%   Author, Book title, page numbers. Publisher, place (year)
%   \end{thebibliography}
\appendix

\section{Splitted velocities for advection equations}
\label{app:split}
In Section~\ref{sec:fvm}, we split the advection velocity for species and compute the fluxes separately. 
In this section, we briefly investigate the stability associated with the splitting strategy by the classical von Neumann analysis.
Since $i$ is used to denote the imaginary unit, we'll use $j$ to denote the grid index. 
Let us consider the following one-dimensional advection equation:
\begin{equation}\label{eq:split_adv}
  \pp{G}{\tau} + V\pp{G}{\eta} = 0\;,
\end{equation}
where $G$ is the advected quantity and $V$ is the constant advection velocity.
Following the splitting strategy, we rewrite $V$ as the sum of $K$ constants:
\begin{equation}\label{eq:split_vel}
  V = V_1 + V_2 + \cdots + V_K
\end{equation}
and solve the corresponding equation by the first-order upwind fluxes and first-order forward Euler time-integrator:
\begin{equation}\label{eq:split_upw_fe}
  \frac{G^{n+1}_{j-1/2}-G^n_{j-1/2}}{\Delta\tau} + \sum_{k=1}^K\frac{\mathcal{F}^{\upw}(V_k;\;G_{j-1/2}^n,\;G_{j+1/2}^n)-\mathcal{F}^{\upw}(V_k;\;G_{j-3/2}^n,\;G_{j-1/2}^n)}{\Delta\eta} = 0\;,
\end{equation}
where $\Delta\tau$ is the time step size and $\mathcal{F}^{\upw}$ is given by (\ref{eq:rev_fvm_upw}).

Clearly, the fluxes collapse into two groups, namely those associated with positive velocities and those associated with negative ones.
Denoting $V_+ = \sum_{1\le k\le K,\;V_k>0}V_k$ and $V_- = \sum_{1\le k\le K,\;V_k<0}V_k$, (\ref{eq:split_upw_fe}) simplifies to:
\begin{equation}\label{eq:split_num}
  \frac{G^{n+1}_{j-1/2}-G^n_{j-1/2}}{\Delta\tau} + \frac{V_+(G^n_{j-1/2}-G^n_{j-3/2})}{\Delta\eta} + \frac{V_-(G^n_{j+1/2}-G^n_{j-1/2})}{\Delta\eta} = 0\;.
\end{equation}
Following the standard von Neumann analysis, we write:
\begin{equation}\label{eq:split_vn}
  G_{j-1/2}^n = a^ne^{ij\kappa\Delta\eta}\;,
\end{equation}
where $a$ is the so called amplifier coefficient and $\kappa$ is the arbitrary wave number; the numerical method is stable if and only if there is a $\Delta\tau_c>0$, such that for all $0\le\Delta\tau\le\Delta\tau_c$ we have $\abs{a}\le1$ for all $\kappa\in\mathbb{R}$.

Denoting $\theta=\kappa\Delta\eta$ for simplicity, plugging (\ref{eq:split_vn}) into (\ref{eq:split_num}) we have:
\begin{displaymath}
  \frac{a^{n+1}e^{ij\theta}-a^ne^{ij\theta}}{\Delta\tau} + \frac{V_+}{\Delta\eta}(a^ne^{ij\theta}-a^ne^{i(j-1)\theta}) + \frac{V_-}{\Delta\eta}(a^ne^{i(j+1)\theta}-a^ne^{ij\theta}) = 0\;;
\end{displaymath}
and it follows that:
\begin{displaymath}
  a = 1 - \frac{V_+\Delta\tau}{\Delta\eta}(1-e^{-i\theta})-\frac{V_-\Delta\tau}{\Delta\eta}(e^{i\theta}-1)\;.
\end{displaymath}
Let the Courant numbers corresponding to $V_+$ and $V_-$ be $\alpha_+=V_+\Delta\tau/\Delta\eta\ge0$ and $\alpha_-=-V_-\Delta\tau/\Delta\eta\ge0$, respectively, it is easy to compute that:
\begin{align*}
  a &= 1 - (\alpha_++\alpha_-)(1-\cos\theta) + i(\alpha_+-\alpha_-)\sin\theta \quad\Rightarrow\\
  \abs{a}^2 &= (1-(\alpha_++\alpha_-)(1-\cos\theta))^2 + ((\alpha_+-\alpha_-)\sin\theta)^2 \\ 
            &= (1-(\alpha_++\alpha_-)(1-\cos\theta))^2 + ((\alpha_++\alpha_-)\sin\theta)^2 - 4\alpha_+\alpha_-\sin^2\theta \\
            &= 1 - 2(\alpha_++\alpha_-)(1-\alpha_+-\alpha_-)(1-\cos\theta)-4\alpha_+\alpha_-\sin^2\theta \\
            &= 1 - \left[2(\alpha_++\alpha_-)(1-\alpha_+-\alpha_-)+4\alpha_+\alpha_-(1+\cos\theta)\right](1-\cos\theta)\;;
\end{align*}
it follows that $\abs{a}\le1$ for all $\theta\in\mathbb{R}$ if and only if for these $\theta$:
\begin{displaymath}
2(\alpha_++\alpha_-)(1-\alpha_+-\alpha_-)+4\alpha_+\alpha_-(1+\cos\theta)\ge0\;,
\end{displaymath}
or equivalently, $\alpha_++\alpha_-\le1$.
Hence the explicit split method is conditionally stable and the Courant condition is:
\begin{equation}\label{eq:split_cfl}
(\abs{V_+}+\abs{V_-})\Delta\tau\le\Delta\eta\quad\textrm{ or }\quad
\left(\sum_{k=1}^K\abs{V_k}\right)\Delta\tau\le\Delta\eta\;.
\end{equation}

\medskip 

Similarly, using the implicit first-order backward Euler time-integrator instead, one computes:
\begin{displaymath}
 a = \frac{1}{1+\alpha_+(1-e^{-i\theta})-\alpha_-(e^{i\theta}-1)}
   = \frac{1}{1+(\alpha_++\alpha_-)(1-\cos\theta) + i(\alpha_+-\alpha_-)\sin\theta}\;,
\end{displaymath}
and $\abs{a}\le1$ is equivalent to:
\begin{displaymath}
(1+(\alpha_++\alpha_-)(1-\cos\theta))^2+(\alpha_+-\alpha_-)^2\sin^2\theta\ge1\;,
\end{displaymath}
which holds naturally for all $\theta$.
To conclude, the implicit split method is unconditionally stable.

\section{Implicit Enhanced Finite Volume Methods}
\label{app:imp}
In this appendix we extend the enhanced method of Section~\ref{sec:fvm} to implicit time-integrators.
First, let us consider the backward Euler time-integrator, which is unconditionally stable combined with the segregated upwind flux, as shown in the previous appendix.

To illustrate the idea, in order to update the solutions from $\tau^n$ to $\tau^{n+1}$, all spatial discretizations happen at $\tau^{n+1}$ instead of $\tau^n$, c.f. the explicit methods.
For example, the counterpart of (\ref{eq:fvm_gen_g}) reads:
\begin{align}\label{eq:imp_gen_g}
  &\ \frac{\eta_{j-1/2}^2[(R^{n+1})^2G_{j-1/2}^{n+1}-(R^n)^2G_{j-1/2}^n]}{\Delta\tau^n} + \frac{1}{\Delta\eta}\left[F_j^{G,n+1}-F_{j-1}^{G,n+1}\right] \\
  \notag
  =&\ \eta_{j-1/2}^2(R^{n+1})^2f_{j-1/2}^{n+1}-\eta_{j-1/2}^2R'^{n+1}R^{n+1}G^{n+1}_{j-1/2}\;, 
\end{align}
where $F_j^{G,n+1}$ approximates:
\begin{equation}\label{eq:imp_gen_flux_g}
  F^{G,n+1}_j \approx \left(\frac{V}{R}-\frac{\eta R'}{R}\right)\eta^2R^2G\Big|_{\eta=\eta_j,\;\tau=\tau^{n+1}}\;,
\end{equation}
and it is segrated into $F_j^{G,n+1} = F_{V,j}^{G,n+1} + F_{R',j}^{G,n+1}$.
%Clearly a $R'^{n+1}$ and a $R^{n+1}$ are required to computed these fluxes.

In Section~\ref{sec:fvm}, (\ref{eq:fvm_t_gcl}) is obtained from the explicit formula~(\ref{eq:fvm_case}); hence it needs to be modified to:
\begin{equation}\label{eq:imp_t_gcl}
  R'^{n+1} = \frac{(R^{n+1})^2-(R^n)^2}{2\Delta\tau^n R^{n+1}}\;.
\end{equation}
In analogous of (\ref{eq:fvm_t_gcl_rp}), we compute $R'^{n+1}$ such that:
\begin{equation}\label{eq:imp_t_gcl_rp}
  R'^{n+1} = \left(1-\frac{1}{4}\Delta\eta^2\right)^{-1}V_{N_{\eta}}^{n+1}\;.
\end{equation}

Now we have similar to Theorem~\ref{thm:fvm_gtcl} the following result:
\begin{theorem}\label{thm:imp_gtcl}
  The numerical method given by the implicit version of (\ref{eq:fvm_gen_g}), (\ref{eq:fvm_gen_m}) and (\ref{eq:fvm_gen_v}), c.f., (\ref{eq:imp_gen_g}), satisfies both DGCL and DTCL if: (1) $\mathcal{F}_j^{n+1}$ is additive and $V$-consistent, (2) $\hat{\mathcal{F}}_j^{n+1}$ is additive and cubic-preserving, (3) $\mathscr{F}_{u,j}^{M,n+1}=F_{u,j}^{M,n+1}$, and (4) $R'^{n+1}$ equals the right hand side of (\ref{eq:imp_t_gcl}).
\end{theorem}
The proof is completely analogous and omitted here.

\smallskip

Extension to higher-order implicit time-integrators are straightforward by using the Diagonally Implicit Runge-Kutta (DIRK) methods, which are well documented in many texts on numericla methods for ordinary differential equations, such as~\cite{JCButcher:2016a}.
In essense, a DIRK method is a multi-stage method with higher time accuracy, where each stage is equivalent to a backward Euler step; hence the method described before extends naturally to these time-integrators.
The particular one that we will use in combine with MUSCL in space is the second-order DIRK method given in Section 361 of~\cite{JCButcher:2016a}.

To demonstrate the numerical performances, we repeat the test in Section~\ref{sec:num_mod_1} using a much larger Courant number $\alpha_{\cfl}=10.0$.
The solutions on a grid of $50$ uniform cells as well as the convergence plots of the terminal radii are plotted in Figure~\ref{fg:imp_mod_sol}.
\begin{figure}\centering
  \begin{subfigure}[b]{.48\textwidth}\centering
    \includegraphics[width=.8\textwidth]{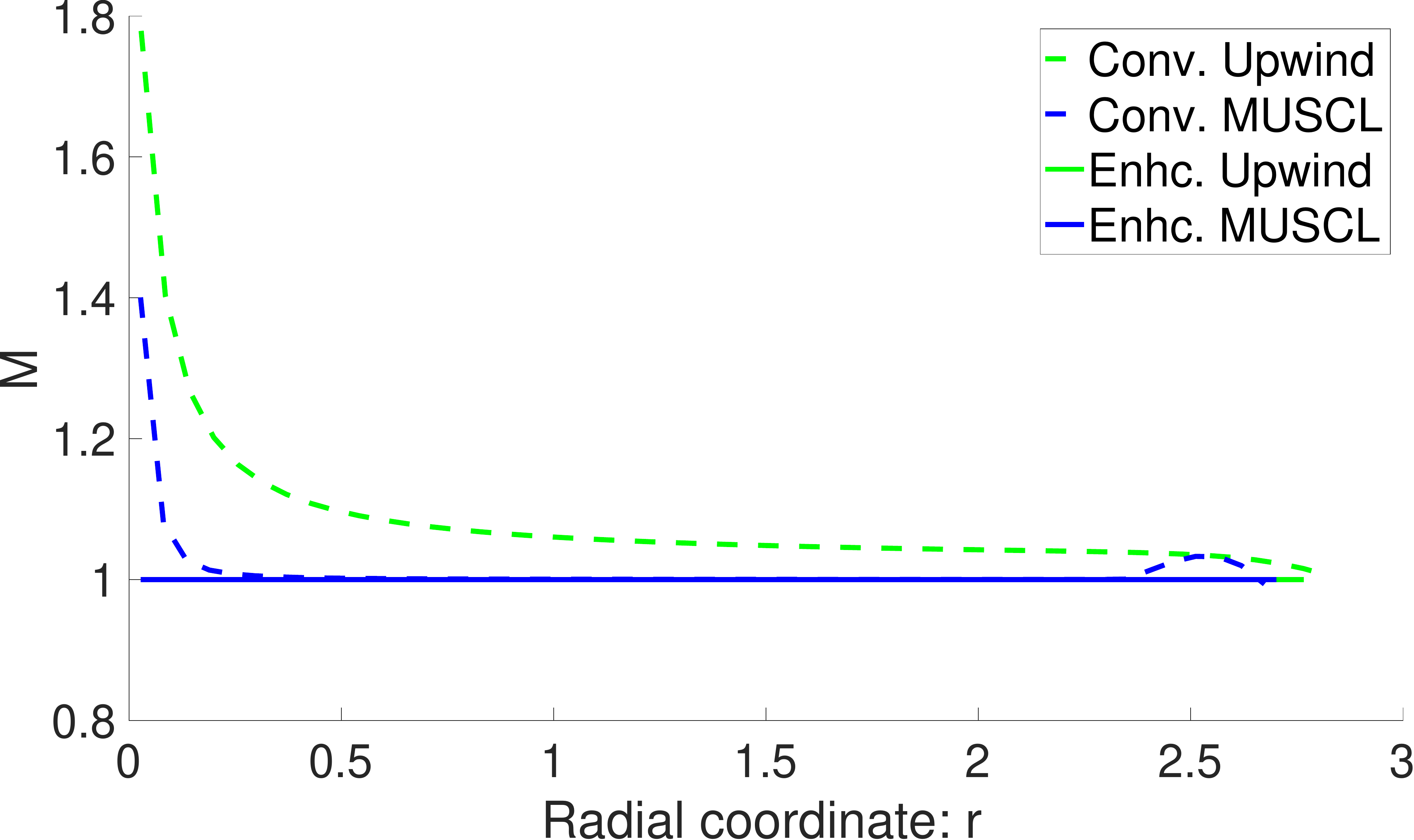}
    \caption{Cell numbers for $M$ at $T=2.0$.}
    \label{fg:imp_mod_sol_m}
  \end{subfigure}
  \begin{subfigure}[b]{.48\textwidth}\centering
    \includegraphics[width=.8\textwidth]{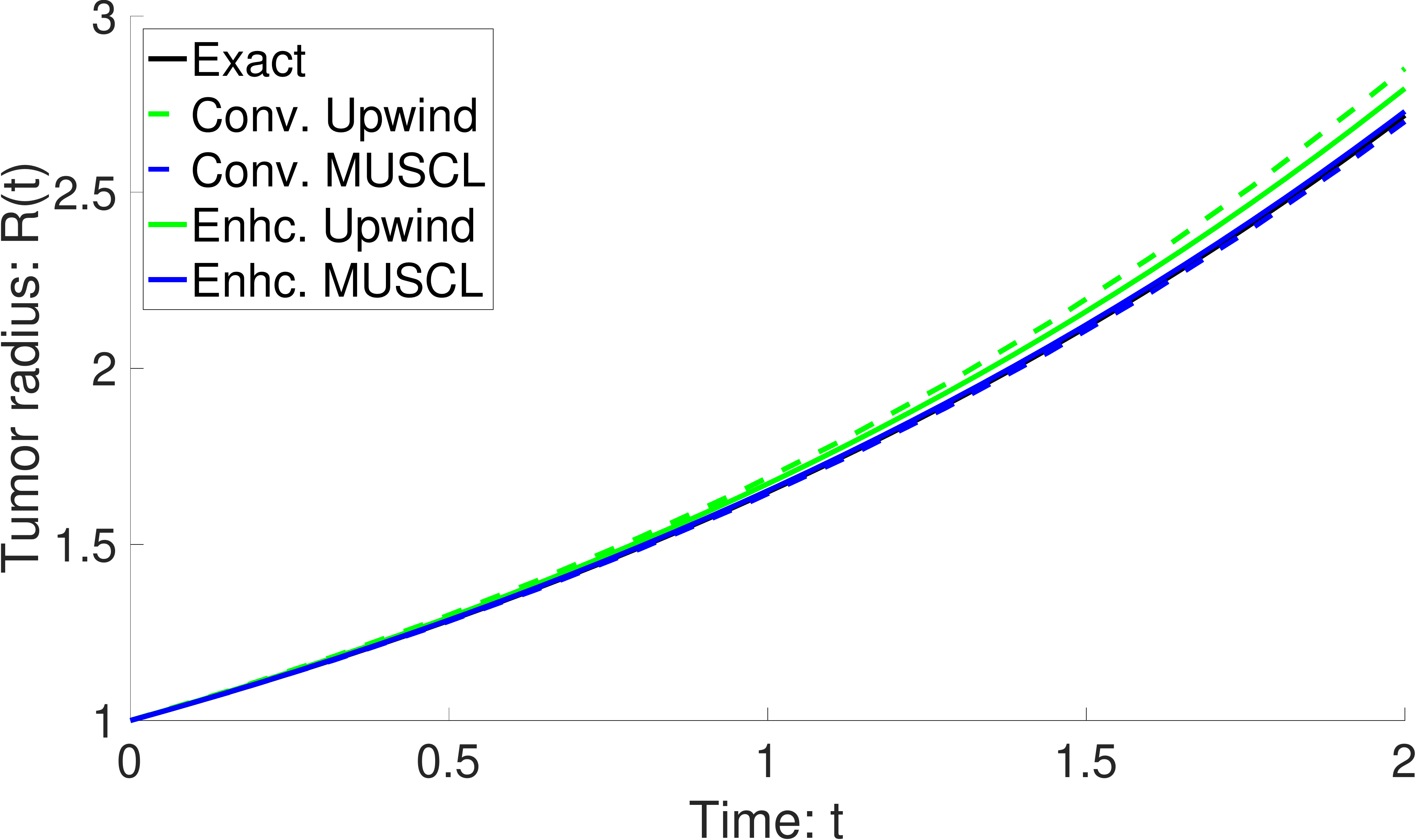}
    \caption{Radius growth history.}
    \label{fg:imp_mod_sol_rad}
  \end{subfigure} \\
  \begin{subfigure}[b]{.48\textwidth}\centering
    \includegraphics[width=.8\textwidth]{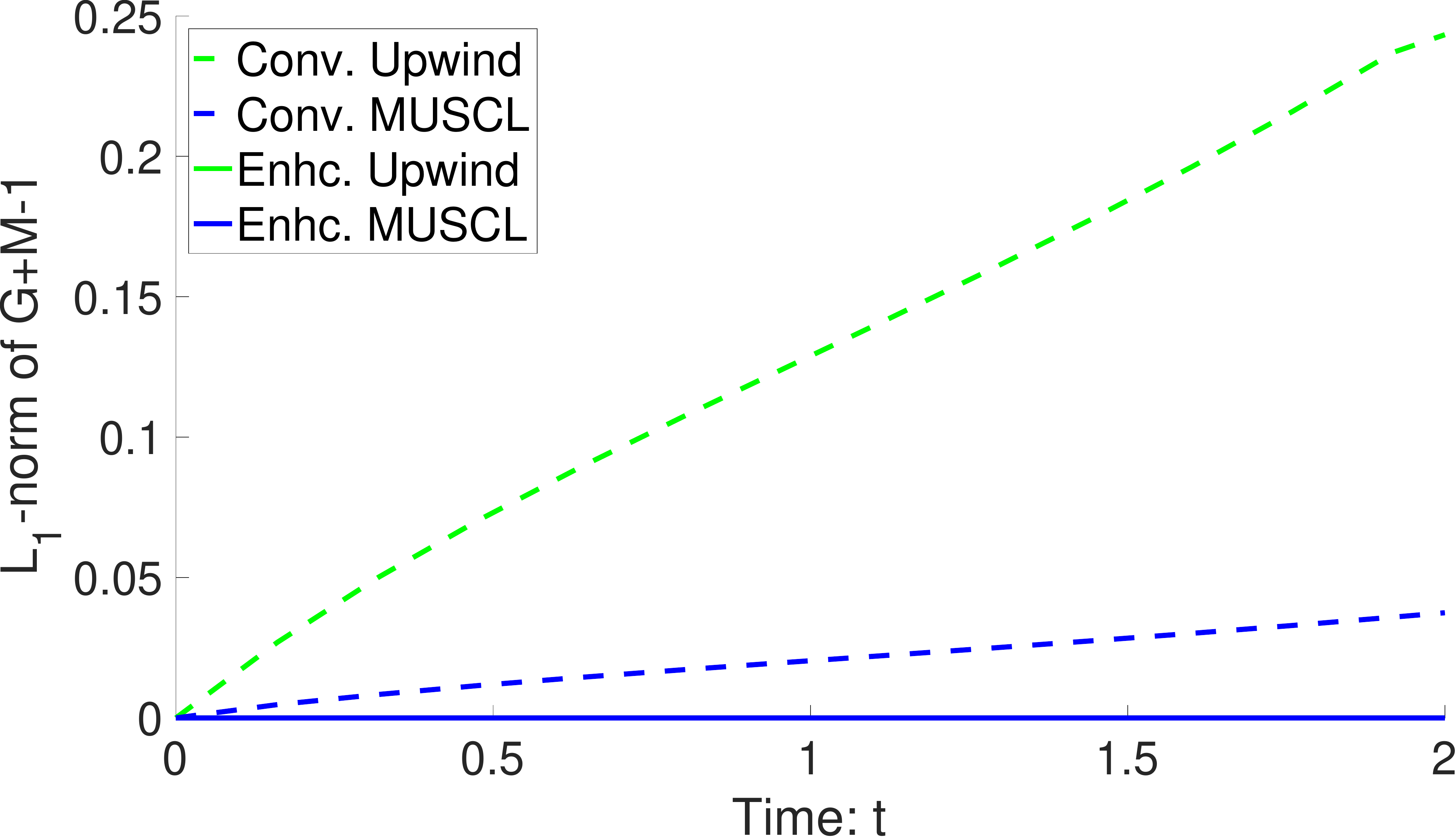}
    \caption{Histories of $d_\theta$.}
    \label{fg:imp_mod_sol_inc}
  \end{subfigure}
  \begin{subfigure}[b]{.48\textwidth}\centering
    \includegraphics[width=.8\textwidth]{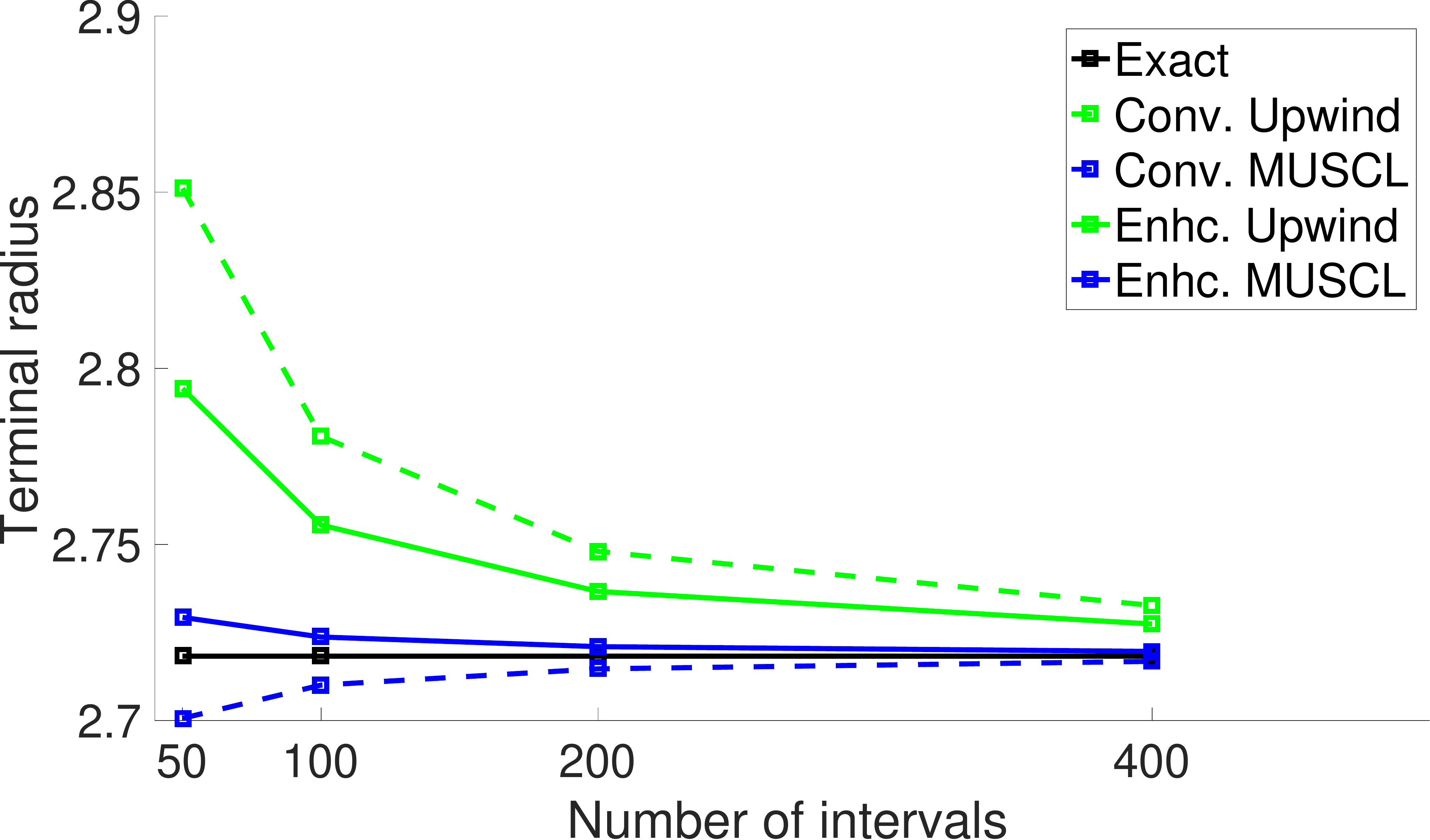}
    \caption{Final radii convergence.}
    \label{fg:imp_mod_sol_conv}
  \end{subfigure}
  \vspace*{-.1in}
  \caption{Implicit solutions to the test in Section~\ref{sec:num_mod_1} on a $50$-cell grid.}
  \label{fg:imp_mod_sol}
\end{figure}
Comparing these figures to Figures~\ref{fg:num_mod_1_inc}--\ref{fg:num_mod_1_conv_rad}, we can draw very similar conclusions, indicating the successful extension of the enhanced methods in Section~\ref{sec:fvm} to implicit time-integrators.

\end{document}